\documentclass{amsart}
\chardef\bslash=`\\

\usepackage{amssymb,amsmath,amsfonts,amsthm,epsfig,amscd,stmaryrd,enumitem}
\usepackage{stmaryrd}
\usepackage[all,cmtip,poly]{xy}
\usepackage{xcolor}
\usepackage{tikz}
\usepackage{hyperref}

\usepackage{silence}
\usepackage{microtype}

\newtheorem{thm}{Theorem}[section]
\newtheorem{theorem}[thm]{Theorem}
\newtheorem{proposition}[thm]{Proposition}
\newtheorem{prop}[thm]{Proposition}
\newtheorem{cor}[thm]{Corollary}
\newtheorem{lem}[thm]{Lemma}
\newtheorem{lemma}[thm]{Lemma}
\newcommand{\dslash}{{\mkern0mu \sslash \mkern0mu}}
\newtheorem{introthm}{Theorem}

\newtheorem{defn}[thm]{Definition}
\theoremstyle{remark}
\newtheorem{remark}[thm]{Remark}
\newtheorem{rem}[thm]{Remark}

\newtheorem{example}[thm]{Example}

\numberwithin{equation}{section}

\newif\iffinalrun
\iffinalrun
\else
 \fi

\iffinalrun
  \newcommand{\need}[1]{}
  \newcommand{\mar}[1]{}
\else
  \newcommand{\need}[1]{{\tiny *** #1}}
  \newcommand{\mar}[1]{\marginpar{\raggedright\tiny  #1}}\fi

\renewcommand\mathbb{\mathbf}

\newcommand{\Lie}{{\operatorname{Lie}\,}}

\newcommand{\frg}{{\mathfrak{g}}}

\newcommand{\rec}{{\operatorname{rec}}}

\newcommand{\rbar}{\overline{r}}

\newcommand{\wv}{{\widetilde{v}}}

\renewcommand{\ell}{l}

\def\PGL{\mathrm{PGL}}

\def\Iw{\mathrm{Iw}}

\newcommand{\ad}{\operatorname{ad}}
\newcommand{\diag}{\operatorname{diag}}
\newcommand{\tr}{\operatorname{tr}}

\newcommand{\A}{\mathbf{A}}
\newcommand{\bA}{\ensuremath{\mathbf{A}}}
\newcommand{\CC}{{\mathbb C}}
\newcommand{\bC}{\ensuremath{\mathbf{C}}}

\newcommand{\bG}{\ensuremath{\mathbf{G}}}

\newcommand{\bQ}{\ensuremath{\mathbf{Q}}}
\newcommand{\Q}{\QQ}
\newcommand{\QQ}{{\mathbb Q}}
\newcommand{\NN}{{\mathbb N}}

\newcommand{\bR}{\ensuremath{\mathbf{R}}}
\newcommand{\R}{\RR}
\newcommand{\RR}{{\mathbb R}}
\newcommand{\bT}{\ensuremath{\mathbf{T}}}
\newcommand{\TT}{{\mathbb T}}

\newcommand{\Z}{\ZZ}
\newcommand{\ZZ}{{\mathbb Z}}
\newcommand{\bZ}{\ensuremath{\mathbf{Z}}}
\newcommand{\ba}{\ensuremath{\mathbf{a}}}
\newcommand{\bbZ}{\ensuremath{\mathbf{Z}}}
\newcommand{\bbQ}{\ensuremath{\mathbf{Q}}}

\newcommand{\cC}{{\mathcal C}}

\newcommand{\cE}{{\mathcal E}}
\newcommand{\cF}{{\mathcal F}}
\newcommand{\cG}{{\mathcal G}}
\newcommand{\cH}{{\mathcal H}}

\newcommand{\cL}{{\mathcal L}}

\newcommand{\cO}{{\mathcal O}}
\newcommand{\cR}{{\mathcal R}}
\newcommand{\cS}{{\mathcal S}}

\newcommand{\cV}{{\mathcal V}}

\newcommand{\m}{\frakm}
\newcommand{\ffrm}{{\mathfrak m}}

\newcommand{\frakm}{\mathfrak{m}}

\newcommand{\frakq}{\mathfrak{q}}
\newcommand{\q}{\frakq}

\newcommand{\Qbar}{\overline{\Q}}

\newcommand{\Zp}{\Z_p}

\newcommand{\Qp}{\Q_p}

\newcommand{\Qpbar}{\Qbar_p}

\DeclareMathOperator{\End}{End}

\DeclareMathOperator{\Gal}{Gal}
\newcommand{\GL}{\mathrm{GL}}

\DeclareMathOperator{\Hom}{Hom}
\DeclareMathOperator{\im}{im}
\DeclareMathOperator{\Ind}{Ind}

\DeclareMathOperator{\ord}{ord}
\DeclareMathOperator{\SL}{SL}
\DeclareMathOperator{\Spec}{Spec}

\DeclareMathOperator{\Sym}{Sym}

\newcommand{\Frob}{\mathrm{Frob}}

\newcommand{\rhobar}{\overline{\rho}}

\newcommand{\Art}{{\operatorname{Art}}}
\newcommand{\Res}{\operatorname{Res}}

\newcommand{\doubleslash}{/\kern-0.2em{/}}

\newcommand{\prm}{\mathrm{p}}
\newcommand{\mrm}{\mathrm{m}}

\begin{document}

\author{James Newton and Jack A. Thorne}
\title[Adjoint Selmer groups of unitary type]{Adjoint Selmer groups of automorphic Galois representations of unitary type}
\begin{abstract} Let $\rho$ be the $p$-adic Galois representation attached to a cuspidal, regular algebraic automorphic representation of $\GL_n$ of unitary type. Under very mild hypotheses on $\rho$, we prove the vanishing of the (Bloch--Kato) adjoint Selmer group of $\rho$. We obtain definitive results for the adjoint Selmer groups associated to non-CM Hilbert modular forms and elliptic curves over totally real fields.
\end{abstract}
\maketitle
\setcounter{tocdepth}{1}
\tableofcontents

\section*{Introduction}

\textbf{Context.} Let $F$ be a CM number field, with maximal totally real 
subfield $F^+$. Fix an algebraic closure $\overline{F}$ of $F$ and a complex 
conjugation $c \in \Gal(\overline{F}/F^+)$. We say that a cuspidal automorphic 
representation $\pi$ of 
$\GL_n(\bA_F)$ is of unitary type if it is conjugate self-dual, i.e.\ if it 
satisfies the relation $\pi^c \cong \pi^\vee$. If $\pi$ is conjugate self-dual 
and 
moreover regular algebraic (a condition on $\pi_\infty$), then for any 
isomorphism $\iota : \overline{\bQ}_p \to \bC$ there is an associated $p$-adic 
Galois representation
\[ r_{\pi, \iota} : \Gal(\overline{F} / F) \to \GL_n(\overline{\bQ}_p), \]
characterized up to isomorphism by compatibility with the local Langlands 
correspondence at each finite place of $F$. The conjugate self-duality of $\pi$ 
implies the existence of an isomorphism $r_{\pi, \iota}^c \cong r_{\pi, 
\iota}^\vee \otimes \epsilon^{1-n}$, where $\epsilon$ is the $p$-adic 
cyclotomic character.

This paper concerns the adjoint Bloch--Kato Selmer group of such a representation. To define it, we note that if $V$ denotes the space on which $r_{\pi, \iota}$ acts, then the conjugate self-duality of $r_{\pi, \iota}$ is reflected in the existence of a perfect, symmetric, and Galois equivariant bilinear pairing
\[ \langle \cdot, \cdot \rangle : V^c \times V \otimes \epsilon^{n-1} \to \overline{\bQ}_p. \]
The existence of this pairing allows us to extend the adjoint action of 
$r_{\pi, \iota}$ on $\End(V)$ to an action of $\Gal(\overline{F} / F^+)$, where 
$c \in \Gal(\overline{F} / F^+)$ acts by the formula $c \cdot X = - X^\ast$ 
(and $X^\ast$ is the adjoint with respect to $\langle \cdot, \cdot \rangle$).

We are interested in is the Bloch--Kato Selmer group
 \begin{multline*} H^1_f(F^+, \End(V)) =\\ \left\{x \in H^1(F^+, \End(V)):x_v \in H^1_f(F^+_v, 
\End(V)) \textup{ for all finite places }v\right\} \end{multline*}   where $H^1_f(F^+_v, 
\End(V))$ is $\ker\left(H^1(F^+_v, \End(V))\to H^1(F^+_v, \End(V) 
\otimes_{\bQ_p} 
B_{crys})\right)$ for $v|p$ and  $H^1_{ur}(F^+_v, \End(V))$ 
for $v \nmid p$.

General conjectures predict the vanishing of this group (see the introduction 
of \cite{All16} for a detailed discussion of this in the present context). We are content here to note 
that this group parameterizes infinitesimal deformations of $r_{\pi, \iota}$ 
which are at the same time conjugate self-dual and geometric, in the sense of 
$p$-adic Hodge theory. 

\textbf{Our results.} The following is the main theorem of this paper.
\begin{introthm}\label{thm_intro_main_theorem}
	Let $F$ be a CM number field, and let $\pi$ be a regular algebraic, cuspidal automorphic representation of $\GL_n(\bA_F)$ of unitary type. Let $p$ be a prime, and let $\iota : \overline{\bQ}_p \to \bC$ be an isomorphism. Suppose that $r_{\pi, \iota}(G_{F(\zeta_{p^\infty})})$ is enormous, in the sense of Definition \ref{dfn:enormous}. Then $H^1_f(F^+, \ad r_{\pi, \iota}) = 0$.
\end{introthm}
For some examples of enormous subgroups, see \S \ref{sec_examples_of_enormous}. 
For example, we note that our condition is satisfied for any $\pi$ such that 
for some finite place $v$ of $F$, $\pi_v$ is a twist of the Steinberg 
representation.

We compare Theorem \ref{thm_intro_main_theorem} with some other results in the literature that are proved using 
related techniques. Kisin \cite{Kis04a} proved the analogue of Theorem 
\ref{thm_intro_main_theorem} for the Galois representations attached to 
classical holomorphic modular forms under some mild conditions on the residual 
representation. Allen \cite{All16} proved a result similar to Theorem \ref{thm_intro_main_theorem}, but 
assuming a stronger condition on the residual representation 
$\rbar_{\pi,\iota}$, requiring in particular that it be irreducible (similar 
results were also obtained by Breuil--Hellmann--Schraen \cite{Bre17}). These 
works use variants of the Taylor--Wiles method, which is a powerful tool 
for studying the deformation theory of automorphic Galois representations.

Our main motivation for this work was to prove a result valid under very weak 
conditions on the residual representation. In particular, we allow 
the case $p = 2$ and $\rbar_{\pi,\iota}$ trivial, which is rather far from the 
cases allowed by \cite{All16}. For example, we obtain the following results for 
2-dimensional representations over totally real fields.
\begin{introthm}\label{thm:intro2}
	Let $F$ be a totally real number field, and let $p$ be a prime.
	\begin{enumerate}
		\item Let $\pi$ be a non-CM, regular algebraic automorphic representation of $\GL_2(\bA_F)$. Then for any isomorphism $\iota: \overline{\bQ}_p \to \bC$, $H^1_f(F, \ad r_{\pi, \iota}) = 0$.
		\item Let $E$ be a non-CM elliptic curve over $F$, and let $r_p(E) : G_F \to \GL_2(\bQ_p)$ denote the associated $p$-adic representation. Then $H^1_f(F, \ad r_p(E)) = 0$.
	\end{enumerate}
\end{introthm}
We emphasise that no additional conditions are required in either case in order to conclude the vanishing of the adjoint Selmer group.

There are three main innovations that allow us to prove a result like Theorem 
\ref{thm_intro_main_theorem}. The first is a control theorem for studying the 
pseudodeformation theory of a representation $\rho : \Gamma \to \GL_n(\bZ_p)$ 
of a profinite group $\Gamma$. We recall that $\rho$ has an associated 
pseudocharacter $\tr \rho$, which can be defined following either Chenevier 
\cite{chenevier_det} or Lafforgue \cite{Laf18} (the proof that these two 
notions are equivalent being due to Emerson \cite{emerson}). If the residual 
representation $\overline{\rho}$ is absolutely irreducible then it is known 
that deforming $\tr \rho$ is equivalent to deforming $\rho$. In general any deformation of $\rho$ gives rise to a deformation of $\tr \rho$, but the two notions are not equivalent.

Here we use Lafforgue's definition of pseudocharacter to show that that  if $\rho 
\otimes_{\bZ_p} \bQ_p$ is absolutely irreducible, then there is a reasonably 
strong link between deformations and pseudodeformations with coefficients in the ring $\bZ_p \oplus \epsilon \bZ_p / (p^N)$. 
Informally, deformations and pseudodeformations are ``the same'', up to bounded 
torsion which depends only on the image $\rho(\Gamma)$. See Proposition 
\ref{prop_effectivity_of_pseudo_characters} for a precise statement.

The second innovation is the formulation, by Wake and Wang-Erickson \cite{WWE_pseudo_def_cond}, of functors of pseudodeformations satisfying deformation conditions (e.g.\ conditions arising from $p$-adic Hodge theory). This is an indispensable tool for making an effective comparison between pseudodeformation rings and Hecke algebras acting on classical automorphic forms.

The third innovation is related to the use of Taylor--Wiles systems in our 
proof. To make use of Taylor--Wiles systems in the study of automorphic forms 
with integral coefficients, one needs to show that if $q$ is a Taylor--Wiles 
place, then the space of automorphic forms with unramified level at $q$  
is isomorphic to the space of automorphic forms with Iwahori level at $q$, 
after 
localization with respect to a suitable eigenvalue of the $U_q$ operator (see 
for example \cite[Lemma 5.8]{CG}). One can argue along these lines only if the 
residual representation $\overline{\rho}$, unramified at $q$ by hypothesis, has 
the property that $\overline{\rho}(\Frob_q)$ has distinct eigenvalues. This is 
the reason for condition (2) in the statement of \cite[Introduction, 
Theorem]{Kis04a} and of course excludes the case where $\rhobar$ is trivial. 
Without this ``independence of $q$'' statement, one does not 
have the finiteness conditions needed to carry out the Taylor--Wiles patching 
argument, at least as outlined in \cite{MR1440309}. 

In his thesis, Pan \cite{Lue} introduced a surprising technique to circumvent 
this issue. Building on Scholze's interpretation of the Taylor--Wiles patching 
argument using ultrafilters \cite{scholze-lt}, Pan constructs a huge 
``pre-patched module'', and 
then shows that using suitable Hecke operators it can be cut down to a size 
making it suitable for use in the Taylor--Wiles argument. We have adapted his 
arguments to our context (in some ways more elementary, since we work with 
fixed weight 
automorphic forms, whereas \cite{Lue} works with completed cohomology). 

\textbf{Applications.} Results such as Theorem \ref{thm_intro_main_theorem} 
have applications to the geometry of eigenvarieties, and this is one of the 
main motivations for proving them (as was already the case for Kisin 
\cite{Kis04a}). 
This is because one can embed (at least locally around an irreducible point) 
eigenvarieties inside deformation spaces of trianguline representations. In 
many cases, the 
vanishing of $H^1_f(F^+, \ad r_{\pi, \iota})$ can be used to prove that this 
embedding is in fact a local isomorphism.

For example, the vanishing of the adjoint Selmer group is a significant part of 
what it means for a $p$-refined Hilbert modular form to be decent, in the sense 
of \cite{Ber19}, and therefore to admit a $p$-adic L-function with good 
interpolation properties. Another application is that one can use an 
understanding of the geometry of the eigenvariety to prove modularity results 
for Galois representations. This possibility is already suggested in Kisin's work 
\cite[(11.13)]{kisin-ocfmc}. We will take this point of view in \cite{New19b}, 
where Theorem \ref{thm_intro_main_theorem} is one of the key inputs to prove 
the 
automorphy of the symmetric power liftings of level one Hecke eigenforms (for 
example, Ramanujan's modular form 
$\Delta$). 

\textbf{Organization of this paper.} In Section \ref{sec:pseudochar} we 
establish our control theorem relating pseudodeformations and deformations (up 
to bounded torsion), and set up the Galois theoretic ingredients for the 
Taylor--Wiles method. In the short Section \ref{sec:oldforms} we prove a 
simple representation-theoretic result which controls the difference between 
spaces of automorphic forms with hyperspecial and Iwahori level at 
Taylor--Wiles places.
In Section \ref{sec:patching} we carry out 
our variation 
on the 
Taylor--Wiles method (inspired by Pan's work) and prove a special case of 
Theorem 
\ref{thm_intro_main_theorem}. Finally, the general case of Theorem 
\ref{thm_intro_main_theorem}, together with Theorem \ref{thm:intro2} and some 
other applications are deduced in Section \ref{sec:applications} using base 
change and potential automorphy.  

\textbf{Acknowledgements.} J.T.'s work received funding from the European 
Research Council (ERC) under the European Union's Horizon 2020 research and 
innovation programme (grant agreement No 714405). J.N.~would like to thank 
Carl Wang-Erickson for helpful discussions about the work 
\cite{WWE_pseudo_def_cond}. We are grateful to the anonymous referee and to Florian Herzig for their detailed comments on this paper.

\section{Notation and preliminaries}\label{sec_notation}

If $F$ is a field of characteristic zero, we generally fix an algebraic closure $\overline{F} / F$ and write $G_F$ for the absolute Galois group of $F$ with respect to this choice. If $F$ is a number field, then we will also fix embeddings $\overline{F} \to \overline{F}_v$ extending the map $F\to F_v$ for each place $v$ of $F$; this choice determines a homomorphism $G_{F_v} \to G_F$. When $v$ is a finite place, we will write $\cO_{F_v} \subset F_v$ for the valuation ring, $\varpi_v \in \cO_{F_v}$ for a fixed choice of uniformizer, $\Frob_v \in G_{F_v}$ for a fixed choice of Frobenius lift, $k(v) = \cO_{F_v} / (\varpi_v)$ for the residue field, and $q_v = \# k(v)$ for the cardinality of the residue field. When $v$ is a real place, we write $c_v \in G_{F_v}$ for complex conjugation. If $S$ is a finite set of finite places of $F$ then we write $F_S / F$ for the maximal subextension of $\overline{F}$ unramified outside $S$ and $G_{F, S} = \Gal(F_S / F)$.

If $p$ is a prime, then we call a coefficient field a finite extension $E / \bbQ_p$ contained inside our fixed algebraic closure $\overline{\bbQ}_p$, and write $\cO$ for the valuation ring of $E$, $\varpi \in \cO$ for a fixed choice of uniformizer, and $k = \cO / (\varpi)$ for the residue field. We write $\cC_\cO$ for the category of complete Noetherian local $\cO$-algebras with residue field $k$. 

If $A$ is a ring and $\rho : \Gamma \to \GL_n(A)$ is a representation, we write $\ad \rho$ for $M_n(A)$ with its adjoint $\Gamma$-action, and $\ad^0 \rho \subset \ad \rho$ for the $A[\Gamma]$-submodule of trace 0 matrices. We will use the self-duality $\ad \rho \times \ad \rho \to A$, $(X, Y) \mapsto \tr X Y$, to identify $\ad \rho$ with its dual when we e.g.\ define dual Selmer conditions using Tate duality (see for example \S \ref{sec_unitary_pseudocharacters}).

If $G$ is a locally profinite group and $U \subset G$ is an open compact subgroup, then we write $\cH(G, U)$ for the set of compactly supported, $U$-biinvariant functions $f : G \to \bZ$. It is a $\bbZ$-algebra, where convolution is defined using the left-invariant Haar measure normalized to give $U$ measure 1; see \cite[\S 2.2]{new-tho}. It is free as a $\bbZ$-module, with basis given by the characteristic functions $[UgU]$ of double cosets.

Let $K$ be a non-archimedean characteristic $0$ local field, and let $\Omega$ 
be an algebraically 
closed field of characteristic $0$. We write $W_K \subset G_K$ for the Weil group of $K$ and $I_K \subset W_K$ for the inertia subgroup. We use the cohomological normalisation of 
class field theory: it is the isomorphism $\Art_K: K^\times \to W_K^{ab}$ which 
sends uniformizers to geometric Frobenius elements. We use the Tate 
normalisation of the local Langlands 
	correspondence 
for $\GL_n$: it is the bijection $\rec_K^T$ between isomorphism classes of irreducible, 
admissible $\Omega[\GL_n(K)]$-modules and isomorphism classes Frobenius-semisimple Weil--Deligne 
representations over $\Omega$ of rank $n$ which is normalised as in \cite[\S 
2.1]{Clo14}. 

If $\rho : G_K \to \GL_n(\overline{\bbQ}_p)$ is a continuous 
representation (assumed to be de Rham if $p$ equals the residue characteristic 
of 
$K$), then we write $\mathrm{WD}(\rho) = (r, N)$ for the associated 
Weil--Deligne representation, and $\mathrm{WD}(\rho)^{F-ss}$ for its 
Frobenius semisimplification. \begin{defn}
We say that a Weil--Deligne representation 
$(r,N)$ is 
generic if there is no non-zero morphism $(r,N) \to (r(1),N)$. We say that a 
continuous representation $\rho$ is generic if $WD(\rho)$ is generic.\end{defn} 
We note that if $WD(\rho)^{F-ss}$ is generic, then $\rho$ is generic. It 
follows from \cite[Lemma 1.1.3]{All16} that if $\pi$ is a generic irreducible 
admissible $\Qpbar[\GL_n(K)]$-module and $WD(\rho)^{F-ss} = \rec_K^T(\pi)$, 
then $\rho$ is generic. 

If $p$ equals the residue characteristic of $K$ and $V$ is the $E$-vector space 
on which 
$\rho$ acts (for some $E \subset \Qpbar$ finite over $\Qp$ with 
$\rho(G_K)\subset\GL_n(E)$), we have subspaces 
\[H^1_f(K,V) \subset H^1_g(K,V) \subset H^1(K,V) 
\] defined by \begin{align*}H^1_f(K,V) &= \ker\left(H^1(K,V) \to 
H^1(K,V\otimes_{\Qp}B_{crys})  \right)\\ H^1_g(K,V) &= \ker\left(H^1(K,V) \to 
H^1(K,V\otimes_{\Qp}B_{dR})\right). \end{align*} We have $H^1_f(K,\End(V)) = 
H^1_g(K,\End(V))$ if and only if $\rho$ is generic \cite[Remark 
1.2.9]{All16}. Similarly, if $p$ does not equal the residue characteristic of 
$K$, we have a subspace $H^1_{ur}(K,V) = \ker\left(H^1(K,V) \to 
H^1(I_K,V)\right)$. For notational compatibility with the $p$-adic case we 
write $H^1_f(K,V)= H^1_{ur}(K,V)$ and $H^1_g(K,V)= H^1(K,V)$. Then we again 
have $H^1_{f}(K,\End(V)) = H^1_g(K,\End(V))$ if and only 
if $\rho$ is 
generic. 

Let $F$ be a number field, and let $S$ be a finite set of finite places of $F$, 
containing the $p$-adic places $S_p$. Let $r: G_{F,S} \to \GL_n(\Qpbar)$ be a 
continuous representation, with underlying $E$-vector space $V$. We have 
global Selmer groups \[H^1_f(F,V) \subset H^1_{g,S}(F,V) \subset 
H^1(F_S/F,V)\] defined by
\begin{align*}H^1_f(F,V) &= \ker\left(H^1(F_S/F,V) \to \prod_{v \in 
S}H^1(F_v,V)/H^1_f(F_v,V)\right)\\
H^1_{g,S}(F,V) &= \ker\left(H^1(F_S/F,V) \to \prod_{v \in 
S}H^1(F_v,V)/H^1_g(F_v,V)\right)\\ &= \ker\left(H^1(F_S/F,V) \to 
\prod_{v \in S_p}H^1(F_v,V)/H^1_g(F_v,V)\right).\end{align*}  We note our convention that $H^1(F_S / F, \ast)$ denotes group cohomology for the group $G_{F, S}$. The group $H^1_f(F,V)$ 
does 
not change when $S$ is enlarged (this is why we do not record $S$ in the 
notation).

If $F$ is a number field and $\pi$ is an automorphic representation of 
$\GL_n(\A_F)$, we say that $\pi$ is regular algebraic if $\pi_\infty$ has the 
same infinitesimal character as an irreducible algebraic representation of 
$\Res_{F/\bQ}\GL_n$. We recall (cf. 
\cite[\S 2.1]{BLGGT}) that if $F$ is a totally real or CM number field, then a 
pair 
$(\pi,\chi)$ comprising an automorphic 
representation $\pi$ of $\GL_n(\bA_F)$ and a Hecke character $\chi : 
(F^+)^\times 
\backslash (\bA_{F^+})^\times \to \bC^\times$ is said to be polarized if there is an isomorphism $\pi^c \cong \pi^\vee \otimes (\chi \circ 
	\mathbf{N}_{F / F^+})$ and, if $F$ is CM, then $\chi_v(-1) = (-1)^n$ for each place $v | \infty$ of $F$. (The sign condition of \cite{BLGGT} in the case that $F$ is totally real can be suppressed, as a consequence of \cite[Theorem 2.1]{Pat15}.) The automorphic representations of unitary type discussed in our 
introduction correspond to polarized automorphic representations 
$(\pi,\delta_{F/F^+}^n)$, where $\delta_{F/F^+}$ is the quadratic character for 
$F/F^+$.

If $(\pi,\chi)$ is a regular algebraic, cuspidal, polarized automorphic 
representation, then for any isomorphism $\iota:\Qpbar \to \CC$ there is an 
associated Galois representation (we refer to \cite[Theorem 2.1.1]{BLGGT} for 
its 
properties) \[r_{\pi,\iota}: G_F\to 
\GL_n(\Qpbar).\] 

If $F$ is CM, then $r_{\pi,\iota}$ extends to a homomorphism $r_{\pi, \iota}:  
G_{F^+}\to 
	\cG_n(\Qpbar)$, with multiplier character $\nu\circ r_{\pi, \iota} = 
	\epsilon^{1-n}r_{\chi,\iota}$ ($\cG_n$ is the algebraic group defined in 
	\cite[\S2.1]{cht}; here the word `extends' is interpreted following the convention described at the top of \cite[p.~ 8]{cht}). This defines an extension of the $G_F$ action on $\ad 
	r_{\pi,\iota}$ to an action of $G_{F^+}$. More explicitly, if we  fix a 
	choice $c \in G_{F^+}$ of complex conjugation, there is a 
	perfect, 
	symmetric pairing $\langle \cdot, \cdot \rangle$ on $\Qpbar^n$ such that 
	\[\langle r_{\pi,\iota}(\sigma)v, r_{\pi,\iota}(\sigma^c)w\rangle = 
	(\epsilon^{1-n}r_{\chi,\iota}(\sigma))\langle v, w\rangle \] for all 
	$\sigma \in G_F, v, w \in \Qpbar^n$ and $c$ acts on $\ad r_{\pi,\iota} = 
	\End(\Qpbar^n)$ 
	by $X \mapsto -X^*$, where $X^*$ is the adjoint with respect to 
	$\langle\cdot,\cdot\rangle$. 

\section{Pseudocharacters}\label{sec:pseudochar}

In this paper we use Lafforgue's notion of pseudocharacter for a reductive group in the case of $\GL_n$ (see \cite[\S 11]{Laf18} or \cite[\S 4]{Boc19}), and Chenevier's notion of group determinant \cite{chenevier_det}. In fact, these are equivalent, but both definitions are useful. We will prove a new result about the deformation theory of pseudocharacters (Proposition \ref{prop_effectivity_of_pseudo_characters}) using Lafforgue's point of view, while we follow \cite{WWE_pseudo_def_cond} in using Chenevier's definition to impose deformation conditions on pseudocharacters. 

\subsection{Pseudocharacters vs.~determinants}

We begin by recalling the relevant definitions. Let $\Gamma$ be a group and fix $n \geq 1$. 
\begin{defn}
	A pseudocharacter of $\Gamma$ of dimension $n$ over a ring $A$ is a collection $\Theta = (\Theta_m)_{m \geq 1}$ of algebra homomorphisms $\Theta_m : \bbZ[\GL_n^m]^{\GL_n} \to \mathrm{Map}(\Gamma^m, A)$ satisfying the following conditions:
\begin{enumerate}
	\item For all $k, l \geq 1$ and for each map $\zeta : \{ 1, \dots, k \} \to \{ 1, \dots, l \}$, each $f \in \bbZ[\GL_n^k]^{\GL_n}$, and each $\gamma_1, \dots, \gamma_l \in \Gamma$, we have
	\[ \Theta_l(f^\zeta)(\gamma_1, \dots, \gamma_l) = \Theta_k(f)(\gamma_{\zeta(1)}, \dots, \gamma_{\zeta(k)}), \]
	where $f^\zeta(g_1, \dots, g_l) = f(g_{\zeta(1)}, \dots, g_{\zeta(k)})$.
	\item For each $k \geq 1$, for each $\gamma_1, \dots, \gamma_{k+1} \in \Gamma$, and for each $f \in \bbZ[\GL_n^k]^{\GL_n}$, we have 
	\[ \Theta_{k+1}(\hat{f})(\gamma_1, \dots, \gamma_{k+1}) = \Theta_{k}(f)(\gamma_1, \dots, \gamma_k \gamma_{k+1}), \]
	where $\hat{f}(g_1, \dots, g_{k+1}) = f(g_1, \dots, g_k g_{k+1})$.
\end{enumerate}
\end{defn}
If $\rho : \Gamma \to \GL_n(A)$ is a representation, then we can define its associated pseudocharacter $t = (t_m)_{m \geq 1} = \tr \rho$ by the formula 
\[ t_m(f)(\gamma_1, \dots, \gamma_n) = f(\rho(\gamma_1), \dots, \rho(\gamma_m)). \]
One can define the operations of twisting and duality on pseudocharacters in a way compatible with the usual operations on representations. For example, let $i : \GL_n \to \GL_n$ be the involution given by $i(g) = {}^t g^{-1}$. If $t$ is a pseudocharacter, then we define a new pseudocharacter $t^\vee$ by the formula $t^\vee_m(f)(\gamma_1, \dots, \gamma_m) = t_m(f')(\gamma_1, \dots, \gamma_m)$, where $f' \in \bbZ[\GL_n^m]$ is defined by 
\[ f'(g_1, \dots, g_m) = f(i(g_1), \dots, i(g_m)). \]
If $t = \tr \rho$, then $t^\vee = \tr \rho^\vee$. 

Similarly, if $\chi : \Gamma \to A^\times$ is a character, then we define the twist $t \otimes \chi$ by the formula $(t \otimes \chi)_m( f )(\gamma_1, \dots, \gamma_m) = f'(\gamma_1, \dots, \gamma_m)$, where $f' \in A[\GL_n^m]^{\GL_n}$ is defined by the formula $f'(g_1, \dots, g_m) = f(\chi_1(\gamma_1) g_1, \dots, \chi_m(\gamma_m) g_m)$. If $t = \tr \rho$, then $t \otimes \chi = \tr (\rho \otimes \chi)$. 

Before giving the definition of group determinant, we fix some notation. Let $A$ be a ring and let $A\operatorname{-alg}$ be the category of commutative $A$-algebras. If $M$ is an $A$-module, then we write $h_M : A\operatorname{-alg} \to \operatorname{Sets}$ for the functor $B \mapsto M \otimes_A B$. 
\begin{defn}
	A group determinant of $\Gamma$ of dimension $n$ over a ring $A$ is a 
	natural transformation of functors $D : h_{A[\Gamma]} \to h_A$ satisfying 
	the following conditions on the induced map $B[\Gamma] \to B$ for every $B 
	\in A\operatorname{-alg}$:
	\begin{enumerate}
		\item $D(1) = 1$.
		\item For any $x, y \in B[\Gamma]$, $D(xy) = D(x) D(y)$.
		\item For any $x \in B[\Gamma]$, $b \in B$, $D(bx) = b^n D(x)$.
	\end{enumerate}
\end{defn}
If $\rho : \Gamma \to \GL_n(A)$ is a representation, then we can define its 
associated group determinant $D(x) = \det(\rho(x))$ (where we extend $\rho$ to 
a homomorphism $\rho : B[\Gamma] \to M_n(B)$ for any $A$-algebra $B$). We omit 
the formulae for the dual or twist of a group determinant.

We now describe the relation between pseudocharacters and group determinants. For each $i = 0, \dots, n$, let $\lambda_i \in \bZ[\GL_n]^{\GL_n}$ be defined by the equation $\det(X - g) = \sum_{i=0}^n (-1)^i \lambda_i(g) X^{n-i} $. If $t$ is a pseudocharacter, we have functions ($i = 0, \dots, n$)
\[ t^{[i]} : \Gamma \to A \]
given by the formulae $t^{[i]}(\gamma) = t_1(\lambda_i)(\gamma)$. By \cite[\S3.1]{Don92}, for any $m \geq 1$
$\bZ[\GL_n^m]^{\GL_n}$ is generated as a ring by the functions 
$\lambda_i(g_{i_1} \dots g_{i_r})$ ($r \in \NN, 1 \leq i_1, \dots, i_r \leq 
m$), together with $\det(g_1 \dots g_m)^{-1}$. The axioms defining a 
pseudocharacter show that we have
\begin{equation}\label{eqn_donkin} t_m(\lambda_i(g_{i_1} \dots g_{i_r}))(\gamma_1, \dots, \gamma_m) = t_1(\lambda_i(g))(\gamma_{i_1} \dots \gamma_{i_r}). 
\end{equation}
It follows that the functions $t^{[i]}$ ($i = 0, \dots, n$) together determine $t$.

If $D$ is a group determinant, then we define functions ($i = 0, \dots, n$)
\[ D^{[i]} : \Gamma \to A \]
by the formula $D(X - \gamma) = \sum_{i=0}^n (-1)^i D^{[i]}(\gamma)  X^{n-i}$ (evaluation of $D$ over the ring $A[X]$). The functions $D^{[i]}$ ($i = 0, \dots, n$) together determine $D$ (by Amitsur's formula, cf. \cite[Lemma 1.12]{chenevier_det}).
\begin{theorem}\label{thm_emerson}
	For any group $\Gamma$ and ring $A$, the pseudocharacters $t$ of dimension 
	$n$ are in canonical bijection with the group determinants $D$ of dimension 
	$n$. This bijection is characterized by the equality $t^{[i]} = 
	D^{[i]}$ for each $i = 0, \dots, n$.
\end{theorem}
\begin{proof}
	See \cite[Theorems 4.0.1 and 5.0.1]{emerson}, which explicitly construct a bijection between the two classes of objects. Suppose that $t$, $D$ are associated. Then for any $\gamma \in \Gamma$, $t_1(\gamma)$ determines a ring homomorphism $\bbZ[\GL_n]^{\GL_n} \to A$, hence a ring homomorphism $t_1(\gamma)[X] : \bbZ[\GL_n]^{\GL_n} \otimes_\bbZ \bbZ[X] \to A[X]$. We may think of $\det(X - g)$ as an element of $\bbZ[\GL_n]^{\GL_n} \otimes_\bbZ \bbZ[X]$, and the proof of \cite[Theorem 4.0.1]{emerson} shows that if $\gamma \in \Gamma$ then $D(X - \gamma) = t_1(\gamma)[X](\det(X-g))$. This equality is equivalent to equalities $D^{[i]}(\gamma) = t^{[i]}(\gamma)$ ($i = 0, \dots, n$).
\end{proof}
We now discuss continuity. Suppose therefore that $\Gamma$ is a profinite group and $A$ is a topological ring. The definitions are as follows. 
\begin{defn} Let $t = (t_m)_{m \geq 1}$ be a pseudocharacter. We say that $t$ is continuous if for each $m \geq 1$, $t_m$ takes values in the set $\operatorname{Map}_{cts}(\Gamma^m, A)$ of continuous functions $\Gamma^m \to A$.
\end{defn}
\begin{defn}
	Let $D$ be a group determinant. We say that $D$ is continuous if each function $D^{[i]}$ ($i = 0, \dots, n$) is continuous.
\end{defn}
It is clear from the definitions that if $\rho : \Gamma \to \GL_n(A)$ is a continuous representation, then $\tr \rho$ is continuous as a pseudocharacter. 
\begin{proposition}
	Let $t = (t_m)_{m \geq 1}$ and $D$ be associated under the bijection of Theorem \ref{thm_emerson}. Then $t$ is continuous if and only if $D$ is.
\end{proposition}
\begin{proof}
	In light of Theorem \ref{thm_emerson}, it is enough to show that if $t$ is 
	a pseudocharacter such that each function $t^{[i]}$ ($i = 0, \dots, n$) is continuous, then $t$ is continuous. This is again a consequence of (\ref{eqn_donkin}) and \cite[\S3.1]{Don92}.
\end{proof}

\subsection{Pseudocharacters vs.~representations}

Now let $p$ be a prime, let $E / \bbQ_p$ be a finite extension, and let 
$\Gamma$ be a profinite group. Let $\rho : \Gamma \to \GL_n(\cO)$ be a 
continuous homomorphism which is absolutely irreducible over $E$. Let $t = 
(t_m)_{m \geq 1} = \tr \rho$ denote the pseudocharacter associated to $\rho$. 
We consider liftings of $\rho$ and of $t$ to the ring $A = \cO \oplus \epsilon 
E / \cO$ (with $\epsilon^2 = 0$). Clearly if $\rho' : \Gamma \to \GL_n(A)$ is a 
lifting of $\rho$, in the sense that $\rho' \text{ mod }(\epsilon) = \rho$, 
then $t' = \tr \rho'$ is a lifting of $t$. We want to show that in fact 
deforming $\rho$ in this way is not too far from deforming $t$. 

We write $\alpha_k : A \to A$ for the $\cO$-algebra homomorphism which acts as multiplication by $p^k$ on the ideal $(\epsilon) \subset A$. We will prove:
\begin{proposition}\label{prop_effectivity_of_pseudo_characters}
	There exists an integer $k_0 \geq 0$, depending only on $\rho(\Gamma)$, with the following properties:
	\begin{enumerate}
		\item For any lifting $t'$ of $t$ to $A$, there exists a homomorphism 
		$\rho' : \Gamma \to \GL_n(A)$ lifting $\rho$ such that $\tr \rho' = 
		\alpha_{k_0} \circ t'$. If $t'$ is continuous, then $\rho'$ can be 
		chosen to be continuous.
		\item If $\rho'_1, \rho'_2 : \Gamma \to \GL_n(A)$ are two liftings with $\tr \rho'_1 = \tr \rho'_2$, then $\alpha_{k_0} \circ \rho'_1$ and $\alpha_{k_0} \circ \rho'_2$ are conjugate under the action of the group $1 + \epsilon M_{n \times n}(E / \cO) \subset \GL_n(A)$; and if $X \in M_{n \times n}(E / \cO)$ is such that $1 + \epsilon X$ centralizes $\rho'_1$, then $p^{k_0} X$ is a scalar matrix. 
	\end{enumerate}
\end{proposition}
For any $m \geq 1$, we define $X_m = (\GL_{n, \cO})^m$, and $Y_m = \Spec 
\cO[\GL_n^m]^{\GL_n}$. We write $\pi_m : X_m \to Y_m$ for the tautological 
morphism. We fix elements $\gamma_1, \dots, \gamma_m \in \Gamma$ such that 
$\rho(\gamma_1), \dots, \rho(\gamma_m)$ generate a Zariski dense subgroup of 
$\rho(\Gamma)$. We may assume that $m \ge 2$.

Let 
\[ x = (g_1, \dots, g_m) = (\rho(\gamma_1), \dots, \rho(\gamma_m)) \in X_m(\cO). \]
If $\gamma, \delta \in \Gamma$, then we define 
\[ x(\gamma) = (\rho(\gamma_1), \dots, \rho(\gamma_m), \rho(\gamma)) \in X_{m+1}(\cO) \]
and 
\[ x(\gamma, \delta) = (\rho(\gamma_1), \dots, \rho(\gamma_m), \rho(\gamma), \rho(\delta)) \in X_{m+2}(\cO). \]
We define $y = \pi_m(x)$, $y(\gamma) = \pi_{m+1}(x(\gamma))$, and $y(\gamma, \delta) = \pi_{m+2}(x(\gamma, \delta))$. Before going further, we recall the following lemma.
\begin{lemma}\label{lem_conormal_sheaves}
	Let $\pi : X \to Y$ be a separated morphism of schemes of finite type over 
	a base $S$. Let $G$ be a separated group scheme, smooth and of finite type 
	over $S$, and suppose that $G$ acts on $X$ in such a way that $\pi$ is 
	$G$-equivariant for the trivial action of $G$ on $Y$. Then:
	\begin{enumerate}
		\item There is a canonical morphism $\Omega_{X / Y} \to \Hom_{\cO_S}(\Lie G, \cO_X)$ of coherent sheaves of $\cO_X$-modules.
		\item If $\pi$ is a $G$-torsor, then this morphism is an isomorphism.
	\end{enumerate}
\end{lemma}
\begin{proof}
	Let $e : X \to G \times X$ be the morphism $x \mapsto (e, x)$, and let $\mu 
	: G \times X \to X \times_Y X$ be the morphism $(g, x) \mapsto (x, gx)$. 
	Then $\mu \circ e$ is the diagonal embedding, and both $e$ and $\mu \circ 
	e$ are closed immersions. The sheaf $\Hom_{\cO_S}(\Lie G, \cO_X)$ may be 
	identified with the conormal sheaf of the morphism $e$ (see e.g.\ \cite[II, Lemme 4.11.7]{SGA3})
while 
	$\Omega_{X/Y}$ may be identified with the conormal sheaf 
	of the morphism $\mu \circ e$. The existence of the morphism therefore 
	follows from \cite[\href{https://stacks.math.columbia.edu/tag/01R4}{Lemma 
	01R4}]{stacks-project}.
	
	Now suppose that $\pi$ is a $G$-torsor. In this case $\mu$ is an isomorphism, and the statement is immediate. 
\end{proof}
Let $\frg = \operatorname{Lie} \PGL_{n, \cO}$ and $\frg^\ast = \Hom(\frg, 
\cO)$, $\frg^\ast_{X_m} = \frg^\ast \otimes_{\cO} \cO_{X_m}$. We apply Lemma 
\ref{lem_conormal_sheaves} to the morphisms $\pi_k : X_k \to Y_k$, with $G = 
\PGL_{n, \cO}$, to obtain complexes (not necessarily exact) of coherent sheaves 
on $X_k$:
\begin{equation*}\label{eqn_int_1}
(\star_k) : 0 \to \pi_k^\ast\Omega_{Y_k / \cO} \to \Omega_{X_k / \cO} \to \frg^\ast_{X_k} \to 0.
\end{equation*}
We observe that e.g.\ $i_x^\ast (\star_m)$ is the complex
\[ i_x^\ast (\star_m) : 0 \to T_y^\ast Y_m \to T_x^\ast X_m \to \frg^\ast \to 0. \]
Here we write $T_x^\ast X_m = i_x^\ast \Omega_{X_m / \cO}$, by definition, and call it the Zariski cotangent space of $X_m$ at the point $x$. 
\begin{lemma}\phantomsection\label{lem_rational_fibres}
	\begin{enumerate}
		\item The complex $(\star_m)[1/p]$ on $X_m[1/p]$ is an exact sequence 
		of locally free sheaves above a Zariski open neighbourhood of the point 
		$y$.
		\item The complex $(\star_{m+1})[1/p]$ on $X_{m+1}[1/p]$ is an exact 
		sequence of locally free sheaves above a Zariski open neighbourhood of 
		the point $y(\gamma)$, for any $\gamma \in \Gamma$.
		\item The complex $(\star_{m+2})[1/p]$ on $X_{m+2}[1/p]$ is an exact 
		sequence of locally free sheaves above a Zariski open neighbourhood of 
		the point $y(\gamma, \delta)$, for any $\gamma, \delta \in \Gamma$.
	\end{enumerate}
\end{lemma}
\begin{proof}
	We show that $\pi_m[1/p]$ is a $\PGL_{n, E}$-torsor above a Zariski open 
	neighbourhood of $y$. The same proof shows the analogous statement for the 
	points $y(\gamma)$ and $y(\gamma, \delta)$, and in each case implies the 
	statement in the lemma (since a $\PGL_{n, E}$-torsor is in particular 
	smooth). Let $U$ denote the open subset of $X_m[1/p]$ corresponding to 
	tuples $(u_1, \dots, u_m)$ which generate an absolutely irreducible 
	subgroup of $\GL_n$. Then \cite[Theorem 4.1]{Ric88} shows that $U$ is 
	precisely the set of stable points of $X_m[1/p]$ (for the action of 
	$\PGL_{n, E}$). In particular, $\pi_m(U)$ is an open subset of $Y_m[1/p]$ 
	and $U = \pi_m^{-1}(\pi_m(U))$. Each point of $U$ has a trivial stabilizer 
	for the $\PGL_{n, E}$ action (Schur's lemma), so it follows from 
	\cite[Proposition 
	8.2]{Bar85} that $\pi_m|_U : U \to \pi_m(U)$ is a $\PGL_{n, E}$-torsor, as 
	required.
\end{proof}
We can use the compactness of $\Gamma$ to upgrade the previous lemma to the following uniform integral statement. 
\begin{lemma}\label{lem_uniform_bound}
	We can find an integer $k_1 \geq 0$ with the following properties: 
	\begin{enumerate}
		\item The cohomology of the complex $i_x^\ast (\star_m)$ is annihilated 
		by $p^{k_1}$.
		\item For any $\gamma \in \Gamma$, the cohomology of the complex 
		$i_{x(\gamma)}^\ast (\star_{m+1})$ is annihilated by $p^{k_1}$.
		\item For any $\gamma, \delta \in \Gamma$, the cohomology of the 
		complex $i_{x(\gamma, \delta)}^\ast (\star_{m+2})$ is annihilated by 
		$p^{k_1}$.
	\end{enumerate}
\end{lemma}
\begin{proof}
	The first part of the lemma follows by Lemma \ref{lem_rational_fibres}. In fact, we can find numbers $k$, $k(\gamma)$, and $k(\gamma, \delta)$ such that the requirements of each point of the lemma are satisfied if $k_1$ is replaced by $k$, $k(\gamma)$, and $k(\gamma, \delta)$ in each case. What we must show is that we can find $k_1$ which exceeds $k$, $k(\gamma)$, and $k(\gamma, \delta)$ for all $\gamma, \delta \in \Gamma$. 
	
	To this end, let us suppose that $k$, $k(\gamma)$, and $k(\gamma, \delta)$ have been chosen to each take their smallest possible values. It suffices to show that $k(\gamma)$ and $k(\gamma, \delta)$ are locally constant as functions of $\gamma, \delta \in \Gamma$. Since $\Gamma$ is compact, this will imply that they are in fact bounded. This local constancy is a consequence of the second part of Lemma \ref{lem_locally_constant_p_torsion} below.
\end{proof}
\begin{lemma}\label{lem_locally_constant_p_torsion}
	Let $Z$ be a scheme of finite type over $\cO$. If $z \in Z(\cO)$, we write 
	$i_z : \Spec \cO \to Z$ for the corresponding morphism.
	\begin{enumerate}
		\item Let $\cF$ be a coherent sheaf on $Z$ such that $\cF[1/p]$ is 
		locally free on a Zariski open neighbourhood $V_z$ of $z \in 
		Z[1/p]$. 
	 Then for any 
		$z \in 
		Z(\cO)$, there exists an 
		open (for the $p$-adic topology) neighbourhood $U$ of $z$ in $Z(\cO)$ 
		such that for any $z' \in U$, $i_{z'}^\ast \cF \cong i_z^\ast \cF$ as 
		$\cO$-modules.
		\item Let $z \in Z(\cO)$ and let
		\[ (\star) : 0 \to \cF_1 \to \cF_2 \to \cF_3 \to 0 \]
		be a complex of coherent sheaves on $Z$, not necessarily exact, but 
		such that on a Zariski open neighbourhood of $z \in Z[1/p]$
		\[ (\star[1/p]) : 0 \to \cF_1[1/p] \to \cF_2[1/p] \to \cF_3[1/p] \to 0 \]
		 is an exact sequence of locally free sheaves. Then there exists a 
		 $p$-adically open 
		 neighbourhood $U$ of $z$ in $Z(\cO)$ and an integer $N \geq 0$ such 
		 that for all $z' \in U$, $p^N H^\ast(i_{z'}^\ast (\star)) = 0$.
	\end{enumerate}	 
\end{lemma}
\begin{proof}
	In each case we are free to replace $Z$ by a Zariski open neighbourhood of 
	the closed point specializing $z$. We can therefore assume that $Z = \Spec 
	B$ is affine. In the 
	first case we can assume that $\cF$ corresponds to a finite $B$-module $M$ 
	and that there is an exact sequence 
	\[ B^a \to B^b \to M \to 0. \]
	We may assume that $\cF[1/p]$ has constant rank $b-k$ on $V_z$, so we get a 
	continuous map $Z(\cO)\cap V_z(E) \to M_{a \times b}(\cO) \cap M_{a 
	\times b, 
	k}(E)$, where $M_{a \times b, k} \subset M_{a \times b}$ is the locally 
	closed subscheme of matrices of rank $k$ (equivalently with vanishing $(k+1) \times (k+1)$ minors but at least one non-zero $k \times k$ minor). Note that $Z(\cO)\cap V_z(E)$ is 
	$p$-adically open in $Z(\cO)$. We 
	are therefore reduced to showing that any matrix $T \in M_{a 
	\times b}(\cO) \cap M_{a \times b, k}(E)$ has an open neighbourhood $U$ 
	such that for $T' \in U$, $\cO^b / T' \cO^a$ is isomorphic to $\cO^b / T 
	\cO^a$. In other words, we need to show that there is an open neighbourhood 
	$U$ in which the Smith normal form of $T$ is constant. Let $m$ be the 
	largest valuation of a non-zero minor of $T$. Choosing $U$ so that for any 
	$T' \in U$, the minors of $T'$ are congruent to those of $T$ modulo 
	$\varpi^{m+1}$, we see that the Smith normal forms of $T$ and $T'$ are 
	indeed equal. (Note that the assumption that the $E$-rank is constant is 
	necessary; otherwise we have the example of $M = \cO[x] / x$ at the point 
	$x = 0$, where $\cO$ is a limit of $\cO / (\varpi^N)$'s.)

	We now turn to the second part. It suffices to show that we can find an 
	integer $N \geq 0$ and a $p$-adically open neighbourhood $U$ of $z$ such 
	that for all $z' \in U$, $p^N$ annihilates $\ker (i_{z'}^\ast \cF_2 \to 
	i_{z'}^\ast \cF_3) / \im (i_{z'}^\ast \cF_1 \to i_{z'}^\ast \cF_2)$. Our 
	hypotheses imply that for $z'$ in a Zariski open neighbourhood of $z$, this 
	group is contained in the torsion subgroup of $i_{z'}^\ast \cF_2 / \im 
	(i_{z'}^\ast \cF_1 \to i_{z'}^\ast \cF_2) = i_{z'}^\ast(\cF_2 / \im(\cF_1 
	\to \cF_2))$, so the result follows on applying the first part to $\cF_2 
	/\im(\cF_1 \to \cF_2)$. 
\end{proof}
We are now in a position to prove Proposition \ref{prop_effectivity_of_pseudo_characters}. Recall that we write $A = \cO \oplus \epsilon E / \cO$. It is helpful to first note that if  $X$ is a scheme over $\cO$ and $x \in X(\cO)$, then the fibre of $X(A) \to X(\cO)$ above $x$ is canonically identified with $\Hom_{\cO}(T_x^\ast X, E / \cO)$. 
\begin{proof}[Proof of Proposition \ref{prop_effectivity_of_pseudo_characters}]
	Let the integer $k_1$ be as in Lemma \ref{lem_uniform_bound}. We will show that we can take $k_0 = 6k_1$. Taking the Pontryagin dual of $i_x^\ast (\star_m)$ and $i_{x(\gamma)}^\ast (\star_{m+1})$ gives us, for any $\gamma \in \Gamma$, a commutative diagram
	\[\scalebox{0.8}{ \xymatrix{ 0 \ar[r] & \frg \otimes_{\cO} E / \cO \ar[r]  
	& \Hom_{\cO}(T_x^\ast X_m, E / \cO) \ar[r] & \Hom_{\cO}( T_y^\ast Y_m, E / 
	\cO) \ar[r] & 0 \\
		0 \ar[r] &\frg \otimes_{\cO} E / \cO \ar[u] \ar[r] & 
		\Hom_{\cO}(T_{x(\gamma)}^\ast X_{m+1}, E / \cO) \ar[r] \ar[u] & 
		\Hom_{\cO}( T_{y(\gamma)}^\ast Y_{m+1}, E / \cO) \ar[r] \ar[u] & 0.}} \]
	The first vertical arrow is the identity, while the other two vertical arrows correspond to forgetting the last entry in $\GL_n^{m+1}$. Both rows have cohomology annihilated by $p^{k_1}$. Consequently the induced map
	\begin{equation}\label{eqn_cartesian} \begin{split}  \Hom_{\cO}&(T_{x(\gamma)}^\ast X_{m+1}, E / \cO) \\ &\to \Hom_{\cO}(T_x^\ast X_m, E / \cO) \times_{\Hom_{\cO}( T_y^\ast Y_m, E / \cO)} \Hom_{\cO}( T_{y(\gamma)}^\ast Y_{m+1}, E / \cO)  \end{split} \end{equation}
	has kernel and cokernel annihilated by $p^{2k_1}$. In particular, given an element $z$ of the target we can define an element of the source unambiguously as follows: choose a pre-image $z'$ of $p^{2 k_1} z$. Then $z'' = p^{2 k_1} z'$ depends only on $z$ and has image $p^{4k_1} z$.
	
	Now suppose given a pseudocharacter $t'$ over $A$ lifting $t$. The data of the pseudocharacter $t'$ (more precisely, $t'_m$) determines an element $y' \in \Hom( T_y^\ast Y_m, E / \cO)$. We fix a choice of $x' \in \Hom_{\cO}(T_x^\ast X_m, E / \cO)$ with image equal to $p^{k_1} y'$. This corresponds to a tuple of elements $(g_1', \dots, g'_m) \in \GL_n(A)^m$ lifting the element $(g_1, \dots, g_m) = (\rho(\gamma_1), \dots, \rho(\gamma_m))$. If $x''$ is any other choice of element with image equal to $p^{k_1} y'$, then $p^{k_1} x - p^{k_1} x''$ is in the image of $\frg \otimes_\cO E / \cO$.
	
	 The pseudocharacter $t'$ also determines elements \[ y'(\gamma) 
	 \in \Hom( T_{y(\gamma)}^\ast Y_{m+1}, E / \cO) \] for any $\gamma \in 
	\Gamma$, and the pair $(p^{k_1} x', p^{2k_1} y'(\gamma))$ lies on the 
	right-hand side of (\ref{eqn_cartesian}). We may define $\rho'(\gamma)$ 
	uniquely as follows: it is the lift of $p^{4k_1} (p^{k_1} x', 
	p^{2k_1} y'(\gamma))$ associated to the map (\ref{eqn_cartesian}). 
	Then $\rho' : \Gamma \to \GL_n(A)$ has associated trace function $\tr \rho' 
	= \alpha_{6k_1} \circ t'$, and its conjugacy class under  $1 + \epsilon \frg \otimes_\cO 
	E / \cO$ is independent of choices. We can verify that $\rho'$ is a 
	homomorphism (i.e.\ that it respects 
	multiplication) using the fact that $t'$ is a 
	pseudocharacter, together with a diagram with rows corresponding to 
	elements $x(\gamma \delta)$ and $x(\gamma, \delta)$. 
	
	Now suppose given two liftings $\rho'_1, \rho'_2$ of $\rho$ to $A$ with 
	$\tr \rho'_1 = \tr \rho'_2 = t'$, say. 	This implies that for each $\gamma 
	\in \Gamma$, the tuples $(\rho'_i(\gamma_1), \dots, \rho'_i(\gamma_m), 
	\rho'_i(\gamma))$ ($i = 1, 2$), when identified with elements of 
	$\Hom_{\cO}( T_{x(\gamma)}^\ast X_{m+1}, E / \cO)$, have equal image in 
	$\Hom_{\cO}( T_{y(\gamma)}^\ast Y_{m+1}, E / \cO)$. Consequently there is 
	$X_\gamma$ in $\frg \otimes_{\cO} E / \cO$ taking $\alpha_{k_1}$ of the 
	first 
	tuple to $\alpha_{k_1}$ of the second. Passing to the top row of the 
	commutative
	diagram, we see that for any $\gamma, \gamma' \in \Gamma$, we have 
	$p^{k_1}(X_\gamma - X_{\gamma'}) = 0$, hence $X = p^{k_1} X_\gamma$ is 
	independent of $\gamma \in \Gamma$. It follows that $X$ takes 
	$\alpha_{2k_1} \circ \rho'_1$ to $\alpha_{2k_1} \circ \rho'_2$. 
	
	It remains to verify that if $t'$ is continuous, then so is $\rho'$. It is 
	enough to show that for each $s \geq 1$, there exists an open subgroup $N 
	\subset \Gamma$ such that $\rho'(N)$ takes values in the subgroup $1 + 
	\varpi^s M_n(\cO)$ of $\GL_n(A)$. Since $\bbZ[\GL_n^{m+1}]^{\GL_n}$ is a 
	finitely generated $\Z$-algebra and $\Gamma$ is compact, there exists $r 
	\ge 1$ such that $t'_{m+1}$ takes values in 
	$\mathrm{Map}(\Gamma^{m+1},A_r)$, 
	where $A_r = \cO \oplus \epsilon \varpi^{-r}\cO/\cO \subset A$. Increasing $s$, we can assume that $s \geq r$, that $\rho'(\gamma_1), \dots, \rho'(\gamma_m)$ lie in $\GL_n(A_s)$, and that there exists an open subgroup $N \subset \Gamma$ such that for all $\gamma \in N$, $\rho(\gamma) \in 1 + \varpi^s M_n(\cO)$ and $t'_{m+1}(\gamma_1, \dots, \gamma_m, \gamma) \equiv t'_{m+1}(\gamma_1, \dots, \gamma_m, 1) \text{ mod } \varpi^s A_s$. We observe that for $\gamma \in N$, there is a commutative square
	\[\scalebox{0.8}{\xymatrix{\Hom_{\cO}(T_{x(\gamma)}^\ast X_{m+1}, 
	\varpi^{-s} \cO 
	/ 
	\cO)\ar[r]\ar[d]&\ar[d] \Hom_{\cO}(T_x^\ast X_m, \varpi^{-s} \cO / \cO) 
	\times \Hom_{\cO}( 
	T_{y(\gamma)}^\ast Y_{m+1},\varpi^{-s} \cO / \cO)   \\ 
	 \Hom_{\cO}(T_{x(1)}^\ast X_{m+1}, \varpi^{-s} \cO / \cO)\ar[r] & 
	 \Hom_{\cO}(T_x^\ast X_m, \varpi^{-s} \cO / \cO) \times 
	 \Hom_{\cO}(T_{y(1)}^\ast Y_{m+1},\varpi^{-s} \cO / \cO)}}\]
	where the horizontal arrows are the ones already considered in 
	(\ref{eqn_cartesian}) (suppressing the subscripts indicating the 
	fibre product to save space), and the vertical ones are bijections arising 
	from the identification between $\Hom_{\cO}(T_{x(\gamma)}^\ast 
	X_{m+1},\varpi^{-s} \cO / \cO)$ and the fibre of $X_{m+1}(\cO / \varpi^s 
	\oplus \epsilon \varpi^{-s} \cO / \cO) \to X_{m+1}(\cO / \varpi^s)$ above 
	$x(\gamma) \text{ mod }\varpi^s = x(1) \text{ mod }\varpi^s$. The 
	horizontal arrows have kernels annihilated by $p^{2k_1}$. Our assumptions 
	imply that the elements of \[ \Hom_{\cO}(T_{x(\gamma)}^\ast X_{m+1}, 
	\varpi^{-s} \cO / \cO)\] and \[ \Hom_{\cO}(T_{x(1)}^\ast X_{m+1}, 
	\varpi^{-s} \cO / \cO)\]
	 corresponding to the images of $\rho'(\gamma)$ and $\rho'(1)$ in $\GL_n(\cO / \varpi^s \oplus \epsilon \varpi^{-s} \cO / \cO)$ are identified under the above bijection. This is what we needed to prove.
\end{proof}
\subsection{Pseudocharacters -- Galois deformation theory}

We again fix a prime $p$ and a finite extension $E / \bQ_p$ with ring of integers $\cO$ and residue field $k$. Let $\cC_\cO$ be the category of complete Noetherian local $\cO$-algebras with residue field $k$. 

Let $F$ be a number field, and let $S$ be a finite set of finite places of $F$, containing the $p$-adic ones. Let $\overline{\rho} : G_{F, S} \to \GL_n(k)$ be a continuous representation. Let $\overline{D}$ denote the associated group determinant over $k$. 

We write $\operatorname{Def}_{\overline{D}, S} : \cC_\cO \to \operatorname{Sets}$ for the functor which associates to each $A \in \cC_\cO$ the set of continuous group determinants $D$ of $G_{F, S}$ over $A$ such that $D \text{ mod }\ffrm_A = \overline{D}$. 
\begin{proposition}
	The functor $\operatorname{Def}_{\overline{D}, S}$ is represented by an object $R_{\overline{D}, S} \in \cC_\cO$.
\end{proposition}
\begin{proof}
	See \cite[\S 3.3]{chenevier_det}. (This reference deals with the case $\cO = W(k)$, but the extension to general coefficients is trivial.)
\end{proof}
\begin{lemma}\label{lem_upper_bound_on_ps_ring}
Fix an integer $q \geq 0$. Then there exists an integer $g_0 = g_0(S, \overline{D}, q)$ such that for any set $Q$ of finite places of $F$ such that $|Q| \leq q$, there exists a surjection $\cO \llbracket X_1, \dots, X_{g_0} \rrbracket \to R_{\overline{D}, S \cup Q}$.
\end{lemma}
\begin{proof}
Let $L / F$ denote the extension cut out by $\overline{\rho}$, and let $M_{S \cup Q}$ denote the maximal pro-$p$ extension of $L$ unramified outside $S \cup Q$. Then there exists $g_1 = g_1(S, \overline{\rho}, q)$ such that the group $\Gal(M_{S \cup Q} / F)$ can be topologically generated by $g_1$ elements (because the dimension of the space $H^1(G_{L, S \cup Q}, k)$ can be bounded in a way depending only on $q$). Moreover, any deformation of $\overline{D}$ to $G_{F, S \cup Q}$ in fact factors through $\Gal(M_{S \cup Q} / F)$ (by \cite[Lemma 3.8]{chenevier_det}). The statement of the lemma follows on applying e.g. the results of \cite[\S 2.37]{chenevier_det}.
\end{proof}
Now fix integers $a \leq b$, and let $\cE_{F, S}^{[a, b]}$ denote the category 
of finite cardinality $\Z_p[G_{F, S}]$-modules $M$ such that for each place $v 
| p$ of $F$, $M$ is isomorphic as $\bZ_p[G_{F_v}]$-module to a subquotient of lattice in a 
semistable representation of $G_{F_v}$ with Hodge--Tate weights in $[a, b]$. 
This defines a stable condition in the sense of \cite[Definition 
2.3.1]{WWE_pseudo_def_cond}.

We write $\operatorname{Def}_{\overline{D}, S}^{[a, b]} \subset 
\operatorname{Def}_{\overline{D}, S}$ for the subfunctor which assigns to $A \in \cC_\cO$ the set of determinants $D$ over $A$
satisfying the following condition:
\begin{equation}\label{eqn_Cayley} \begin{split} 
&\text{There exists a Cayley--Hamilton representation }(\cO[G_{F, S}], D) 
	\to (B, D')\\&\text{over }A\text{ such that for each }n \geq 1, B/ \ffrm_A^n B\text{ lies in }\cE_{F, S}^{[a, b]}, \text{when equipped}\\&  \text{with its left }G_{F, S}\text{-action}.\end{split}
\end{equation}
The notion of Cayley--Hamilton representation is defined in \cite[Definition 
2.1.8]{WWE_pseudo_def_cond}. We recall that it is an $\cO$-algebra homomorphism 
$\rho : \cO[G_{F, S}] \to B$, where $B$ is a finitely generated $A$-algebra and 
$D' : B \to A$ is a determinant such that $D' \circ \rho = D$ and the 
associated Cayley--Hamilton ideal $\operatorname{CH}(D') \subset B$ is zero. 
Note that in this situation $B$ is necessarily finitely generated as an 
$A$-module (\cite[Proposition 2.1.7]{WWE_pseudo_def_cond}).
\begin{proposition}
	The functor $\operatorname{Def}^{[a, b]}_{\overline{D}, S}$ is represented by an object $R^{[a, b]}_{\overline{D}, S} \in \cC_\cO$.
\end{proposition}
\begin{proof}
	See \cite[Theorem 2.5.5]{WWE_pseudo_def_cond}.
\end{proof}
Now suppose given a lift $\rho : G_{F, S} \to \GL_n(\cO)$ of $\overline{\rho}$ with the following properties:
\begin{itemize}
	\item $\rho \otimes_\cO E$ is absolutely irreducible.
	\item For each place $v | p$ of $F$, $\rho|_{G_{F_v}} \otimes_\cO E$ is semistable with Hodge--Tate weights in the interval $[a, b]$.
\end{itemize}
Let $D$ denote the associated group determinant over $\cO$. Then $D$ determines a 
homomorphism $R^{[a, b]}_{\overline{D}, S} \to \cO$. We write $\q$ for the 
kernel. Let $W = \ad \rho$, $W_E = W \otimes_\cO E$, $W_{E/\cO} = W_E / W$, $W_m = \ad \rho 
\otimes_\cO \cO 
/\varpi^m$. We have a Selmer group $H^1_{\cE_{F, 
		S}^{[a, b]}}(F, W_m)$ defined by local conditions: if $v \not\in S$, we 
		take 
the unramified subgroup of $H^1(F_v, W_m)$, if $v \in S - S_p$ we impose no 
condition and if $v \in S_p$ we take the subspace of $H^1(F_v, W_m)$ 
corresponding to self-extensions of $\rho|_{G_{F_v}}\otimes_{\cO}\cO/\varpi^m$ 
which are subquotients of 
lattices in semistable representations with Hodge--Tate weights in the interval 
$[a, b]$.
\begin{proposition}\label{prop_comparison_of_tangent_spaces}
	 There exists a canonical homomorphism
	\begin{equation}\label{eqn_reps_to_pseudochars} \tr_{m} : H^1_{\cE_{F, S}^{[a, b]}}(F, W_m) \to \Hom_\cO(\q / \q^2, \cO / \varpi^m). 
	\end{equation}
	Moreover, there is a constant $c \geq 1$ depending only on $\rho$ (and not 
	on $S$, $[a,b]$, or $m$) such that for any $m \geq 1$, the kernel and cokernel of 
	$\tr_m$ are both annihilated by $p^c$.
\end{proposition}
In applications of the proposition we will enlarge $S$ by adding Taylor--Wiles 
places. This is why it is important that the constant $c$ is independent of the 
set $S$.
\begin{proof}
	We first describe the map $\tr_m$. Let $A_m = \cO \oplus \epsilon \varpi^{-m}\cO / 
	\cO \subset A = \cO\oplus \epsilon E/\cO$. If $k \geq 0$, then we 
	write 
	$\alpha_k : A_m \to A_m$ for the $\cO$-algebra homomorphism that sends 
	$\epsilon$ to $p^k \epsilon$. A class $[\phi] \in H^1_{\cE_{F, S}^{[a, 
	b]}}(F, W_m)$ corresponds to an equivalence class of liftings 
	\[ \rho_\phi : G_{F, S} \to \GL_n(A_m) \]
	such that $\rho_\phi \text{ mod }(\epsilon) = \rho$ and for each $N \geq 
	1$, $\rho_\phi \text{ mod }\varpi^N \in \cE_{F, S}^{[a, b]}$ (this can be 
	seen by 
	considering the extension of scalars along the injective ring homomorphism 
	$A_m \hookrightarrow \cO\times \cO/\varpi^m[\epsilon]$ and using the fact 
	that $\cE_{F, S}^{[a, b]}$ is closed under taking 
	sub-$\Zp[G_{F,S}]$-modules).
	On the other hand, we can identify  $\Hom_\cO(\q / \q^2, \cO / \varpi^m)$ 
	with the pre-image  under the map \[  \Hom_\cO(R_{\overline{D}, S}^{[a, 
	b]}, A_m) \to \Hom_\cO(R_{\overline{D}, S}^{[a, b]}, \cO) \] of the 
	classifying map of $D$.	The map $\tr_m$ is the one which sends $[\phi]$ to 
	the classifying map of the pseudocharacter $\tr \rho_\phi$ over $A_m$. Note 
	that multiplication by $p^k$ on either side of 
	(\ref{eqn_reps_to_pseudochars}) corresponds at the level of representations 
	(resp. determinants) to composition with $\alpha_k$.
	
	Next we analyze the kernel of $\tr_m$. Suppose that $\tr \rho_\phi = \tr \rho$. Let $k_0$ be the constant of Proposition \ref{prop_effectivity_of_pseudo_characters}: it depends only on $\rho(G_{F, S}) \subset \GL_n(\cO)$. Then we can find $X \in  M_n(E / \cO)$ such that
	\[ (1 + \epsilon X) \rho_{p^{k_0} \phi} (1 - \epsilon X) = \rho, \]
	or equivalently such that $p^{k_0} \phi$ becomes a coboundary in 
	$H^1(F, W_{E/\cO})$. Since the kernel of $H^1(F, W_m) \to H^1(F, 
	W_{E/\cO})$ is isomorphic to $H^0(F,W_{E/\cO})\otimes \cO/\varpi^m$, it is 
	killed by a uniformly bounded power of $p$ and we see that the same 
	is true for the kernel of $\tr_m$.
	
	Now we analyze the cokernel of $\tr_m$. This is more subtle. Let $D'$ be a 
	determinant of $G_{F, S}$ over $A_m$ corresponding to an element of the 
	right-hand side of (\ref{eqn_reps_to_pseudochars}). By assumption, there 
	exists a Cayley--Hamilton representation $r : A_m[G_{F, S}] \to B$ and a 
	determinant $D'' : B \to A_m$ such that $D' = D'' \circ r$, and the finite 
	quotients of $B$ lie in $\cE_{F, S}^{[a, b]}$. We may assume without loss 
	of generality that $r$ is surjective. According to the characterization of 
	$\ker(D') \subset A_m[G_{F, S}]$ given by \cite[Lemma 1.19]{chenevier_det}, 
	we have $\ker(r) \subset \ker(D')$.
	
	By Proposition \ref{prop_effectivity_of_pseudo_characters},  there exists a 
	homomorphism $\rho_\phi : G_{F, S} \to \GL_n(A)$ such that the associated 
	group determinant is $\alpha_{k_0} \circ D'$. It also follows from 
	Proposition 
	\ref{prop_effectivity_of_pseudo_characters} that the associated cohomology 
	class $[\phi] \in H^1(F,W_{E/\cO})$ is killed by multiplication by 
	$\varpi^mp^{k_0}$ (as $\tr(\rho_{\varpi^m \phi}) = \tr(\rho)$).  
	We deduce that $p^{k_0}[\phi]$ is contained in the image of $H^1(F,W_m)$ in 
	$H^1(F,W_{E/\cO})$. So we may in fact assume that we have $\rho_\phi: G_{F, 
	S} \to 
	\GL_n(A_m)$ such that the associated 
	group determinant is $\alpha_{2k_0} \circ D'$.
	
	We must show that there is $c 
	\geq 0$ depending only on $\rho$ such that for each $N \geq 1$, $\alpha_{c} 
	\circ \rho_\phi \text{ mod }\varpi^N$ defines an object of $\cE_{F, S}^{[a, 
	b]}$ (as then $p^{c} [\phi]$ is a pre-image under 
	(\ref{eqn_reps_to_pseudochars}) of $\alpha_{2k_0 + c}\circ D'$).
	
	Let $\mathcal{A}_\phi = \rho_\phi(A_m[G_{F, S}]) \subset M_n(A_m)$. Let 
	$k_1 \geq 0$ be an integer such that $p^{k_1} M_n(\cO) \subset 
	\rho(\cO[G_{F, S}])$ (this exists since $\rho\otimes_\cO E$ is absolutely 
	irreducible, 
	by assumption, and hence $\rho(E[G_{F,S}]) = M_n(E)$). Then 
	$\mathcal{A}_{p^{k_1} \phi}$  contains 
	$p^{k_1} M_n(A_m)$. Indeed, let $E_{ij}$ denote the elementary matrix 
	in $M_n(\cO)$ with entries $1$ in the $(i, j)$ spot and 0 elsewhere. Then 
	$p^{k_1} E_{ij} \in \rho(\cO[G_{F, S}])$, so there is $X_{ij} \in 
	M_n(\cO / \varpi^m)$ such that $p^{k_1} E_{ij} + \epsilon X_{ij} \in 
	\mathcal{A}_\phi$ and $p^{k_1} E_{ij} + \epsilon p^{k_1} 
	X_{ij} \in \mathcal{A}_{p^{k_1} \phi}$. After multiplying by 
	$\epsilon$, we see that $p^{k_1} \epsilon E_{ij} \in 
	\mathcal{A}_{p^{k_1}\phi}$ for all $(i, j)$, hence $p^{k_1} 
	E_{ij} \in \mathcal{A}_{p^{k_1} \phi}$.
	
	Let $D''' : \mathcal{A}_{p^{k_1} \phi} \to A_m$ denote the 
	determinant induced by the natural inclusion $\mathcal{A}_{p^{k_1} \phi} \to M_n(A_m)$. If $x \in \ker D'''$, then \cite[Lemma 1.19]{chenevier_det} 
	shows that $\tr(x p^{k_1} E_{ij}) = 0$ for all $i, j$. (Here $\tr : M_n(A_m) \to A_m$ is the usual trace of an $n \times n$ matrix, not a pseudocharacter.)
	Hence $\ker D'''$ is contained in the intersection of 
	$\mathcal{A}_{p^{k_1} \phi}$ with $M_n(A_m)[p^{k_1}]$, and is therefore
	annihilated by the homomorphism $\alpha_{k_1}: M_n(A_m) \to M_n(A_{m})$.
	
	It follows that there exists a commutative diagram of $A_m$-algebras
	\[ \xymatrix{ A_m[G_{F, S}] \ar[r]^{\rho_{p^{k_1}\phi}}\ar[d]_r& 
	\mathcal{A}_{p^{k_1} \phi} 
	\ar[d]\ar[r] & M_n(A_m) \ar[d]^{\alpha_{k_1}} \\ B \ar[r] & A_m[G_{F, S}] / 
	\ker(\alpha_{2k_0 + k_1} \circ D') \ar[r] & M_n(A_{m}). }\]
	Indeed, the quotient map $A_m[G_{F, S}] \to A_m[G_{F, S}] / 
	\ker(\alpha_{2k_0 + k_1} \circ D')$ factors through $B$ because $\ker(r) 
	\subset \ker(D') \subset \ker(\alpha_{2k_0 + k_1} \circ D')$, and bottom right arrow exists because $\alpha_{2k_0 + k_1} \circ D' = D''' \circ \rho_{p^{k_1} \phi}$. The existence of this diagram shows that the  finite 
	quotients of $M_n(A_{m})$, with 
	induced action of $A_m[G_{F, S}]$ by left multiplication via 
	$\rho_{p^{2 k_1} \phi}$, satisfy condition $\cE_{F, S}^{[a, b]}$ 
	(since this holds for finite quotients of $B$). 
	This implies that the finite quotients of $\alpha_{2k_1}\circ\rho_\phi$ 
	also satisfy the condition $\cE_{F, S}^{[a, b]}$. We deduce that the cokernel of $\tr_m$ is annihilated by $p^{2k_0 + 2k_1}$.
\end{proof}

\subsection{Pseudocharacters -- Taylor--Wiles data}\label{sec_unitary_pseudocharacters}

We continue our discussion of the Galois deformation theory of pseudocharacters, now focusing on what happens when we impose conjugate self-duality conditions and allow additional primes of ramification. We thus fix the following notation:
\begin{itemize}
	\item $p$ is a prime and $E / \bQ_p$ is a coefficient field.
	\item $F / F^+$ is a CM quadratic extension of a totally real field.
	\item $S$ is a finite set of finite places of $F^+$ containing the $p$-adic ones. We assume that each place of $S$ splits in $F$, and fix for each $v \in S$ a choice of place $\wv$ of $F$ lying above $v$. We set $S_p = \{ v \mid v | p \}$ and $\widetilde{S} = \{ \wv \mid v \in S \}$.
	\item $r : G_{F^+, S} \to \cG_n(\cO)$ is a continuous representation such that $r|_{G_{F, S}} \otimes_\cO E$ is absolutely irreducible. We set $\rho = r|_{G_{F, S}} : G_{F, S} \to \GL_n(\cO)$ and write $D$ for the group determinant of $\rho$ and $\overline{D}$ for its reduction modulo $\varpi$. We set $\chi = \nu \circ r$. The existence of $r$ implies that the $\cO[G_{F, S}]$-module structure on $W = \ad \rho$ extends to an $\cO[G_{F^+, S}]$-module structure (set $W = \ad r$, where $\ad r$ is defined as in \cite[\S 2.1]{cht}), and similarly for $W_E$, $W_m$, etc.
	\item $a \leq b$ are integers such that $D$ defines a homomorphism $R_{\overline{D}, S}^{[a, b]} \to \cO$. Thus for each $v \in S_p$, $\rho|_{G_{F_\wv}} \otimes_\cO E$ is semistable with all Hodge--Tate weights in the range $[a, b]$. Note that $\chi|_{G_{F_{\wv}}}$ is then semistable and there exists $w \in \bZ$ such that $\chi \epsilon^{w}$ has finite order. We assume moreover that $a + b = w$.
\end{itemize}
Let $R_S$ denote the quotient of $R_{\overline{D}, S}^{[a, b]}$ corresponding to pseudocharacters $D'$ such that $(D')^c = (D')^\vee \otimes \chi|_{G_{F, S}}$. Then $\rho$ determines a homomorphism $R_S \to \cO$, and we write $\q_S$ for its kernel. 

We 
define Selmer conditions $\cL_S =  \{\cL_v \} = \{ \cL_{v, m} \}$ for $W_m$ as 
follows: if $v \not\in S$, then $\cL_v$ is the unramified subgroup of 
$H^1(F^+_v, W_m)$. If $v \in S - S_p$, then $\cL_v = H^1(F_v^+, W_m)$. If $v 
\in S_p$, then $\cL_v$ is the subspace of $H^1(F_v^+, W_m)$ corresponding to 
self-extensions of $\rho|_{G_{F_\wv}}\otimes_{\cO}\cO/\varpi^m$ which are 
subquotients of lattices in 
semistable representations with Hodge--Tate weights in the interval $[a, b]$. 
We define dual 
Selmer conditions $\cL^\perp = \{\cL^\perp_v\}$ to be given by the annihilators 
of $\cL_v$ under local duality. 

The corresponding 
Selmer groups are defined by \begin{align*}H^1_{\cL_S}(F^+,W_m) &= 
\ker\left(H^1(F^+,W_m) 
\to \prod_v H^1(F^+_v,W_m)/\cL_{v,m}  \right)\\ H^1_{\cL_S^\perp}(F^+,W_m(1)) 
&= \ker\left(H^1(F^+,W_m(1)) 
\to \prod_v H^1(F^+_v,W_m(1))/\cL^\perp_{v,m}  \right).\end{align*} 

These Selmer groups are finite length $\cO$-modules. We denote their length by 
$h^1_{\cL_S}(F^+, W_m), h^1_{\cL_S^\perp}(F^+,W_m(1))$. Taking inverse limits 
with respect to the projection maps $W_{m+1}\to W_m$ and direct limits with 
respect to the injections $W_m \cong \varpi W_{m+1} \subset W_{m+1}$, we also 
define Selmer 
groups with characteristic $0$ and divisible coefficients:
\begin{align*}H^1_{\cL_S}(F^+,W_E) &= 
\left(\varprojlim_m H^1_{\cL_S}(F^+, W_m)\right)\otimes_{\cO}E\\ 
H^1_{\cL_S}(F^+,W_{E/\cO}) &= 
\varinjlim_m H^1_{\cL_S}(F^+, W_m).
\end{align*} 

 Proposition \ref{prop_comparison_of_tangent_spaces} has the following 
 consequence:
\begin{proposition}\label{prop_comparison_of_csd_tangent_spaces}
	For each $m \geq 1$, there is a canonical homomorphism
	\begin{equation}\label{eqn_comparison_of_tangent_spaces_with_duality} \tr_{m, S} : H^1_\mathcal{L_S}(F^+, W_m) \to \Hom_\cO(\q_S / \q_S^2, \cO / \varpi^m). 
	\end{equation}
	Moreover, there is a constant $d \geq 0$ depending only on $r$ (and not on 
	$S$, $[a, b]$ or $m$) such that $p^d$ annihilates the kernel and cokernel of 
	$\tr_{m, S}$.
\end{proposition}
\begin{proof}
	The map (\ref{eqn_reps_to_pseudochars}) is $\Gal(F/F^+)$-equivariant for 
	the action on the left-hand side induced by the $G_{F^+,S}$ action on $W_m$ 
	and the action on the right-hand side defined as follows: the non-trivial 
	element $c \in \Gal(F / F^+)$ acts on 
	$R_{\overline{D}, S}^{[a, b]}$ by sending a pseudocharacter $D'$ to 
	$(D')^{c, \vee} \otimes \chi|_{G_{F, S}}$, and this induces an action on 
	the right-hand side of (\ref{eqn_reps_to_pseudochars}). Note that this 
	action makes 
	sense because of our condition $a + b = w$. The right-hand side of 
	(\ref{eqn_comparison_of_tangent_spaces_with_duality}) is the $c$-invariants 
	in the right-hand side of (\ref{eqn_reps_to_pseudochars}).  The left-hand 
	side of	(\ref{eqn_comparison_of_tangent_spaces_with_duality}) maps to the 
	$c$-invariants in the left-hand side of 
	(\ref{eqn_reps_to_pseudochars}) with bounded kernel and cokernel. This is 
	enough.
\end{proof}

Here is a variant that will be useful when it comes to deduce our main 
vanishing result.
\begin{prop}\phantomsection\label{prop_generic_tangent_spaces}\begin{enumerate}
	\item There is an isomorphism $\tr_{E, S} : H^1_{\cL_S}(F^+, W_E) \to 
	\Hom_\cO(\q_S / \q_S^2, E)$.

\item The natural map $H^1_{\cL_S}(F^+,W_E) \to H^1(F_S/F^+,W_E)$ identifies 
$H^1_{\cL_S}(F^+,W_E)$ with the geometric Selmer group 
$H^1_{g,S}(F^+,W_E)$
\item Suppose that for each $v \in S$, $\rho|_{G_{F_\wv}}$ is generic. Then 
$H^1_{g,S}(F^+,W_E) = H^1_f(F^+,W_E)$.
\end{enumerate}
\end{prop}
\begin{proof}
The first part follows from Proposition 
\ref{prop_comparison_of_csd_tangent_spaces} by taking the inverse limit over 
$m$ and inverting $p$. The main result of \cite{Tong} implies that 
$H^1_{\cL_S}(F^+,W_E)$ classifies (polarized) semistable self-extensions of 
$\rho_E$. A de 
Rham self-extension of $\rho_E$ is automatically semistable \cite[Corollary 
1.27]{nekovar}, so the second part follows. The third part follows from the 
equality of the respective local Selmer groups in the generic case.
\end{proof}
For $m' \ge m$ the inverse 
image of $\cL_{v,m'}$ in $H^1(F_v^+,W_m)$ under the map
\[ H^1(F_v^+,W_m) \to H^1(F_v^+,W_{m'}) \]
induced by the injection $W_m \to W_{m'}$
 equals $\cL_{v,m}$. Indeed, 
the natural map $H^1(F_v^+,W_m) \to  H^1(F_v^+,W_{m'})$ corresponds to pushing 
forward a $\GL_n(A_m)$-valued lifting to $\GL_n(A_{m'})$, which preserves semistability (cf. the argument of \cite[Proposition 1.1]{Ram93} -- here are are writing $A_m = \cO \oplus \epsilon \varpi^{-m}\cO / 
	\cO$, as in the proof of Proposition \ref{prop_comparison_of_tangent_spaces}). We record a consequence 
of this in the following lemma.
\begin{lem}\label{lem:selmertorsion}
	The natural map $H^1_{\cL_S}(F^+,W_m) \to 
	H^1_{\cL_S}(F^+,W_{E/\cO})[\varpi^m]$ is surjective.
\end{lem} 
\begin{proof}We consider the commutative diagram with exact rows:
	\begin{equation*} 
\scalebox{0.8}{
	\xymatrix{
		0 \ar[r] &  
		H^1_{\cL_S}(F^+, W_{m})\ar[r]\ar[d] & H^1(F_S / F^+, W_{m})\ar[r]\ar[d] 
		& \bigoplus_{v \in S_p}H^1(F_v^+,W_m)/\cL_{v,m} \ar[d]\\
		0 \ar[r] & 	H^1_{\cL_S}(F^+, W_{E/\cO})[\varpi^m]\ar[r]& 
		H^1(F_S / F^+, W_{E/\cO})[\varpi^m]\ar[r] & 
		\bigoplus_{v \in S_p}\varinjlim_{m'}H^1(F_v^+,W_{m'})/\cL_{v,m'} }}
	\end{equation*}
	The central vertical map is surjective. The right vertical map is injective 
	(by the observation preceding this lemma). So the left vertical map is 
	surjective.
\end{proof}

If $Q$ is a set of finite places of $F^+$ and $N$ is a positive integer, we 
say that $Q$ is \emph{a set of Taylor--Wiles places of level $N$} 
(relative to $r$, $S$) if it satisfies the following conditions:
\begin{itemize}
	\item $Q \cap S = \emptyset$.
	\item For each $v \in Q$, $v = w w^c$ splits in $F$; and $\rho(\Frob_w)$ has $n$ distinct eigenvalues $\alpha_{w, 1}, \dots, \alpha_{w, n} \in \cO$.
	\item For each $v \in Q$, $q_v \equiv 1 \text{ mod }p^N$.
\end{itemize}
A \emph{Taylor--Wiles datum of level $N \geq 1$} is a tuple $(Q, \widetilde{Q}, 
(\alpha_{\wv, 1}, \dots, \alpha_{\wv, n})_{\wv \in \widetilde{Q}})$, where $Q$ 
is a set of Taylor--Wiles places of level $N$, $\widetilde{Q}$ is a set 
consisting of a choice, for each $v \in Q$, of a place $\wv$ of $F$ lying above 
$v$, and $(\alpha_{\wv, 1}, \dots, \alpha_{\wv, n})$ is a choice of ordering of 
the eigenvalues of $\rho(\Frob_{\wv})$. 
\begin{lemma}\label{lem_greenberg_wiles_formula}
	Suppose that the following conditions are satisfied:
	\begin{enumerate}
		\item For each $v \in S$, $\rho|_{G_{F_\wv}}$ is generic.
		\item For each place $v | \infty$, $\chi(c_v) = -1$.
	\end{enumerate}
	 Then there exists $d \geq 0$ with the following property: for every $N \ge 
	 1$, every Taylor--Wiles datum $(Q, \widetilde{Q}, (\alpha_{\wv, 1}, \dots, 
	 \alpha_{\wv, n})_{\wv \in \widetilde{Q}})$ of level $N$, and every $1 \leq 
	 m \leq N$, we have:
	\[ h^1_{\cL_{S \cup Q}}(F^+, W_m) \leq d + h^1_{\cL_{S \cup Q}^\perp}(F^+, 
	W_m(1)) + m n |Q| + \sum_{v \in Q} \sum_{i \neq j} \ord_{\varpi} 
	(\alpha_{\wv, i} - \alpha_{\wv, j}). \]
\end{lemma}
\begin{proof}
	Fix a Taylor--Wiles datum. By the usual Greenberg--Wiles formula, we have
	\begin{multline*}
	h^1_{\cL_{S \cup Q}}(F^+, W_m) = h^1_{\cL_{S \cup Q}^\perp}(F^+, W_m(1)) + h^0(F^+, W_m) - h^0(F^+, W_m(1)) \\
+ \sum_{v \in S \cup Q} (\ell_{v, m} - h^0(F_\wv, W_m)) - \sum_{v | \infty} l( (1 + c_v) W_m),
	\end{multline*}
	where $l_{v, m} = l(\cL_{v, m})$ and $l$ denotes the length of an $\cO$-module. The contribution from the infinite places 
	is $m [F^+ : \bQ] n(n-1) / 2$, up to 
	a uniformly bounded error. The global terms $h^0(F^+, W_m)$ and $h^0(F^+, 
	W_m(1))$ are both uniformly bounded, the first since $\rho$ is absolutely 
	irreducible and the second since $\rho$ is absolutely irreducible and $\rho, \rho(1)$ have different sets of Hodge--Tate weights.

	If $v \in Q$, then $(\ell_{v, m} - h^0(F_\wv, W_m)) = h^0(F_\wv, W_m(1)) = 
	h^0(F_\wv, W_m)$, since we are assuming $m \leq N$, and this is bounded 
	above by $nm + \sum_{i \neq j} \ord_\varpi (\alpha_{\wv, i} - \alpha_{\wv, 
	j})$. If $v \in S - S_p$, then $(\ell_{v, m} - h^0(F_\wv, W_m))$ is 
	uniformly bounded, by \cite[Proposition 1.2.2]{All16}.
	
	Finally, suppose that $v \in S_p$. Let $R_v^{\square, [a, b]} \in 
	\mathcal{C}_\cO$ denote the object representing the functor of lifts of 
	$\overline{\rho}|_{G_{F_\wv}}$ whose projections to Artinian quotients are 
	subquotients of lattices in semistable representations with all Hodge--Tate 
	weights in the interval $[a, b]$. The representation $\rho|_{G_{F_\wv}}$ 
	determines a homomorphism $R_v^{\square, [a, b]} \to \cO$. If $\q_v$ 
	denotes its kernel, then by definition we have $l_{v, m} - h^0(F_\wv, W_m) = l( \q_v / 
	\q_v^2 \otimes_\cO \cO / \varpi^m) - m n^2 $. We wish to show that $l_{v, m}  - h^0(F_\wv, W_m) - m [F_v^+ : \bQ_p] 
	n(n-1)/2$ is bounded independently of $m$. This in turn will follow if we 
	can show that $\q_v / \q_v^2 \otimes_\cO E$ has dimension $n^2 + [F_v^+ : \bQ_p] 
	n(n-1)/2$ as $E$-vector space.
	
	However, the argument of \cite[Proposition 2.3.5]{kis04}, together with \cite[Conjecture 1.0.1]{Tong} (stated as a conjecture but proved in that paper), shows that the completed local ring of $R_v^{\square, [a, b]}$ at $\q_v$ represents the functor $\cC_E \to \operatorname{Sets}$ of lifts of $\rho|_{G_{F_\wv}} \otimes_\cO E$ whose Artinian quotients are semistable with all Hodge--Tate weights in the interval $[a, b]$.
	The tangent space to this functor (which is equal to $\q_v / \q_v^2 
	\otimes_\cO E$) is computed in the proof of \cite[Theorem 
	1.2.7]{All16}, which gives the desired result.
\end{proof}

\begin{lem}\label{lem:quotientlength}
	Suppose $M$ is a finitely generated $\cO$-module and let $N\ge 1$ and $d, g 
	\ge 
	0$ be integers. Suppose we know that for all $m \le N$ we have 
	\[\ell(M/\varpi^m) \le gm + d.\] Then there is a map $\cO^g \to M/\varpi^N$ 
	with cokernel of length $\le d$. 
\end{lem}
\begin{proof}
	We prove the lemma by induction on the number of generators of $M$. The 
	lemma is obvious if $M$ is cyclic. For a general $M$, we can first replace 
	$M$ by $M/\varpi^N$ without changing anything. Now let $M' = M/C$ where $C$ 
	is a cyclic submodule of $M$ of maximal length and let $N' \le N$ be the 
	maximal length of a cyclic submodule of $M'$. It suffices to prove that 
	there is a map $\cO^{g-1} \to M'/\varpi^{N'} = M'/\varpi^N$ with cokernel 
	of length $\le d$. For all $m \le N'$ we have $\ell(M'/\varpi^m) = 
	\ell(M/\varpi^m) - m \le (g-1)m + d$. By induction, we have the desired map 
	$\cO^{g-1} \to M'/\varpi^{N'}$. 
\end{proof}

\begin{cor}\label{cor_if_we_kill_dual_Selmer} Suppose that $\rho$ satisfies the 
hypotheses of Lemma \ref{lem_greenberg_wiles_formula}. Then there exists $d \in 
\NN$ such that for all $N \in \NN$ and every Taylor--Wiles datum of level $N$,
	there is a map \[\cO^{n |Q|} \to H^1_{\cL_{S \cup Q}}(F^+,W_N)\] with 
	cokernel 
	of length $\le d + h^1_{\cL_{S \cup Q}^\perp}(F^+,W_N(1)) + 
	\sum_{v \in Q}\sum_{i \ne 
		j}\ord_{\varpi}(\alpha_{\wv,i} - \alpha_{\wv,j})$.
\end{cor}
\begin{proof}
Using Lemma \ref{lem_greenberg_wiles_formula} and Lemma \ref{lem:quotientlength}, we see that it is enough to find $d_0, d_1 \in \NN$ such that for any $1 \leq m \leq N$ and any Taylor--Wiles datum of level $N$, we have
\begin{equation}\label{eqn:modselmer}  l(H^1_{\cL_{S \cup Q}}(F^+, W_N) / (\varpi^m)) \leq h^1_{\cL_{S \cup Q}}(F^+, W_m) + d_0
\end{equation}
and
\begin{equation}\label{eqn:moddualselmer} h^1_{\cL^\perp_{S \cup Q}}(F^+, W_m(1)) \leq h^1_{\cL^\perp_{S \cup Q}}(F^+, W_N(1)) + d_1. 
\end{equation}
We treat these in turn. For the first inequality, we note that Lemma \ref{lem:selmertorsion} shows that the map
\[ H^1_{\cL_{S \cup Q}}(F^+, W_m) \to H^1_{\cL_{S \cup Q}}(F^+, W_{E / \cO})[\varpi^m] \]
is surjective, with kernel a subquotient of $H^0(F^+, W_{E / \cO})$. It follows that there is a surjective homomorphism
\[ H^1_{\cL_{S \cup Q}}(F^+, W_N) / (\varpi^m) \to H^1_{\cL_{S \cup Q}}(F^+, W_{E / \cO})[\varpi^N] / (\varpi^m) \]
with kernel a subquotient of $H^0(F^+, W_{E / \cO})$. Since we have
\[ l(H^1_{\cL_{S \cup Q}}(F^+, W_{E / \cO})[\varpi^N] / (\varpi^m)) = l(H^1_{\cL_{S \cup Q}}(F^+, W_{E / \cO})[\varpi^m]), \]
we see that (\ref{eqn:modselmer}) holds with $d_0 = h^0(F^+, W_{E / \cO})$.

For the second inequality, we note that the kernel of the natural map
\begin{equation}\label{eqn_dual_m_to_N} H^1(F_{S} / F^+, W_m(1)) \to H^1(F_S / F^+, W_N(1)) \end{equation}
is contained in the kernel of the map
\[ H^1(F_S / F^+, W_m(1)) \to H^1(F_S / F^+, W_{E / \cO}(1)), \]
which is a subquotient of $H^0(F^+, W_{E/ \cO}(1))$ (which is finite, by the 
same argument showing boundedness of $h^0(F^+,W_m(1))$ in the proof of Lemma 
\ref{lem_greenberg_wiles_formula}). We see that (\ref{eqn:moddualselmer}) will 
hold with $d_1 = h^0(F^+, W_{E / \cO}(1))$ provided that the map 
(\ref{eqn_dual_m_to_N}) sends $H^1_{\cL_{S \cup Q}^\perp}(F^+, W_m(1))$ into 
$H^1_{\cL_{S \cup Q}^\perp}(F^+, W_N(1))$. Recalling the definition of our 
local conditions, this means we must show that if $v \in S_p$, then the map 
$H^1(F^+_v, W_m(1)) \to H^1(F^+_v, W_N(1))$ sends $\cL_{v, m}^\perp$ to 
$\cL_{v, N}^\perp$. By duality, we must show that if $v \in S_p$, then the map 
$H^1(F^+_v, W_N) \to H^1(F^+_v, W_m)$ induced by the surjection $W_N \to W_m$ 
sends $\cL_{v, N}$ into $\cL_{v, m}$. However, this follows immediately from 
the definitions.
\end{proof}
For Corollary \ref{cor_if_we_kill_dual_Selmer} to be useful, we need to be able to find Taylor--Wiles data with good properties. To do this, we first introduce a useful definition.
\begin{lemma}\label{lem_non-zero_cohomology}
	Let $H \subset \GL_n(\cO)$ be a compact subgroup, and suppose the 
	characteristic polynomial of every element of $H$ splits over $E$ (this 
	condition is always satisfied after possibly enlarging $E$). Denote the 
	associated representation of $H$ by $\rho$, so we have $\cO[H]$-modules $W = \ad 
	\rho$, $W_E$, $W_{E/\cO}$ as above. Then the 
	following conditions are equivalent:
	\begin{enumerate}
		\item For all simple $E[H]$-submodules $V \subset W^0_E = \ad^0\rho 
		\otimes E$, we can find $h\in 
		H$ with $n$ distinct eigenvalues and $\alpha \in E$	such that $\alpha$ is 
		an eigenvalue of $h$ and $\tr e_{h,\alpha} V \ne 0$ (where 
		$e_{h,\alpha} \in W_E$ denotes the $h$-equivariant projection to 
		the $\alpha$-eigenspace).
		\item For all simple $E[H]$-submodules $V \subset W_E$, we can find $h\in 
		H$ with $n$ distinct eigenvalues and $\alpha \in E$	such that $\alpha$ is 
		an eigenvalue of $h$ and $\tr e_{h,\alpha} V \ne 0$.
		\item For all non-zero $E[H]$-submodules $V \subset W_E$, there exists $h \in H$ with $n$ distinct eigenvalues such that $V \not\subset (h - 1)W_E$.
		\item For all non-zero divisible $\cO[H]$-submodules $V \subset W_{E / 
		\cO}$, there exists $h \in H$ with $n$ distinct eigenvalues such that 
		$V \not\subset (h - 1)W_{E / \cO}$.
	\end{enumerate}
\end{lemma}
\begin{proof}
	We note that (1) and (2) are equivalent because the scalar matrices $Z_E \subset W_E$ give a complement to $W_E^0$ in $W_E$, and the condition $\tr e_{h, \alpha} Z_E \neq 0$ is satisfied for any regular semisimple element $h \in H$ and eigenvalue $\alpha \in E$.
	
	If $h \in \GL_n(\cO)$ has $n$ distinct eigenvalues, then it acts semisimply 
	on $W_E$. In particular, there is a unique $h$-invariant direct sum 
	decomposition $W_E = W_E^h \oplus (h-1)W_E$. If $V \subset W_E$ is a 
	$h$-invariant subspace, then there is a similar direct sum decomposition $V 
	= V^h \oplus (h-1)V$. The condition that there exists an eigenvalue $\alpha 
	\in E$ of $h$ such that $\tr e_{h, \alpha} V \ne 0$ is equivalent to the 
	condition that the projection of $V$ to $W_E^h$ is non-zero, or 
	equivalently that $V^h \neq 0$. This is in turn equivalent to the condition 
	that $V \not\subset (h-1) W_E$. This shows that (2) and (3) are equivalent. 
	
	Now we show that (3) and (4) are equivalent. For this we note that there is a $\GL_n(\cO)$-equivariant, inclusion-preserving bijection between the $E$-subspaces of $W_E$ and the divisible $\cO$-submodules of $W_{E / \cO}$; this sends $V \subset W_E$ to $V + W/ W$ and $V' \subset W_{E / \cO}$ to
	\[ V = \{ v \in W_E \mid \forall n \geq 0, \varpi^{-n} v \text{ mod }W 
	\in V'\}. \]
	The proof in this case is complete on noting that $(h-1)W_E$ corresponds to $(h-1)W_{E / \cO}$ under this bijection.
\end{proof}
\begin{defn}\label{dfn:enormous}
	We say that a subgroup $H \subset \GL_n(\cO)$ is enormous if for all simple $E[H]$-submodules $V \subset W_E$, we can find $h\in 
	H$ with $n$ distinct eigenvalues in $E$ 
	and $\alpha \in E$	such that $\alpha$ is 
	an eigenvalue of $h$ and $\tr e_{h,\alpha} V \ne 0$.
\end{defn}
\begin{remark}
The above is a natural analogue of the definition of enormous subgroups in 
positive characteristic \cite[Definition 4.10]{KT}. In contrast to the positive 
characteristic case, we do not need to assume any vanishing of cohomology 
groups for $H$. The necessary vanishing will follow from the purity of our 
Galois representations (see \cite[Lemma 6.2]{Kis04a}, this goes back to Serre 
for Tate modules of abelian varieties \cite{Ser71}).
\end{remark}
\begin{lemma}
	Let $H \subset \GL_n(\cO)$ be an enormous subgroup. Then $H$ acts 
	absolutely irreducibly on $E^n$, and in particular we have $H^0(H,W^0_E) = 
	0$.
\end{lemma}
\begin{proof}
We need to show that the $E$-linear span of $H$ in $W_E$ is 
everything (by Burnside's theorem on matrix algebras). Consider \[U = \{u \in 
W_E : \tr(hu) = 0\quad \forall h \in H\}.\] If $U$ is non-zero, let $V$ be a 
simple $E[H]$-submodule of $U$ so we have an $h \in H$ with $\alpha \in E$ 
such that $\tr e_{h,\alpha}V \ne 0$. This is a contradiction, since 
$e_{h,\alpha}$ is a polynomial in $h$. We conclude that $U = 0$ and since the 
trace pairing on $W_E$ is perfect the span of $H$ is indeed $W_E$. (Compare 
with \cite[Appendix, Lemma 1]{jack}.)
\end{proof}

\begin{lem}\label{lem:findTWprimes}
	Let $q \geq \mathrm{corank}_\cO H^1(F_S / F^+,W_{E / \cO}(1))$, and suppose 
	that $\rho$ satisfies the following conditions:
	\begin{enumerate}
		\item For all but finitely many finite places $v \nmid S$ of $F$, the eigenvalues of $\rho(\Frob_v)$ are algebraic numbers which have absolute value $q_v^{w/2}$ with respect to any complex embedding.
		\item $\rho(G_{F(\zeta_{p^\infty})})$ is enormous.
	\end{enumerate}
	Then there exists $d \in \NN$ such that for any $N \in \NN$ we can find a 
	Taylor--Wiles datum $(Q, \widetilde{Q}, (\alpha_{\wv, 1}, \dots, 
	\alpha_{\wv, n})_{\wv \in \widetilde{Q}})$ of level $N$ with $|Q| = q$ such
	that
	\begin{enumerate}
		\item for all $v \in Q$ and $i \ne j$ we have 
		$\ord_{\varpi}(\alpha_{\wv,i}-\alpha_{\wv,j}) \le d$
		\item $h^1_{\cL_{S \cup Q}^\perp}(F^+,W_N(1)) \le d$.
	\end{enumerate}
\end{lem}
\begin{proof}
	If $Q$ is a set of Taylor--Wiles places then we have an exact sequence \[ 0 
	\to 
	H^1_{\cL_{S \cup Q}^\perp}(F^+,W_N(1)) \to  H^1_{\cL_S^\perp}(F^+, W_N(1)) 
	\to \bigoplus_{v \in Q}H^1(k(v), W_N(1)).\]
	Suppose we could find $\sigma_1,\ldots,\sigma_q \in 
	G_{F(\zeta_{p^\infty})}$ such 
	that
	\begin{enumerate}
		\item[(a)] for each $i = 1, \dots, q$, $\rho(\sigma_i)$ has $n$ 
		distinct eigenvalues in $E$;
		\item[(b)] the kernel of the map \[H^1(F_S / F^+, W 
		_{E/\cO}(1)) \to \bigoplus_{i=1}^q 
		H^1(\widehat{\bZ}, W_{E/\cO}(1)) \cong \bigoplus_{i=1}^q W_{E 
		/ \cO}(1) / (\sigma_i - 1) W_{E / \cO}(1)
		\] (product of restriction maps associated to the homomorphisms $\widehat{\bZ} \to G_{F^+, S}$, the $i^\text{th}$ such homomorphism sending $1$ to $\sigma_i$) is a finite length $\cO$-module.
	\end{enumerate}
	Then consideration of the following diagram:
	\begin{equation*} 
	\xymatrixcolsep{.1in}
\scalebox{0.8}{\xymatrix{
		0 \ar[r] & H^0(F^+, W_{E/\cO}(1))/\varpi^N\ar[r]\ar[d] & 
		H^1(F_S / F^+, W_N(1))\ar[r]\ar[d] & H^1(F_S / F^+, W_{
			E/\cO}(1))[\varpi^N]\ar[r]\ar[d] & 0 \\
		0 \ar[r] & \bigoplus_{i=1}^q H^0(\langle \sigma_i \rangle, 
		W_{E/\cO}(1))/\varpi^N\ar[r] & 
		\bigoplus_{i=1}^q H^1(\langle \sigma_i \rangle, W_N(1))\ar[r] & 
		\bigoplus_{i=1}^q H^1(\langle \sigma_i \rangle, 
		W_{E/\cO}(1))[\varpi^N]\ar[r] & 0 }}
	\end{equation*}
	shows that the kernel of the map 
	\[ H^1(F_S/F^+, W_N (1)) \to \oplus_{i=1}^q H^1(\langle \sigma_i 
	\rangle, W_N(1)) \]
	has length bounded independently of $N$
	(note that $H^0(F^+, W_{E/\cO}(1))$ is a finite length $\cO$-module). An 
	application of the Chebotarev density theorem would then yield the theorem.

       To complete the proof, it therefore suffices to show that for any 
        non-zero homomorphism $f : E / \cO \to H^1(F_S / F^+, W 
_{E/\cO}(1))$,  we can find $\sigma \in 
G_{F(\zeta_{p^\infty})}$ 
such that $\rho(\sigma)$ has $n$ distinct eigenvalues in $E$ and $\Res^{G_{F^+, 
S}}_{\langle \sigma \rangle} \circ f : E / \cO \to W_{E / \cO}(1) / (\sigma - 
1)W_{E / \cO}(1)$ is still non-zero (as then we can argue by induction to get 
$\sigma_1, \dots, \sigma_q$ satisfying conditions (a), (b) above).

Let $F_\infty = F(\zeta_{p^\infty})$, let $L_\infty' / F^+$ be the 
extension cut out by $W_E(1)$, and let $L_\infty = L_\infty' \cdot 
F_\infty$. Then $H^1(L_\infty' / F^+, W_E(1)) = 0$, by \cite[Lemma 
6.2]{Kis04a}. (It is in our appeal to this result that we make use of the purity assumption in the statement of the lemma.) The extension $L_\infty / L_\infty'$ is finite (as 
$E(\epsilon 
\delta_{F / F^+}) \subset W_E(1)$), so $H^1(L_\infty / F^+, W_E(1)) = 0$. 
It follows that $H^1(L_\infty / F^+, W_{E / \cO}(1))$ is killed by a power of 
$p$ and hence the homomorphism 
\begin{align*} \Res^{G_{F^+, S}}_{G_{L_\infty,S_{L_\infty}}} \circ f : E / 
\cO \to &H^1(F_S/L_\infty, W_{E / \cO}(1))^{G_{F^+, S}} \\&\cong \Hom_{G_{F^+, 
S}}(G_{L_\infty,S_{L_\infty}}, W_{E / \cO}(1)) \end{align*}
is still non-zero. Let $M \subset W_{E / \cO}(1)$ be the $\cO$-submodule 
generated by the elements $f(x)(\sigma)$, $x \in E/\cO$, $\sigma \in 
G_{L_\infty}$; it is a non-zero divisible $\cO[G_{F_\infty}]$-submodule of 
$W_{E / \cO}(1)$. Using Lemma \ref{lem_non-zero_cohomology}, we see that there 
exists $\sigma \in G_{F_\infty}$ such that $\rho(\sigma)$ has $n$ distinct 
eigenvalues in $E$ and $M \not\subset (\sigma-1)W_{E / \cO}(1)$. Consequently, 
there exists $m \geq 0$ and $\tau \in G_{L_\infty}$ such that $f(1 / 
\varpi^m)(\tau) \not\in (\sigma - 1)W_{E / \cO}(1)$.

If $f(1 / \varpi^m)(\sigma) \not\in (\sigma - 1)W_{E / \cO}(1)$, then we're 
done: $\Res^{G_{F^+, S}}_{\langle \sigma \rangle} \circ f$ is non-zero. 
Otherwise, we can assume $f(1 / \varpi^m)(\sigma) \in (\sigma - 1)W_{E / 
\cO}(1)$, and then $\Res^{G_{F^+, S}}_{\langle \tau \sigma \rangle} \circ f$ is 
non-zero. This completes the proof.
\end{proof}

Putting everything together, we obtain:
\begin{cor}\label{cor:TWGalois}
 Let $q \geq \mathrm{corank}_\cO H^1(F_S / F^+,W_{E / \cO}(1))$, and suppose that $\rho$ 
	satisfies the following conditions:
	\begin{enumerate}
		\item For all but finitely many finite places $v \nmid S$ of $F$, the eigenvalues of $\rho(\Frob_v)$ are algebraic numbers which have absolute value $q_v^{w/2}$ with respect to any complex embedding.
		\item For each $v \in S$, $\rho|_{G_{F_\wv}}$ is generic.
		\item For each place $v | \infty$ of $F^+$, $\chi(c_v) = -1$.
		\item $\rho(G_{F(\zeta_{p^\infty})})$ is enormous.
	\end{enumerate}
	Then there exists $d \in \NN$ such that for each $N \in \NN$ we can find a 
	Taylor--Wiles datum $Q_N$ of level $N$ with $|Q_N| = q$ and a map 
	\[ \cO\llbracket x_1,\ldots,x_{nq} \rrbracket \to R_{S\cup Q_N} \] 
	such that the images of the $x_i$ are in $\q_{S \cup Q_N}$ and 
	\[\q_{S \cup Q_N}/(\q_{S \cup Q_N}^2,x_1,\ldots,x_{nq})\] is a quotient of 
	$(\cO / 
	\varpi^d)^{g_0}$, where $g_0 = g_0(S, \overline{\rho}, q)$ is as 
	defined in the statement of Lemma \ref{lem_upper_bound_on_ps_ring}. 
\end{cor}
\begin{proof}
	Combining Proposition \ref{prop_comparison_of_csd_tangent_spaces}, 
	Corollary \ref{cor_if_we_kill_dual_Selmer} and Lemma \ref{lem:findTWprimes} 
	we deduce that there exists a constant $d$ such that for each $N$ we can 
	find $Q_N$ of 
	level $N$ and an $\cO$-module map $\cO^{nq} \to \q_{S \cup Q_N}/\q_{S \cup 
	Q_N}^2 
	\otimes_\cO \cO / 
	\varpi^N$ with cokernel killed by $\varpi^d$ (note that the two 
	$\cO$-modules $\q_{S \cup Q_N}/\q_{S \cup Q_N}^2 
	\otimes_\cO \cO / 
	\varpi^N$ and $\Hom_{\cO}(\q_{S \cup Q_N}/\q_{S \cup Q_N}^2,\cO/\varpi^N)$ 
	are abstractly isomorphic). This allows us to define a map $\cO\llbracket 
	x_1,\ldots,x_{nq} \rrbracket \to R_{S\cup Q_N}$ with the $x_i$ mapping to 
	images of generators of $\cO^{nq}$ in $\q_{S \cup Q_N}/\q_{S \cup Q_N}^2 
	\otimes_\cO \cO / \varpi^N$, so that
	\[ \q_{S \cup Q_N}/(\q_{S \cup Q_N}^2,x_1,\ldots,x_{nq}) \otimes_\cO \cO / 
	\varpi^N \]
	is killed by $\varpi^d$. We need to explain how to deduce the slightly 
	stronger result in the statement of the corollary. We first note that 
	$\q_{S \cup Q_N} / \q_{S \cup Q_N}^2$ is a quotient of $\cO^{g_0}$. Indeed, it is a 
	finitely generated $\cO$-module, and there is an isomorphism $\q_{S \cup Q_N} / 
	\q_{S \cup Q_N}^2 \otimes_\cO \cO / \varpi \cong \ffrm_{R_{S \cup Q_N}} / (\ffrm^2_{R_{S \cup Q_N}}, \varpi)$, 
	so we can apply Nakayama's lemma together with Lemma 
	\ref{lem_upper_bound_on_ps_ring}.
	
	We may assume that $N > d$. In this case $M = \q_{S \cup Q_N}/(\q_{S \cup 
	Q_N}^2,x_1,\ldots,x_{nq})$ is a quotient of $\cO^{g_0}$ with the property 
	that $M / (\varpi^N)$ is killed by $\varpi^d$. This is only possible if $M$ 
	is itself killed by $\varpi^d$, implying that $M$ is a quotient of $(\cO / 
	\varpi^d)^{g_0}$.
\end{proof}

\subsection{Some examples of enormous subgroups}\label{sec_examples_of_enormous}

Let $E / \bQ_p$ be a coefficient field, and let $H \subset \GL_n(\cO)$ be a compact subgroup.
\begin{lemma}\label{lem_alg_implies_enormous} Suppose that for each $h \in H$, the characteristic polynomial of $h$ has all of its roots in $E$.
	\begin{enumerate}
		\item Let $H' \subset H$ be a closed subgroup. If $H'$ is enormous, then so is $H$. 
		\item Let $G \subset \GL_n$ denote the Zariski closure of $H$. If $G^\circ$ contains regular semisimple elements and acts absolutely irreducibly on $E^n$, then $H$ is enormous.
	\end{enumerate}
\end{lemma}
\begin{proof}
	The first part is immediate from the definitions. For the second, we can assume that $G = G^\circ$. Since $G$ acts absolutely irreducibly, $G(E)$ spans $W_E$. Let $H^{reg} \subset H$ denote the set of regular semisimple elements. It is Zariski dense in $G$. Indeed, by hypothesis $G^{reg}$ is a non-empty Zariski open subset of $G$. The Zariski closure of $H$ is contained in the union of the Zariski closure of $H^{reg}$ and $G - G^{reg}$. This forces the Zariski closure of $H^{reg}$ to be equal to $G$.
	
	We must show that for any non-zero $v \in W_E$, there exists $h \in H^{reg}$ such that $\tr hv \neq 0$. If $\tr hv = 0$ for all $h \in H^{reg}$, then Zariski density implies that $\tr g v = 0$ for all $g \in G$. This contradicts the fact that the elements of $G(E)$ span $W_E$.
\end{proof}
\begin{example}\label{ex:noncmenormous} Let $F$ be a totally real or CM number 
field, and let $\pi$ 
be a regular algebraic, cuspidal automorphic representation of $\GL_2(\bA_F)$. 
Let $\iota : \overline{\bQ}_p \to \bC$ be an isomorphism, and let $\rho : G_F 
\to \GL_n(\cO)$ be a model of $\Sym^{n-1} r_{\pi, \iota}$ defined over $\cO$. 
If $\Sym^2 \pi$ is cuspidal, then (after possibly enlarging $E$) 
$\rho(G_{F(\zeta_{p^\infty})})$ is an enormous subgroup of $\GL_n(\cO)$.
	
To see this, it is enough to note that the Zariski closure of the image of $r_{\pi, \iota}$ 
contains $\SL_2$, and therefore that the Zariski 
closure of $r_{\pi,\iota}(G_{F(\zeta_{p^\infty})})$ also contains $\SL_2$ 
(because passage to derived subgroup respects Zariski closure, cf.\ \cite[Ch. 
I, \S 2.1]{Bor91}). We can then appeal to Lemma \ref{lem_alg_implies_enormous}.

We justify the claim that the Zariski closure of $r_{\pi, \iota}(G_F)$ contains 
$\SL_2$. The identity component of the Zariski closure of $r_{\pi, \iota}(G_F)$ 
is a reductive subgroup of $\GL_2$ which contains regular semisimple elements 
(by \cite[Theorem 1]{Sen}). The only possibility we need to rule out is that 
$r_{\pi, \iota}(G_F)$ normalizes a maximal torus in $\GL_2$. In this case, 
there is a quadratic extension $F' / F$ such that $r_{\pi, \iota}|_{G_{F'}}$ is 
reducible. It's therefore enough to show that for any quadratic extension $F' / 
F$, $r_{\pi, \iota}|_{G_{F'}}$ is irreducible. We observe that if $r_{\pi, 
\iota}|_{G_{F'}}$ is reducible, then it's isomorphic to a sum $\chi_1 \oplus 
\chi_2$ of characters. Moreover, $\chi_1, \chi_2$ are de Rham and almost 
everywhere unramified, so therefore can be extended to compatible systems of 
1-dimensional Galois representations. It follows that $r_{\pi, 
\iota'}|_{G_{F'}}$ is reducible for any other prime $p'$ and isomorphism 
$\iota' : \overline{\bQ}_{p'} \to \bC$. In particular, $\Sym^2 r_{\pi, \iota'}$ 
is reducible. However, \cite[Theorem 5.5.2]{BLGGT} implies that for a Dirichlet 
density 1 set of primes $p'$, the representation $\Sym^2 r_{\pi, \iota'}$ is 
irreducible (note that for automorphic representations of $\GL_3(\bA_F)$ the 
asssumption of `extremely regular weight' in \emph{loc.~cit.}~coincides with 
the usual notion of regular weight).
\end{example}
\begin{example}Let $F$ be a CM field, and let $\pi$ be a polarizable 
automorphic 
representation of $\GL_n(\bA_F)$ such that for some finite place $v_0$ of $F$, 
$\pi_{v_0}$ is a twist of the Steinberg representation. Let $\iota : 
\overline{\bQ}_p \to \bC$ be an isomorphism. Then $r_{\pi, 
\iota}(G_{F(\zeta_{p^\infty})})$ is enormous. 
	
Indeed, let $G$ denote the Zariski closure of $r_{\pi, \iota}(G_F)$. 
Local-global compatibility at the place $v_0$ implies that $G^\circ$ contains a 
regular unipotent element (if $v_0 | p$, we argue as in \cite[Lemma 
3.2]{Shi19}), so in particular it acts absolutely irreducibly on $E^n$. 
Then \cite[Proposition 4]{Sch06} (see also \cite[Classification Theorem 
11.6]{katz-gauss}) shows that the derived group of $G^\circ$ is 
one of a finite list of possibilities, and that in any case it contains regular 
semisimple elements. We can again appeal to Lemma 
\ref{lem_alg_implies_enormous}.
\end{example}

\section{A result about Hecke algebras}\label{sec:oldforms}

Let $p$ be a prime, let $n \geq 2$, and let $F_v / \bbQ_l$ be a finite 
extension for some $l \neq p$. Let $G = \GL_n(F_v)$, $U = \GL_n(\cO_{F_v})$, and let $I \subset 
U$ be the standard Iwahori subgroup (i.e. the pre-image in $U$ of the upper-triangular matrices in $GL_n(k(v))$). Let $E / \bbQ_p$ be a coefficient field. For $\cO$ sufficiently large (i.e.\ containing a square root of $q_v$), the 
Iwahori Hecke algebra $\cH_I = \cH(G, I) \otimes_\bbZ \cO$  has the 
Bernstein presentation
\[ \cH_I \cong \cO[X_\ast(T)] \widetilde{\otimes} \cO[I \backslash U / I]. \] 
The map 
$\cO[X_\ast(T)] \to \cH_I $ sends a dominant cocharacter $\lambda \in 
X_\ast(T)_+$ to the Hecke operator $q_v^{-l(\lambda)/2}[I\lambda(\varpi_v)I]$, 
where $l(\cdot)$ is 
the usual length function on the extended affine Weyl group. The twisted tensor product indicates the usual tensor product as $\cO$-modules, with the algebra structure on $\cH_I$ determined by the relations of \cite[Proposition 3.6]{Lus89}. We identify $\cO[X_\ast(T)] = \cO[x_1,x_1^{-1}, \dots, 
x_n, x_n^{-1}]$ ($x_i$ 
is the cocharacter embedding $\bG_m$ into the $i$th diagonal entry of $T$). The 
centre, $Z(\cH_I)$, is identified by the Bernstein presentation with the algebra 
of symmetric polynomials $\cO[X_\ast(T)]^{S_n} = \cO[e_1, e_2,\dots, e_n, 
e_n^{-1}]$ ($e_1, \dots, 
e_n$ are the usual elementary symmetric polynomials in $x_1, \dots, x_n$). 

Our identification of $\cO[X_\ast(T)]$ with a polynomial algebra allows us to 
speak of 
polynomials as being elements of the Hecke algebra. In particular, we can think 
of $\Delta = \prod_{1 \leq i < j \leq n} (x_i - x_j)$ as being an element of 
$\cH_I$, and its square $\Delta^2$ as being an element of the centre $Z(\cH_I)$.

To simplify notation, let $\cR = \cO[X_\ast(T)]^{S_n}$, $\cS = \cO[X_\ast(T)]$. Then $\cS$ is a free $\cR$-module, a basis being given by the monomials $x_{\mathbf{a}} = x_1^{a_1} \dots x_n^{a_n}$ for $\mathbf{a} = (a_1, \dots, a_n) \in \bbZ^n$ satisfying $0 \leq a_i \leq i-1$ for each $i = 1, \dots, n$. We write $B$ for the set of tuples $\mathbf{a}$ satisfying these conditions. 

If $M$ is an $\cO[\GL_n(F_v)]$-module $M$, then $M^U$ is an $\cR$-submodule of $M^I$ (and in fact, if $z \in Z(\cH_I)$ and $M \in M^U$, then we have the formula $z \cdot m = ([U]z) \cdot_U m$, where $\cdot_U$ denotes the action of $\cH_U = \cH(G, U) \otimes_\bZ \cO$ on $M^U$ -- see \cite[\S 4.6]{Hai10}). Thus there there is a canonical (and functorial) 
morphism
\begin{equation}\label{eqn_Hecke_alg_morphism} M^U \otimes_{\cR} \cS \to M^I, 
\end{equation}
given by the formula $m \otimes s \mapsto sm$.  Since $\cS$ is free over $\cR$, $M^U 
\otimes_{\cR} \cS$ may be identified with $\oplus_{\mathbf{a} \in B} M^U$, and 
the above map with $(m_{\mathbf{a}})_{\mathbf{a} \in B} \mapsto 
\sum_{\mathbf{a} \in B} x_{\mathbf{a}} \cdot m_{\mathbf{a}}$.

The aim of this short section is to prove the following result, which will be 
applied in Section \ref{sec:patching} (see Proposition \ref{prop:oldforms}).
\begin{proposition}\label{prop_uniform_pseudo-isomorphism}
	Let $N \geq 1$, and let $M$ be an $\cO / \varpi^N[\GL_n(F_v)]$-module. Suppose that $q_v \equiv 1 \text{ mod }\varpi^N$. Then the above morphism $M^U \otimes_{\cR} \cS \to M^I$ has kernel and cokernel annihilated by $\Delta^{n!}$.
\end{proposition}
Note that $\Delta^{n!}$ always lies in $Z(\cH_I)$. This is important for us since it means that in the global situation, $\Delta^{n!}$ will be in the image of the pseudodeformation ring (through which decomposition groups act via a homomorphism to the Bernstein centre). 

Before proving the proposition, we establish an auxiliary result.
\begin{lemma}
	Consider the $n! \times n!$ matrix $P$ with coefficients in $\bbZ[x_1, \dots, x_n]$ given by the formula $P_{\sigma, \mathbf{a}} = \sigma(x_{\mathbf{a}})$ ($\sigma \in S_n$, $\mathbf{a} \in B$). Then there exists a unique matrix $Q = (Q_{\mathbf{a}, \sigma})$ with coefficients in $\bbZ[x_1, \dots, x_n]$ such that $PQ = QP = \Delta^{n!}$.
\end{lemma}
\begin{proof}
	It suffices to show existence, since then uniqueness follows by linear 
	algebra over $\bbQ(x_1, \dots, x_n)$. The square of the determinant of $P$ 
	is equal to the discriminant of the ring extension $\bbZ[e_1, \dots, e_n] 
	\to \bbZ[x_1, \dots, x_n]$. Using \cite[Tag 0C17]{stacks-project}, we see 
	that the discriminant of this ring extension equals $\Delta^{n!}$ (the 
	different ideal is generated by $\Delta$). Therefore the determinant of $P$ 
	is equal to $\Delta^{n! / 2}$, up to sign. 
	
	The existence of the adjugate matrix implies that there is a matrix $Q'$ with coefficients in $\bbZ[x_1, \dots, x_n]$ such that $PQ' = \Delta^{n! / 2}$. We then take $Q = \Delta^{n! / 2} Q'$. 
\end{proof}
We observe that for all $\mathbf{a} \in B$, $\sigma, \tau \in S_n$, we have $\sigma(Q_{\mathbf{a}, \tau}) = Q_{\mathbf{a}, \sigma \tau}$. Indeed, this follows from the identity $\sigma(P) \sigma(Q) = \Delta^{n!}$ and the uniqueness of inverses.
\begin{proof}[Proof of Proposition \ref{prop_uniform_pseudo-isomorphism}] Since $q_v \equiv 1 \text{ mod }\varpi^N$, we can identify $\cO / \varpi^N[I \backslash U / I] = \cO / \varpi^N[S_n]$, and $\cH(G, I) \otimes_\bbZ \cO / \varpi^N$ with the group algebra of the extended affine Weyl group $X_\ast(T) \rtimes S_n$ (because the Iwahori--Matsumoto relations are a $q$-deformation of the relations defining the group algebra of the affine Weyl group).
	Let $e = \sum_{\sigma \in S_n} \sigma \in \cO / \varpi^N[S_n] \subset \cH(G, I) \otimes_\bbZ \cO / \varpi^N$. Then $e 
	= [U]$, so in particular $e M^I \subset M^U$. Recalling that $[I]$ is the 
	unit of $\cH_I$, we note that $e$ need not be an idempotent, since $e^2 = [U:I]e$ (note that $q_v \equiv 1 
	\text{ mod }\varpi^N \implies [U:I] \equiv n! 
	\text{ mod }\varpi^N$, and we do not rule out the 
	case $p \leq n$). 
	
	We have defined a map $f : \oplus_{\mathbf{a} \in B} M^U \to M^I$ by the formula $(m_{\mathbf{a}})_{\mathbf{a} \in B} \mapsto \sum_{\mathbf{a} \in B} x_{\mathbf{a}} \cdot m_{\mathbf{a}}$. We define a map $g : M^I \to \oplus_{\mathbf{a} \in B} M^U$ by the formula $g(m) = (e Q_{\mathbf{a}, 1} m)_{\mathbf{a} \in B}$.
	
	We now compute $f \circ g$ and $g \circ f$. We have for $m \in M^I$
	\[ f(g(m)) = \sum_{\mathbf{a} \in B} x_{\mathbf{a}} e Q_{\mathbf{a}, 1} m =  \sum_{\mathbf{a} \in B} \sum_{\sigma \in S_n} x_{\mathbf{a}} \sigma(Q_{\mathbf{a}, 1}) \sigma(m). \]
	This in turn we can rewrite as
	\[  \sum_{\sigma \in S_n} \sum_{\mathbf{a} \in B} P_{1, \mathbf{a}} Q_{\mathbf{a}, \sigma} \sigma(m) = \Delta^{n!} m. \]
	Similarly, we have for $m = (m_{\mathbf{a}})_{\mathbf{a} \in B} \in \oplus_{\mathbf{a} \in B} M^U $:
	\[ g(f(m))_{\mathbf{a}} = e Q_{\mathbf{a}, 1} \sum_{\mathbf{b} \in B} x_{\mathbf{b}} \cdot m_{\mathbf{b}} = \sum_{\sigma \in S_n} \sum_{\mathbf{b} \in B} Q_{\mathbf{a}, \sigma} P_{\sigma, \mathbf{b}} \sigma(m_{\mathbf{b}}). \]
	Note that $S_n$ acts trivially on $M^U$. We can therefore rewrite the above expression as
	\[ \sum_{\mathbf{b} \in B} \sum_{\sigma \in S_n} Q_{\mathbf{a}, \sigma} P_{\sigma, \mathbf{b}} m_{\mathbf{b}} = \Delta^{n!} m_{\mathbf{a}}. \]
	This completes the proof. 	
\end{proof}

\section{Patching}\label{sec:patching}

In this section we prove our main technical result (Theorem 
\ref{thm_vanishing_of_adjoint_Selmer_group}).
\subsection{Set-up}
	We suppose given the following data:
\begin{itemize}
	\item A CM number field $F$ with maximal totally real subfield $F^+$. We assume that $F / F^+$ is everywhere unramified and that $[F^+ : \bQ]$ is even.
	\item An integer $n \geq 2$ and a cuspidal, polarized, regular algebraic automorphic representation 
	$(\pi,\delta_{F/F^+}^n)$ of $\GL_n(\bA_F)$ (i.e.~$\pi$ is of unitary type).
	\item A prime $p$ and an isomorphism $\iota : \overline{\bQ}_p \to \bC$. We 
	assume that for each place $w | p$ of $F$, $\pi_w$ has an Iwahori-fixed 
	vector.
	\item A finite set $S$ of finite places of $F^+$, containing the set $S_p$ of $p$-adic places and all places above which $\pi$ is ramified. We assume that each place of $S$ splits in $F$.
\end{itemize}
We recall that under these conditions we define an extension of $r_{\pi, \iota}$ to a homomorphism to $\cG_n$, which then gives the action of $G_{F^+}$ on $\ad r_{\pi, \iota}$ (see \S \ref{sec_notation}).
\begin{theorem}\label{thm_vanishing_of_adjoint_Selmer_group}
	With set-up as above, assume moreover that $r_{\pi, 
	\iota}(G_{F(\zeta_{p^\infty})})$ is enormous. Then \[H^1_f(F^+, \ad r_{\pi, 
	\iota}) = 0.\]
\end{theorem}
We note here that the assumptions of Lemma \ref{lem_greenberg_wiles_formula} 
hold for $r_{\pi,\iota}$ by \cite[Theorem 2.1.1]{BLGGT} (which collects 
together results of many people). 

The proof of Theorem \ref{thm_vanishing_of_adjoint_Selmer_group} will use 
automorphic forms on definite unitary groups. To this end, we can find the 
following data:
\begin{itemize}
	\item For each place $v \in S$, a choice of place $\wv$ of $F$ lying above $v$. We set $\widetilde{S} = \{ \wv \mid v \in S \}$ and $\widetilde{S}_p = \{ \wv \mid v \in S_p \}$.
	\item A Hermitian form $\langle \cdot, \cdot \rangle : F^n \times F^n \to F$ such that the associated unitary group $G$ (defined on $R$-points by $G(R) = \{ g \in \GL_n(F \otimes_{F^+} R) \mid g^\ast g = 1 \}$) is definite at infinity and quasi-split at each finite place of $F^+$.
	\item A reductive group scheme over $\cO_{F^+}$ extending $G$.
	\item For each finite place $v = w w^c$ of $F^+$ which splits in $F$, an 
	isomorphism $\iota_w : G_{\cO_{F^+_v}} \to \Res_{\cO_{F_w} / \cO_{F^+_v}} 
	\GL_n $ of group schemes over $\cO_{F^+_v}$. We assume that the 
	induced isomorphism $\iota_w : G(F^+_v) \to \GL_n(F_w)$ is in the same 
	inner class as the isomorphism given by inclusion $G(F_v^+) \subset 
	\GL_n(F_w) \times \GL_n(F_{w^c})$, followed by projection to the first 
	factor.
	\item An automorphic representation $\sigma$ of $G(\bA_{F^+})$ with the following properties:
	\begin{itemize}
		\item For each finite inert place $v$ of $F^+$, $\sigma_v^{G(\cO_{F^+_v})} \neq 0$ and $\sigma_v$, $\pi_v$ are related by unramified base change.
		\item For each split place $v = w w^c$ of $F^+$, $\sigma_v \cong \pi_w 
		\circ \iota_w$.
		\item If $v | \infty$ is a place of $F^+$, then the infinitesimal character of $\sigma_v$ respects that of $\pi_v$ under base change. (We recall this relation more precisely below.)
	\end{itemize}
	\item An open compact subgroup $U = \prod_v U_v$ of $G(\bA_{F^+}^\infty)$ with the following properties:
	\begin{itemize}
		\item For each place $v \in S_p$, $U_v = \iota_{\wv}^{-1}(\Iw_{\wv})$, where $\Iw_\wv \subset \GL_n(\cO_{F_\wv})$ is the standard Iwahori subgroup (defined as in \S \ref{sec:oldforms}).
		\item For each inert place $v$ of $F^+$, $U_v = G(\cO_{F^+_v})$.
		\item $(\sigma^\infty)^U \neq 0$.
		\item $U$ is sufficiently small: for all $g \in G(\bA_{F^+}^\infty)$, $g U g^{-1} \cap G(F^+) = \{ 1 \}$.
	\end{itemize}
\end{itemize}
(We can find such a $G$ because $[F^+ : \bQ]$ is even. The existence of $\sigma$ is deduced from that of $\pi$ using \cite[\S 5]{labesse}.)
We can regard $\sigma_\infty$ as an algebraic representation of the group 
$(\Res_{F^+ / \bQ} G)_\bC$. Let $\widetilde{I}_p \subset \Hom(F, 
\overline{\bQ}_p)$ denote the set of 
embeddings inducing places $\wv \in \widetilde{S}_p$. Then our choices 
determine an isomorphism
\[ (\Res_{F^+ / \bQ} G)_{\overline{\bQ}_p} \cong \prod_{\tau \in \widetilde{I}_p} \GL_n. \]
Let $\lambda = (\lambda_\tau)_{\tau \in \widetilde{I}_p} \in 
(\bZ_+^n)^{\widetilde{I}_p}$ denote the highest weight of the algebraic 
representation $V_\lambda$ of $(\Res_{F^+ / \bQ} G)_{\overline{\bQ}_p}$ such 
that $V_\lambda \otimes_{\iota, \overline{\bQ}_p} \bC \cong 
\sigma_\infty^\vee$. We can define a highest weight $\xi$ for 
$(\Res_{F/\bQ}\GL_n)_{\overline{\bQ}_p}$ by letting $\xi_\tau = \lambda_\tau$ 
and $\xi_{\tau c} = -w_0\lambda_\tau$ for $\tau \in \widetilde{I}_p$ ($w_0$ is 
the longest element in the Weyl group of $\GL_n$). The 
infinitesimal character of 
$\pi_\infty$ is the same as that of $V_\xi^\vee\otimes_{\iota,\Qpbar}\CC$.   

The Hodge--Tate weights of $r_{\pi, \iota}$ may be described as follows: if $\tau \in \widetilde{I}_p$, then
\[ \mathrm{HT}_\tau(r_{\pi, \iota}) = \{ \lambda_{\tau, 1} + n-1, \lambda_{\tau, 2} + n-2, \dots, \lambda_{\tau, n} \}. \]
We fix once and for all integers $a \leq b$ such that for all $\tau \in \Hom(F, \overline{\bQ}_p)$, the elements of $\mathrm{HT}_\tau(r_{\pi, \iota})$ are contained in $[a, b]$ and $a + b = n-1$.

Let $E / \bQ_p$ be a coefficient field containing the image of every embedding $F \to \overline{\bQ}_p$. After possibly enlarging $E$, we can assume that there is a model $\rho : G_{F, S} \to \GL_n(\cO)$ of $r_{\pi, \iota}$, which extends to a homomorphism $r : G_{F^+, S} \to \cG_n(\cO)$ such that $\nu \circ r = \epsilon^{1-n} \delta_{F / F^+}^n$. Let $\overline{D}$ denote the group determinant of $\overline{\rho}$, which is then defined over $k$. 

With these choices the pseudodeformation ring denoted $R_S$ in \S 
\ref{sec_unitary_pseudocharacters} is defined, as well as the prime ideal $\q_S 
= \ker(R_S \to \cO)$ determined by $\rho$. Moreover, for any Taylor--Wiles 
datum $(Q, \widetilde{Q}, (\alpha_{\wv, 1}, \dots, \alpha_{\wv, n})_{v \in Q})$ 
we have the auxiliary ring $R_{S \cup Q}$. We introduce one more object here: 
it is the maximal quotient $R_{S \cup Q} \to R_{S \cup Q, ab}$ over which for 
each $v \in Q$, the restriction of the universal pseudocharacter to $W_{F_\wv}$ 
factors through $W_{F_\wv}^{ab}$. Thus we have a diagram
\[ R_{S \cup Q} \to R_{S \cup Q, ab} \to R_S. \]

\subsection{Hecke algebras}

We can find a representation $\cV_\lambda$ of the group scheme 
$(\Res_{\cO_{F^+} / \bZ} G)_\cO$, finite free over $\cO$, and such that $\cV_\lambda \otimes_\cO 
\overline{\bQ}_p \cong V_\lambda$. (For example, use the construction of 
\cite[\S 2.2]{ger}.) Thus $\cV_\lambda(\cO)$ is a finite free $\cO$-module 
which receives an action of $U_p = \prod_{v \in S_p} U_v$. For any open compact 
subgroup $V = \prod_v V_v \subset U$, and any $\cO$-algebra $A$, we define 
$S_\lambda(V, A)$ to be the set of functions $f : G(\bA_{F^+}^\infty) \to 
\cV_\lambda(A)$ such that for each $v \in V$, $\gamma \in G(F^+)$, $g \in 
G(\bA_{F^+}^\infty)$, $v_p f(\gamma g v) = f(g)$. We observe that
\[ \varinjlim_{U^p} S_\lambda(U^p U_p, A) \]
has a natural structure of $A[U^p]$-module, and the $U^p$-invariants are $S_\lambda(U, A)$. It follows that $S_\lambda(U, A)$ has a natural structure of $\cH(G(\bA_{F^+}^{\infty, p}), U^p)$-module. A standard argument (cf. \cite[Lemma 2.2.5]{ger}) shows that there is an isomorphism of $\cH(G(\bA_{F^+}^{\infty, p}), U^p)$-modules
\[ S_\lambda(U, \cO) \otimes_{\iota, \cO} \bC \cong \oplus_{\mu} (\mu^\infty)^U, \]
where the sum is over automorphic representations of $G(\bA_{F^+})$ (with multiplicity) such that $\mu_\infty \cong \sigma_\infty$. 

If $V = \prod_v V_v$ is an open compact subgroup of $U$ and $T$ is a finite set 
of places of $F^+$ containing all places such that $V_v \neq G(\cO_{F^+_v})$, 
then we write $\bT_\lambda^T(V, A)$ for the $A$-subalgebra of 
$\End_A(S_\lambda(V, A))$ generated by the unramified Hecke operators at 
split places away from $T$. After possibly enlarging $E$, the existence of 
$\sigma$ implies the existence of a homomorphism
\[ h_{V, \sigma} : \bT^T_\lambda(V, \cO) \to \cO \]
giving the Hecke eigenvalues of $\iota^{-1} \sigma^\infty$. On the other hand, the results of \cite[\S 5]{labesse} imply the existence of a group determinant $D_{V, \lambda}$ of $G_{F, T}$ valued in $\bT^T_{\lambda}(V, \cO)$ (construction as in \cite[Proposition 4.11]{jackreducible}).

Let $\ffrm \subset \bT^S_\lambda(U, \cO)$ denote the unique maximal ideal containing $\ker h_{U, \sigma}$, and set
\[ S_\emptyset = S_\lambda(U, \cO)_\ffrm, \bT_\emptyset = \bT^S_\lambda(U, \cO)_\ffrm. \]
Then there is a surjective homomorphism $R_{\overline{D},S} \to \bT_\emptyset$ classifying $D_{U, \lambda}$.

\begin{lemma}\label{lem:pseudowithconditionstohecke}
The map $R_{\overline{D},S} \to \bT_\emptyset$ factors through the quotient 
$R_S$.
\end{lemma}
\begin{proof}
If we invert $p$ then $\TT_\emptyset\otimes_{\cO} \overline{\bQ}_p = 
\prod_{\mu}E_\mu$ is a 
product of fields indexed by automorphic representations $\mu$ of 
$G(\A_{F^+})$ with $\mu^{U}_{\ffrm} \ne 0$ and $\mu_\infty \cong 
\sigma_\infty$. To prove 
the lemma, it suffices to show that each of the maps 
$R_{\overline{D},S}\to E_\mu$ factors through the quotient 
$R_{S}$: in other words, the conjugate self-duality condition and the semi-stability condition of (\ref{eqn_Cayley}). These conditions follow from local--global compatibility for the Galois 
representation associated to the base change of $\mu$.
\end{proof}

\subsection{Automorphic data associated to Taylor--Wiles data}

Suppose given a set $Q$ of Taylor--Wiles places. In this case we define open compact subgroups $U_0(Q) = \prod_v U_0(Q)_v$ and $U_1(Q) = \prod_v U_1(Q)_v$ as follows:
\begin{itemize}
	\item If $v \not\in Q$, then $U_0(Q)_v = U_1(Q)_v = U_v$.
	\item If $v \in Q$, then $U_0(Q)_v = \iota_\wv^{-1}(\Iw_\wv)$ and $U_1(Q)_v$ is the smallest open subgroup of $U_1(Q)_v$ such that $U_0(Q)_v / U_1(Q)_v$ is a $p$-group.
\end{itemize}
We set $\Delta_Q = U_0(Q) / U_1(Q)$, which may be naturally identified with 
$\prod_{v \in Q} k(v)^\times(p)^n$. We write $\ffrm_Q$ for the intersection of 
$\ffrm$ with $\bT_\lambda^{S \cup Q}(U, \cO)$, $\ffrm_{0, Q}$ for the pre-image 
of $\ffrm_Q$ in $\bT_\lambda^{S \cup Q}(U_0(Q), \cO)$, and $\ffrm_{1, Q}$ for 
the pre-image of $\ffrm_{0, Q}$ in $\bT_\lambda^{S \cup Q}(U_1(Q), \cO)$. We 
define
\[ \bT_{0, Q} = \bT_\lambda^{S \cup Q}(U_0(Q), \cO)_{\ffrm_{0, Q}}, \bT_Q = 
\bT_\lambda^{S \cup Q}(U_1(Q), \cO)_{\ffrm_{1, Q}}. \]

As in Lemma \ref{lem:pseudowithconditionstohecke}, we have a surjective map 
$R_{S \cup Q} \to \TT_Q$. Note that the natural map $ \bT_\lambda^{S \cup Q}(U, 
\cO)_{\ffrm_Q} \to \bT_\emptyset$ is in fact an isomorphism, and so there are 
surjections
\[ \bT_{Q} \to \bT_{0, Q} \to \bT_\emptyset. \]
So far we have not used any Hecke operators at places $v \in Q$. For any $v \in 
Q$, $\alpha \in F_\wv^\times$, and $1 \le i \le n$, we let $t_{v, i} : F_\wv^\times \to \cH(G(F_v^+), 
U_1(Q)_v)$ denote the composite with $\iota_\wv^{-1}$ of the homomorphism defined just 
above \cite[Proposition 2.2.7]{10author} (and denoted $t_{v, i}$ there). That proposition shows that if $\pi_v$ is an irreducible admissible 
$\overline{\bQ}_p[G(F_v^+)]$-module such that $\pi_v^{U_1(Q)_v} \neq 0$, then 
for any $\sigma \in W_{F_\wv}$ and $\alpha \in F_\wv^\times$ such that 
$\Art_{F_\wv}(\alpha) = \sigma|_{F_\wv^{ab}}$, the characteristic polynomial of 
$\rec^T_{F_\wv}(\pi_v \circ \iota_\wv^{-1})$ on $\sigma$ equals 
\begin{equation}\label{eqn:twlgc}\sum_{i=0}^n 
(-1)^i X^{n-i} e_{v, i}(\alpha, \pi_v),\end{equation} where $e_{v, i}(\alpha) 
\in  
\cH(G(F_v^+), U_1(Q)_v)$ is the $i$th elementary symmetric polynomial in $t_{v, 
1}(\alpha), \dots, t_{v, n}(\alpha)$, and $e_{v, i}(\alpha, \pi_v) \in 
\overline{\bQ}_p$ is the scalar by which it acts on $\pi_v^{U_1(Q)_v}$. The elements 
$e_{v,i}(\alpha)$ generate the centre of 
$\cH(G(F_v^+),U_1(Q)_v)\otimes_{\cO}\Qpbar$,
 by \cite[Proposition 
4.11]{Fli11}. 

We define $\bT_{0, Q}^Q \subset \End(S_\lambda(U_0(Q), \cO)_{\ffrm_{0, Q}})$ to 
be the subalgebra generated by $\bT_{0, Q}$ and the elements $t_{v, i}(\alpha)$ 
for all $v \in Q$, $i = 1, \dots, n$ and $\alpha \in F_\wv^\times$.
We 
define 
$\bT_{Q}^Q \subset \End(S_\lambda(U_1(Q), \cO)_{\ffrm_{1, Q}})$ similarly. Thus 
$\bT_{Q}^Q$ is an $\cO[\Delta_Q]$-algebra (the image of $\cO[\Delta_Q]$ in $\bT_Q^Q$ is generated by the elements $t_{v,i}(\alpha)$ with $\alpha \in 
\cO_{F_\wv}^\times$). Neither $\bT_{0, Q}^Q$ nor 
$\bT_{Q}^Q$ need be local rings. 
We denote by $\ba_Q$ 
the augmentation 
ideal of $\cO[\Delta_Q]$. 
\begin{lem}
	$S_\lambda(U_1(Q),\cO)$ is a free $\cO[\Delta_Q]$-module and the trace 
	map induces $S_\lambda(U_1(Q),\cO)/\ba_Q \cong 
	S_\lambda(U_0(Q),\cO)$. 
\end{lem}
\begin{proof}
	The proof is identical to that of \cite[Lemma 3.3.1]{cht}, using that $U$ 
	(and hence any subgroup of $U$) is sufficiently small.
\end{proof}
We let $A_Q = \otimes_{v \in 
	Q}\cO[t_{v, 1}(\varpi_v)^{\pm 1}, \dots, t_{v, n}(\varpi_v)^{\pm 1}]$. This is a polynomial 
	subalgebra 
	of  $\otimes_{v \in Q}\cH(G(F_v^+), U_0(Q)_v)$ that receives an action of 
	the 
	group $W_Q = \prod_{v \in Q} 
S_n$. For every $m \geq 1$, we have a canonical morphism of $\bT_Q$-modules 
\[\eta_{Q,m}: S_\lambda(U,\cO/\varpi^m)_{\m}\otimes_{A_Q^{W_Q}}A_Q \to 
S_\lambda(U_0(Q),\cO/\varpi^m)_{\m_{0, Q}},\] as in (\ref{eqn_Hecke_alg_morphism}).

For each $v \in Q$, the universal pseudocharacter over $R_{S \cup Q, ab}$ 
determines by restriction an $n$-dimensional pseudocharacter $\gamma_v$ of 
$W_{F_\wv}^{ab}$ valued in $R_{S \cup Q, ab}$. Each restriction 
$\gamma_v|_{I_{F_\wv}}$ factors through the quotient $k(v)^\times(p)$ of 
$\Art_{F_\wv}(\cO_{F_\wv}^\times)$ (compare \cite[Lemma 3.8]{chenevier_det}).

On the other hand, for each $i = 1, \dots, n$, there is a character $\alpha_{v, 
i} : W_{F_\wv}^{ab} \to (\TT_Q^Q)^\times$ given by the formula $\alpha_{v, 
i}(\Art_{F_\wv}(\alpha)) = t_{v, i}(\alpha)$. We write $\alpha_v$ for the 
pseudocharacter 
$\alpha_v = \alpha_{v, 1} \oplus \dots \oplus \alpha_{v, n}$.

These two families of pseudocharacters are related by the following lemma, 
which is a formulation of local--global compatibility at $v \in Q$. 

\begin{lemma}\phantomsection\label{lem_abelian_pseudocharacter}\label{lem:twlgc}
\begin{enumerate}	\item The map \[R_{S\cup Q} \to \TT_Q\] factors through 
	the quotient $R_{S \cup Q, ab}$.
	\item 	Let $v \in Q$. The composite of $\gamma_v$ with the map $R_{S \cup 
	Q, ab} \rightarrow \TT_Q\hookrightarrow 
	\TT_Q^{Q}$ equals $\alpha_v$.
	\item The image of the map \[R_{S\cup Q, ab} \to \TT_Q \hookrightarrow 
	\TT_Q^Q\] 
contains the Hecke operators $e_{v,i}(\alpha)$ for each $v 
\in Q, i 
= 1, \cdots, n$ and $\alpha \in F_{\wv}^\times$.
\end{enumerate}
\end{lemma}
\begin{proof}
If we invert $p$ then $\TT_Q\otimes_{\cO} \overline{\bQ}_p = \prod_{\mu}E_\mu$ is a 
product of fields indexed by automorphic representations $\mu$ of 
$G(\A_{F^+})$ with $\mu^{U_1(Q)}_{\ffrm_{1, Q}} \ne 0$ and $\mu_\infty \cong \sigma_\infty$. To prove 
the first part of the lemma, it suffices to show that each of the maps 
$R_{S\cup 
Q} \to E_\mu$ factors through the quotient 
$R_{S \cup Q, ab}$. This follows from \cite[Prop.~2.2.7]{10author}. The second 
and third parts of the lemma follow from the formula (\ref{eqn:twlgc}) which 
computes the 
characteristic 
polynomials 
of $\rec^T_{F_{\wv}}(\mu_{\wv}\circ \iota_{\wv}^{-1})$ evaluated on elements of 
$W_{F_{\wv}}$.
\end{proof}
We caution the reader that the map $R_{S \cup Q, ab} \to \TT_Q^Q$ is 
not in 
general surjective, because of the presence of Hecke operators at $Q$ which do 
not lie in the Bernstein centre. 

The following proposition will be crucial for controlling our patched modules 
of automorphic forms. As mentioned in the introduction, this is inspired by 
arguments of Pan \cite{Lue}.
\begin{prop}\label{prop:oldforms}
	Fix $d \in \NN$. There exists a constant $c \in \NN$ (depending only on 
	$d$) 
	such that, for any $N$ and any Taylor--Wiles 
	datum $(Q, \widetilde{Q}, (\alpha_{\wv, 1}, \dots, \alpha_{\wv, n})_{\wv 
	\in \widetilde{Q}})$ for $r_{\pi,\iota}$ of level $N$ satisfying \[\sum_{v 
	\in Q}\sum_{1 \leq i 
	< j \leq n}\ord_{\varpi}(\alpha_{\wv,i}-\alpha_{\wv,j})\le d,\] there is an 
	element 
	$f_Q \in R_{S \cup Q, ab}$ such that 
	\begin{enumerate}
		\item $f_Q$ kills the kernel and cokernel of $\eta_{Q,m}$ for all $m 
		\le N$
		\item The image $f_{Q,\sigma}$ of $f_Q$ under the composition of 
		maps 
		\[R_{S \cup Q, ab} \to 
		\TT_Q \overset{h_{U_1(Q), \sigma}}{\to} \cO\] satisfies 
		$\ord_\varpi(f_{Q,\sigma}) \le c$.
	\end{enumerate}
\end{prop}
\begin{proof}
	We set \[\overline{f}_Q = \left(\prod_{v \in Q}  \prod_{1 \leq i < j \leq 
	n} 
	(t_{v,i}(\varpi_v) - t_{v,j}(\varpi_v))^{n!}\right) \in \TT_Q^Q\] and let 
	$f_Q$ be a pre-image of $\overline{f}_Q$ in $R_{S \cup Q,ab}$ (such a 
	pre-image exists by Lemma \ref{lem_abelian_pseudocharacter}). It follows 
	from Proposition \ref{prop_uniform_pseudo-isomorphism} that $f_Q$ kills the 
	kernel and cokernel of $\eta_{Q,m}$ for all $m \le N$. If we take $c = n!d$ 
	then, again using Lemma \ref{lem_abelian_pseudocharacter}, we see that the 
	second part of the proposition is satisfied.
\end{proof}

We give one last piece of structure. Suppose fixed an ordering $Q =  \{ v_1, 
\dots, v_q \}$ and for each $v \in Q$ a surjection $\bbZ_p \to k(v)^\times(p)$. 
This data determines a surjection $(\bbZ_p^n)^q \to \prod_{v \in Q} 
k(v)^\times(p)^n = \Delta_Q$, hence a surjective algebra homomorphism $S_\infty 
\to \cO[\Delta_Q]$, where $S_\infty = \cO\llbracket y^{(i)}_1,\ldots,y^{(i)}_q: 
1 \le i \le n \rrbracket$. The group $W_Q = \prod_{v \in Q} S_n$ acts on 
$S_\infty$ by permutation of co-ordinates, and the invariant subring 
$S_\infty^{W_Q}$ may be identified with $\cO \llbracket e^{(i)}_1, \ldots, 
e^{(i)}_q : 1 \leq i \leq n \rrbracket$, where $e^{(i)}_j$ is the $i$th 
elementary 
symmetric polynomial in $y^{(1)}_j, \dots, y^{(n)}_j$. The ring $S_\infty^{W_Q}$ 
also has a role to play as a consequence of the following easy lemma:
\begin{lemma}
	The functor of deformations of the trivial pseudocharacter of $\bbZ_p$ of 
	dimension $n$ is represented by $\cO \llbracket X_1, \dots, X_n 
	\rrbracket^{\mathfrak{S}_n}$, with the universal characteristic polynomial 
	$\chi(t)$ of $1 \in \Zp$ given by $\prod_{i=1}^n((t-1) - X_i)$.
\end{lemma}
\begin{proof}
	Indeed, a residually trivial pseudocharacter of $\bbZ_p$ of dimension $n$ 
	over a ring $A \in \cC_\cO$ is precisely a point of $(\GL_n 
	\dslash \GL_n)(A)$ lying over the image of the identity in $(\GL_n \dslash 
	\GL_n)(k)$. Now we use the identification of the adjoint quotient $\GL_n 
	\dslash \GL_n$ with the quotient of the diagonal maximal torus by the Weyl 
	group. The universal deformation is given by the orbit of the matrix 
	$\diag(1+X_1,\ldots,1+X_n)$.
\end{proof}
Consequently, there is a homomorphism $S_\infty^{W_Q} \to R_{S \cup Q, ab}$, classifying the pullback of the tuple $(\gamma_v)_{v \in Q}$ to a tuple of $n$-dimensional pseudocharacters of the group $\bbZ_p$. There is also a homomorphism $S_\infty^{W_Q} \to \TT_Q^Q$, classifying the pullback of $(\alpha_v)_{v \in Q}$ to a tuple of $n$-dimensional pseudocharacters of the group $\bbZ_p$. This coincides with the restriction to $S_\infty^{W_Q}$ of the homomorphism $S_\infty \to \TT_Q^Q$ determined by the $\cO[\Delta_Q]$-algebra structure on $\TT_Q^Q$. Lemma 
\ref{lem:twlgc} has the following corollary.
\begin{lemma}
	The map $S_\infty^{W_Q} \to \TT_Q^Q$ factors through $\bT_Q$, and the map $R_{S \cup Q, ab} \to \bT_Q$ is a homomorphism of $S_\infty^{W_Q}$-algebras.  
\end{lemma}

\subsection{The patching argument}\label{subsec_patching_argument}
\begin{itemize}
	\item Fix $q = \mathrm{corank}_\cO H^1 (F_S / F^+,\ad 
	\rho(1)\otimes_{\cO} E / \cO)$. Applying
	Corollary \ref{cor:TWGalois}, we fix for each $N \geq 1$ a Taylor--Wiles datum $Q_N$ of level $N$, and we write $\Delta_N = \Delta_{Q_N}$, 
	$\ba_N= \ba_{Q_N}$, $R_N = 
	R_{S \cup Q_N, ab}$, $\TT_N = \TT_{Q_N}$. We set $R_0 = R_S$. We set $\q_N 
	= \ker(R_N \xrightarrow{h_{U_1(Q_N),\sigma}} \cO)$ and $\q_0 = \ker(R_0 
	\xrightarrow{h_{U,\sigma}} \cO)$. Thus $\q_N$ is the pre-image of $\q_0$ 
	under the 
	natural map $R_N \to R_0$.
	\item Let $S_\infty = \cO\llbracket y^{(i)}_1,\ldots,y^{(i)}_q: 1 \le i \le n \rrbracket$ and 
	fix orderings of $Q_N$ and generators of $k(v)^\times(p)$ for all $N$ and 
	all $v \in Q_N$ and thus 
	surjective maps $S_\infty \to \cO[\Delta_N]$. Let $\ba_\infty \subset 
	S_\infty$ be the augmentation ideal (equal to the inverse image of $\ba_N$ 
	under each of the previously defined maps).
	
	\item We moreover fix uniformisers 
	$\varpi_v$ for all $v \in Q_N$ (for every $N$). This allows us to think of 
	the 
	pseudocharacters $\gamma_v$ as pseudocharacters of $k(v)^\times(p) 
	\times \ZZ$. Recalling that we have fixed a generator of 
	$k(v)^\times(p)$ and an ordering on $Q_N$, for every 
	$N$ we have a $q$-tuple  $(\gamma_{N, 1}, \dots, 
	\gamma_{N, q})$ of $n$-dimensional pseudocharacters of $(\ZZ_p 
	\times \ZZ)$ with 
	coefficients in $R_N$. 
	\item We have actions of $t_{v,i}(\varpi_v)$ on 
	$S_\lambda(U_0(Q_N),\cO)_{\ffrm_{0, Q}}$ and 
	$S_\lambda(U_1(Q_N),\cO)_{\ffrm_{1,Q}}$ for each $v 
	\in Q_N$ and $i = 1,\ldots,n$. Using these actions, together with the fixed 
	orderings on $Q_N$, we obtain an action of the algebra
	\[A = \otimes_{j=1}^q \cO[(t_j^{(1)})^{ \pm 1}, \ldots, 
	(t_j^{(n)})^{ \pm 1}]\]
on these spaces, together with an 
	identification of $A$ with $A_{Q_N}$ sending $t_j^{(i)}$ to 
	$t_{v_j,i}(\varpi_{v_j})$. We have characters $\alpha_{j}^{(i)} : \bZ_p \times \bZ \to (S_\infty \otimes_\cO A)^\times$ for $i = 1, \dots, n$ and $j = 1, \dots, q$.   	By Lemma \ref{lem_abelian_pseudocharacter}, the pushforward of the pseudocharacter $\alpha_{j} = \tr \alpha_j^{(1)} \oplus \dots \oplus \alpha_j^{(n)}$ to $\End_\cO(S_\lambda(U_1(Q_N),\cO)_{\ffrm_{1,Q}})$ takes values in $\bT_N$ and equals the pushforward of $\gamma_{N, j}$ there.
	\item We can identify all the Weyl groups $W_{Q_N}$ (using our fixed 
	orderings of $Q_N$ for each $N$). We denote them all by $W$. 
	There is a natural $W$-action on $S_\infty$, compatible with the maps to 
	$\cO[\Delta_N]$. The invariants $S_\infty^{W}$ are a regular local 
	$\cO$-algebra, with $S_\infty$ a finite free $S_\infty^{W}$-algebra 
	($S_\infty^{W}$ is a power series algebra over the elementary symmetric 
	polynomials in $y_j^{(1)},\ldots,y_j^{(n)}$ for each $j$). Write 
	$\ba_\infty^{W} = \ba_\infty \cap S_\infty^W$, this is the ideal of 
	$S_\infty^W$ generated 
	by the elementary symmetric polynomials. There is also a natural $W$-action 
	on $A$. 
	
 \item We let $g = n q$ and $R_\infty = \cO\llbracket x_1,\ldots,x_g \rrbracket$, and let 
 $\q_\infty = (x_1,\ldots,x_g) \in \Spec(R_\infty)$. For each $N$ we 
 have a map $R_\infty \to R_N$ such that $\q_\infty R_N \subset \q_N$ and  $\q_N/(\q_N^2,\q_\infty)$ is killed by 
 a power of $\varpi$ which is independent of $N$. 
	
	\item Fix a non-principal ultrafilter $\cF$ on $\NN$, and let $\bR = \prod_{N \in \mathbf{N}} \cO$. If $I \in \cF$, then we define $e_I = (\delta_{N \in I})_{N \in \mathbf{N}} \in \bR$. Then $S = \{ e_I \mid I \in \cF \}$ is a multiplicative subset of $\bR$, and we define $\bR_\cF = S^{-1} \bR$. The natural map $\bR \to \bR_\cF$ is surjective and factors through the projection $\prod_{N \geq 1} \cO \to \prod_{N\ge m}\cO$ for any $m \ge 1$. The ring $\bR_\cF$ can also be described as the localization of $\bR$ at the prime ideal $\{ (x_N)_{N \in \mathbf{N}} \mid \exists I \in \cF, \forall N \in I, x_N \in \varpi \cO \}$.
\end{itemize}

\begin{defn}
	We define
	\begin{itemize}
		\item $M_1 = 
		\varprojlim_m\left(\bR_{\cF}\otimes_{\bR}\prod_{N\ge 
		m}\left(S_\lambda(U_1(Q_N),\cO)_{\m_{1, Q_N}}/\m_{S_\infty}^m\right) \right) $
		\item $M_0 = 
		\varprojlim_m\left(\bR_{\cF}\otimes_{\bR}\prod_{N\ge 
			m}S_\lambda(U_0(Q_N),\cO/\varpi^m)_{\m_{0, Q_N}} \right) $
		\item $M = 
		\varprojlim_m\left(\bR_{\cF}\otimes_{\bR}\prod_{N\ge 
			m}S_{\lambda}(U,\cO/\varpi^{m})_{\m}\otimes_{A_{Q_N}^W}A_{Q_N}\right)$
	\end{itemize}
\end{defn}
Here $A_{Q_N}^W$ acts on $S_{\lambda}(U,\cO/\varpi^{m})_{\m}$ via the spherical 
Hecke algebra action at places in $Q_N$.  We note that we 
naturally obtain compatible actions of $A$ on $M$, $M_0$ and $M_1$. Identifying 
$S_\lambda(U,\cO)_\m$ with 	
$\varprojlim_m\left(\bR_{\cF}\otimes_{\bR}\prod_{N\ge 
	m}S_{\lambda}(U,\cO/\varpi^{m})_{\m}\right)$, we equip 
	$S_\lambda(U,\cO)_\m$ with an $A^W$ action ($A^W$ acts on the $N$ factor in 
	the product via its identification with $A_{Q_N}^W$) and we see that we 
	have a natural isomorphism $M \cong S_\lambda(U,\cO)_\m\otimes_{A^W}A$. 

\begin{lem}\phantomsection\label{lem:patchedmodules}
	\begin{enumerate}
		\item $M_1$ is a flat $S_\infty$-module.
		\item The trace maps induce $M_1/\ba_\infty \cong 
		M_0$. 
		\item We have a map $\eta : M \to 
		M_0$ induced by the $\eta_{Q_N,m}$, which has kernel and 
		cokernel killed by $f$, where $f = (f^2_{Q_N}) \in 
		\prod_{N\in \NN} R_N$ and $f_{Q_N}$ is as in the statement of Proposition \ref{prop:oldforms}.
	\end{enumerate}
\end{lem}
\begin{proof}
For the first two parts we apply \cite[Lemma 4.4.4(2)]{Lue} (see also 
\cite[\href{https://stacks.math.columbia.edu/tag/0912}{Tag 
0912}]{stacks-project}): it suffices to 
prove 
that for each $m$ 
\[M_{1,m} = \bR_{\cF}\otimes_{\bR}\prod_{N\ge 
	m}\left(S_\lambda(U_1(Q_N),\cO)_{\m_{1, Q_N}}/\m_{S_\infty}^m\right) \] is a flat 
	$S_\infty/\m_{S_\infty}^m$-module, the natural transition maps 
	$M_{1,m+1} \to M_{1,m}$ induce isomorphisms $M_{1,m+1}/\m_{S_\infty}^m 
	\cong M_{1.m}$ and the trace maps induce $M_{1,m}/\ba_{\infty} \cong 
	M_{0,m}$. 
	
	Flatness of $M_{1,m}$ follows from flatness of $S_\lambda(U_1(Q_N),\cO)_{\m_{1, Q_N}}$ 
	over 
	$\cO[\Delta_N]$ (note that $S_\infty/\m_{S_\infty}^N$ is a quotient of 
	$\cO[\Delta_N]$). 
	
	We have \[M_{1,m+1}/\m_{S_\infty}^m = 
	\bR_{\cF}\otimes_{\bR}\prod_{N\ge 
		m+1}\left(S_\lambda(U_1(Q_N),\cO)_{\m_{1, Q_N}}/\m_{S_\infty}^m\right)  = 
		M_{1,m} \] and 
		\begin{multline*} M_{1,m}/\ba_{\infty} = 
	\bR_{\cF}\otimes_{\bR}\prod_{N\ge 
	m}\left(S_\lambda(U_1(Q_N),\cO)_{\m_{1, Q_N}}/(\ba_\infty+\m_{S_\infty}^m)\right)  \\ = 
\bR_{\cF}\otimes_{\bR}\prod_{N\ge 
m}\left(S_\lambda(U_0(Q_N),\cO/\varpi^m)_{\m_{0, Q_N}}\right) \end{multline*} (see 
		\cite[Lemma 4.5.9]{Lue} for the first equalities).
	
	The third part follows from \cite[Lemma 4.5.12]{Lue}.
\end{proof}

Now we can define a patched pseudodeformation ring:
\begin{defn}
	For $m \ge 1$ we define $R_m^{\prm} = \RR_\cF\otimes_{\RR} \prod_{N \ge 
	1}R_N/(\m_{R_N}f_{Q_N})^m$ and then define $R^{\prm} = \varprojlim_m 
	R_m^{\prm}$.
\end{defn}
\begin{lem}\label{lem:Galoisboundedlength}
	For each $m \ge 1$ there is an integer $n(m)$ (independent of $N$) such 
	that $(\m_{R_N}f_{Q_N})^{n(m)}$ annihilates 
	$S_\lambda(U_1(Q_N),\cO)_{\m_{1, Q_N}}/\m_{S_\infty}^m $ for all $N \ge m$.
\end{lem}
\begin{proof}
	By considering the $\ba_\infty$-adic filtration on 
	$S_\lambda(U_1(Q_N),\cO)_{\m_{1, Q_N}}/\m_{S_\infty}^m $ it suffices to prove that 
	there is an 
	integer $n(m)$ (independent of $N$) such that $(\m_{R_N}f_{Q_N})^{n(m)}$ 
	annihilates 
	$S_\lambda(U_0(Q_N),\cO/\varpi^m)_{\m_{0, Q_N}}$ for all $N \ge m$. 
	
	Since 
	$f_{Q_N}S_\lambda(U_0(Q_N),\cO/\varpi^m)_{\m_{0, Q_N}}$ is a finite length 
	$\cO$-module with length bounded by $q n!$ times that of $S_\lambda(U,\cO/\varpi^m)_\m$, 
	its length as an $R_N$-module is bounded independently of $N$ and therefore 
	it is annihilated by a $\m_{R_N}^{n(m)}$ for some $n(m)$ independent of $N$.
\end{proof} 
It follows from Lemma \ref{lem:Galoisboundedlength} that $R^{\prm}$ acts on 
$M_1$ (this is why $R^{\prm}$ is defined the way it is). 
We are going to use \cite[Lemma 4.5.3]{Lue} a few times, so we restate it here:
\begin{lem}\label{lem:lueprodlim}
Suppose for $i \in \NN$, $M_i$ is an $\cO$-module equipped with a decreasing 
filtration by $\cO$-modules $M_i \supset M_{i,1} \supset M_{i,2} \cdots$. Then 
the natural map
\[\prod_{i \ge 1} M_i \to \varprojlim_m\left(\RR_\cF\otimes_\RR \prod_{i \ge 1} 
M_i/M_{i,m} \right) \] is surjective with kernel given by elements of the form 
$(m_i)$ such that for each $m$ there exists $I_m \in \cF$ with $m_i \in 
M_{i,m}$ 
for all $i \in I_m$.
\end{lem}

We have a natural map 
$\prod_{N \ge 1} R_N \to R^{\prm}$ which is surjective by Lemma  
\ref{lem:lueprodlim}. 
We also have a natural map $R^{\prm} \to R_0$ given by 
taking the limit over $m$ of \[R_m^{\prm} \to  \RR_\cF\otimes_{\RR} \prod_{N 
\ge 
	1}R_0/(\m_{R_0})^m = R_0/(\m_{R_0})^m.\] 
\begin{lem}
The map $R^{\prm} \to R_0$ we have just defined is surjective.
\end{lem}	
\begin{proof}
We again apply Lemma \ref{lem:lueprodlim}: this implies that the natural map 
\[\prod_{N 
\ge 
1}R_N \to \varprojlim_m \left(\RR_\cF\otimes_\R \prod_{N\ge 1} 
R_0/(\m_{R_0})^m\right)\] is surjective; on the other hand 
it factors 
through our map $R^{\prm} \to R_0$.
\end{proof}

From the $n$-dimensional pseudorepresentations $(\gamma_{N, j})_{N \ge 1}$ ($j 
= 1, \dots, q$) with coefficients in 
$\prod_{N \ge 1}R_N$, we obtain $n$-dimensional pseudorepresentations 
$\gamma_{\infty, j}$ ($j = 1, \dots, q$) of 
$\Z_p\times \ZZ$ with 
coefficients in $R^{\prm}$. On the other hand, $M_1$ has a natural structure of 
$S_\infty \otimes_\cO A$-module, and we have defined characters 
$\alpha_{j}^{(i)} : \bZ_p \times \bZ \to (S_\infty \otimes_\cO A)^\times$ and 
pseudocharacters $\alpha_j = \tr \alpha_{j}^{(1)} \oplus \dots \oplus 
\alpha_j^{(n)}$.

\begin{lem}\label{lem:patchedtwlgc}
	Fix $1 \leq i \leq q$.
	\begin{enumerate}
		
		\item Composing $\gamma_{\infty, i}$ with the map $R^{\prm} \to R_0$ 
		gives a pseudorepresentation which is inflated from the `unramified quotient'
		$ \Z_p\times \ZZ \to \bZ$ (i.e.\ projection to second factor).
		\item The composite of $\gamma_{\infty, j}$ with the map $R^{\prm} \to 
		\End(M_1)$ equals the composite of $\alpha_{j}$ with the map $S_\infty 
		\otimes_\cO A \to \End(M_1)$. Consequently, the map $R^{\prm} \to 
		\End(M_1)$ is a homomorphism of $S_\infty^W$-algebras.
	\end{enumerate}
	
\end{lem}
\begin{proof}
	The first part follows from the analogous statement for each of the 
	pseudorepresentations $\gamma_{N, i}$ (which holds because the 
	pseudorepresentations classified by $R_0$ are unramified at places 
	in $Q_N$). The second part follows from Lemma \ref{lem:twlgc}.
\end{proof}
	
\begin{defn}
	We let $\q^{\prm}$ be the prime ideal in $R^{\prm}$ given by the inverse 
	image of $\q_0 \subset R_0$.
\end{defn}

\begin{lem}\label{lem:imageideal}
	The image of $\prod \q_N \subset \prod_{N \ge 1} R_N$ in $R^{\prm}$ is 
	equal to $\q^{\prm}$.
\end{lem}
\begin{proof}
Write $I$ for the kernel of $\prod_{N \ge 1} R_N \to R^{\prm}$ and $I'$ for the 
image of $I$ in $\prod_{N\ge 1} R_N/\q_N= \prod_{N \ge 1} \cO$. It suffices to 
prove that the map $\prod_{N\ge 1}R_N \to R_0$ induces an isomorphism 
\[\left(\prod_{N\ge 1}R_N\right)/ (I,\prod \q_N) = \left(\prod_{N\ge 
1}\cO\right)/ I' \cong R_0/\q_0 
= \cO.\]

The ideal $I$ is the set of elements $(x_N) \in \prod R_N$ 
such that for each $m \ge 1$ there exists $I_m \in \cF$ with $x_N \in 
(\m_{R_N}f_{Q_N})^m$ for all $N \in I_m$. We have $(\ffrm_{R_N} f_{Q_N})^m + \q_N = (\varpi^m f_{Q_N}^m) + \q_N \subset (\varpi^m) + \q_N$. It follows that $I'$ is contained in the kernel of the map  $\prod_{N \geq 1} \cO \to \cO$. We need to show that $I'$ equals the kernel. To this end, choose a tuple of elements $(y_N) \in \prod_{N \geq 1} \cO$ which does lie in the kernel. Recall that there is a constant $c$ 
such that the image of 
$f_{Q_N}$ in 
$R_N/\q_N = \cO$ has $\varpi$-adic 
valuation $\le c$. Let $I_m = \{ N \geq 1 \mid \ord_\varpi y_N \geq m(c+1) \}$. 
Then $I_1 \supset I_2 \supset I_2 \supset \dots$ and $\cap_{m \geq 1} I_m = 
\emptyset$. Moreover, each $I_m$ is in $\cF$ (since $(y_N)$ is in the kernel of 
the map to $\cO$). 

We have $(\ffrm_{R_N} f_{Q_N})^m + \q_N = (\varpi^m f_{Q_N}^m) + \q_N$, and 
this contains $(\varpi^{(c+1)m}) + \q_N$. Therefore we can for each $m \geq 
1$ and $N \in I_m$ find an element $x_{N, m} \in (\ffrm_{R_N} f_{Q_N})^m$ such 
that $x_{N, m} + 
\q_N 
= y_N$. We define a tuple $(x_N)_{N \geq 1} \in \prod_{N \geq 1} R_N$ by taking 
$x_N$ to be an arbitrary pre-image of $y_N$ if $N \not\in I_1$ and $x_N = x_{N, 
m}$ if $N \in I_m - I_{m+1}$. Then $(x_N)$ lies in $I$ and its image in 
$\prod_{N \geq 1} \cO$ equals $(y_N)$, as required.
\end{proof}

\begin{lem}\label{lem:imagepowerideal}
\begin{enumerate} \item We have an equality of ideals $\prod_{N\ge 1}\q_N^m = 
\left(\prod_{N\ge 1}\q_N\right)^m$ in $\prod_{N \ge 1} R_N$.
\item For $m \ge 1$ the image of $\prod_{N\ge 1}\q_N^m$ in $R^{\prm}$ is equal to  
$(\q^{\prm})^m$.
\end{enumerate}
\end{lem}
The possibility of proving a statement like this one is mentioned in \cite[Remark 4.6.10]{Lue}.
\begin{proof}
It suffices to prove the first part. Recall from the 
proof of Corollary \ref{cor:TWGalois} that there exists an integer 
$g_0$ such that for every $N$ there exists a surjection 
$\cO\llbracket x_1,\ldots,x_{g_0}\rrbracket\to R_N$ such that the images of the $x_i$ are in 
$\q_N$ 
(since $\q_N/\q_N^2$ can be generated by $g_0$ elements). Now it suffices to 
prove that we have an equality of ideals $\prod_{N\ge 1}(x_1,\ldots,x_{g_0})^m 
= \left(\prod_{N\ge 1}(x_1,\ldots,x_{g_0})\right)^m$ in $\prod_{N\ge 
1}\cO\llbracket x_1,\ldots,x_{g_0} \rrbracket$. We conclude using the fact that for any 
ring $R$ 
and any ideal $I \subset R$ we have $I^m \prod_{N\ge 1}R = \left(I\prod_{N\ge 
	1}R\right)^m$, and we also have $I \prod_{N \ge 1} R = 
\prod_{N\ge 1} I$ when $I$ is finitely generated.
\end{proof}
For the statement of the next proposition, we recall that our data includes, for each $N \geq 1$, a map $R_\infty = \cO \llbracket x_1, \dots, x_g \rrbracket \to R_N$ sending the ideal $\q_\infty = (x_1, \dots, x_g)$ to $\q_N$. The diagonal map $R_\infty \to \prod_{N\ge 1} R_N$ induces a map 
$R_\infty 
\to 
R^{\prm}$ which sends $\q_\infty$ into $\q^{\prm}$. 
\begin{prop}\phantomsection\label{prop:RinftytoRpatched}
	\begin{enumerate}
		\item 	The $\cO$-module
		\[\q^{\prm}/((\q^{\prm})^2,\q_\infty)\] is killed by $\varpi^c$, where 
		$c$ is as in Corollary \ref{cor:TWGalois}. 
		\item The natural map on completed local rings
		\[(R_\infty)~\widehat{}_{\q_\infty} \to 
		(R^{\prm})~\widehat{}_{\q^{\prm}} \]
		is surjective. In particular, 	$(R^{\prm})~\widehat{}_{\q^{\prm}}$ is a complete Noetherian local $E$-algebra with residue field $E$.
	\end{enumerate}
\end{prop}
\begin{proof}
	It 
	follows from Corollary \ref{cor:TWGalois} that the 
	cokernel of the map 
	\[\prod_{N \ge 1}\q_\infty/(\q_\infty)^2 \to \prod_{N \ge 1}\q_N/(\q_N)^2  
	\] 
	is killed by $\varpi^c$.  Applying Lemmas \ref{lem:imageideal} and 
	\ref{lem:imagepowerideal}, it remains 
	to show that the image of $\q_\infty/(\q_\infty)^2$ in 
	$\q^{\prm}/(\q^{\prm})^2$ is the same as the image of $\prod_{N \ge 
	1}\q_\infty/(\q_\infty)^2$. This is done as in the proof of  
	\cite[Prop.~4.6.16]{Lue}:  it suffices to 
	show that the composition of maps \[\q_\infty/(\q_\infty)^2 \to \prod_{N 
	\ge 
		1}\q_\infty/(\q_\infty)^2 \to \cO \otimes_{\prod_{N\ge 1} \cO} \prod_{N \geq 1} 
		\q_\infty/(\q_\infty)^2  \] is surjective. Here the first map is the 
		diagonal embedding and we regard $\cO$ as a $\prod_N\cO$-algebra via 
		the map $\prod_N\cO\to R^{\prm}/\q^{\prm} \cong \cO$. We conclude using 
		Lemma \ref{easyalgebra}.
		
		This shows the first part of the proposition. For the second, we see that the first part implies that each of the maps
		\[ g_i : (R_\infty / \q_\infty^i)_{\q_\infty} \to (R^\prm / (\q^\prm)^i)_{\q^\prm} \]
		is surjective. To check that $\varprojlim_i g_i$ is surjective, it is enough to note that the sequence $(\ker g_i)_{i \geq 1}$ satisfies the Mittag-Leffler condition (because each of these ideals has finite length, being contained in an Artinian local ring).
\end{proof}
\begin{lem}\label{easyalgebra}
Let $R$ be a commutative ring and $M$ a finitely generated $R$-module. Suppose 
we have a $R$-algebra map $\prod_{N\ge 1} R \overset{\lambda}{\to} R$. Then the 
composite map \[M \to \prod_{N\ge 1}M\to R \otimes_{\prod_{N \geq 1} R} \prod_{N \ge 1} 
M \] is surjective.
\end{lem}
\begin{proof}
If $M$ is finite free over $R$, then $\prod_{N \ge 1}M$ has a $(\prod_N 
R)$-basis given 
by diagonally embedded basis elements for $M$, and the statement is clear. In 
general, we write $M$ as a quotient of a finite free $R$-module $F$. The 
composition of $F \to  R \otimes_{\prod_{N \geq 1} R} \prod_{N\ge 1} F \to R 
\otimes_{\prod_{N \geq 1} 
R} \prod_{N\ge 1} M$ is surjective and factors through $M$.
\end{proof}

\begin{rem}
	Note that we have not shown that $\q^{\prm}$ is finitely generated, so we 
	rely on the comparison with $R_\infty$ to show that $\q^{\prm}$-adic 
	completion  
	$(R^{\prm})~\widehat{}_{\q^{\prm}}$ is $\q^{\prm}$-adically 
	complete! 
\end{rem}

Now we are going to consider the modules:

\begin{itemize}
	\item $\mrm_1 = (M_1/\ba_{\infty}^2)_{\q^{\prm}}$
	\item $\mrm_0 = (M_0)_{\q^{\prm}}$
	\item $\mrm = M_{\q^{\prm}} = M_{\q_{0}}$.
\end{itemize}

\begin{lem}\phantomsection\label{lem:artinianpatched}\begin{enumerate}
		\item $\mrm_1$ is a finite free 
		$S_{\infty,\ba_{\infty}}/\ba_{\infty}^2$-module. 
		\item The 
trace maps induce  an isomorphism
$\mrm_1/\ba_{\infty} \cong \mrm_0$.
\item The map $\eta$ induces an 
isomorphism $\eta: \mrm \cong \mrm_0$.\end{enumerate}
\end{lem}
\begin{proof}
	We start with the third part: this follows immediately from Lemma 
	\ref{lem:patchedmodules}, since by Proposition \ref{prop:oldforms} the 
	image of $f$ in $R^{\prm}$ is not in 
	$\q^{\prm}$. The second part also follows immediately from Lemma 
	\ref{lem:patchedmodules}. It remains to show the first part. Since the 
	inverse image of $\q^{\prm}$ in $S_{\infty}^W$ is $\ba_{\infty}^W$, the 
	action of 
	$S_\infty$ on $\mrm_1$ factors through the localisation 
	$S_\infty\otimes_{S_\infty^W}(S_\infty^W)_{\ba_\infty^W} = 
	S_{\infty,\ba_{\infty}}$ (note that $\ba_\infty$ is the unique point of 
	$\Spec(S_\infty)$ in the pre-image of $\ba_\infty^W$ under the finite map 
	$\Spec(S_\infty) \to \Spec(S_\infty^W)$). We know from Lemma 
	\ref{lem:patchedmodules} that $M_1/\ba_\infty^2$ is a flat 
	$S_\infty/\ba_\infty^2$-module, so the localisation $\mrm_1$ is a flat 
	$S_\infty/\ba_\infty^2$-module, and hence a flat 
	$S_{\infty,\ba_{\infty}}/\ba_{\infty}^2$-module. Since $\mrm_1/\ba_\infty$ 
	is finite dimensional (combining the second and third parts), $\mrm_1$ is 
	finitely generated over the Artinian local ring 
	$S_{\infty,\ba_{\infty}}/\ba_{\infty}^2$.
\end{proof}

Since $\mrm_1$ is a finite dimensional $E$-vector space, the action of the 
local 
$E$-algebra $(R^{\prm})_{\q^{\prm}}$ factors through an action by (an Artinian 
quotient of) $(R^{\prm})~\widehat{}_{\q^{\prm}}$. It follows from the third 
part of Lemma \ref{lem:artinianpatched} that the action of 
$(R^{\prm})~\widehat{}_{\q^{\prm}}$ on $\mrm_0 \cong \mrm$ factors through the 
composition of surjective maps \begin{equation}\label{patchedRtoT} 
(R^{\prm})~\widehat{}_{\q^{\prm}} \to 
(R_{0})~\widehat{}_{\q_{0}} \to 
(\TT_{\emptyset})_{\q_{0}} = E  \end{equation}

Now we consider again our pseudorepresentations $\gamma_{\infty, j}$ ($1 \leq j 
\leq q$) of 
$\Zp\times\ZZ$ with coefficients in $R^{\prm}$.
\begin{defn}
 For $1 \leq j \leq q$, we let $\delta_j \in R^{\prm}$ denote the discriminant 
 of the characteristic polynomial $\chi_j(t) \in R^{\prm}[t]$ of $(0, 1) \in 
 \Z_p \times \Z$ under the pseudorepresentation $\gamma_{\infty, j}$.
\end{defn}
\begin{lemma}\label{lem:distinctroots}
	For $1 \leq j \leq q$, $\delta_j \notin \q^{\prm}$. Moreover, $\chi_j(t) 
	\text{ mod }\q^\prm$ splits into linear factors in $E[t]$.
\end{lemma}
\begin{proof}
	To show that $\delta_j \neq 0$, it suffices to show that for some $m \ge 1$ the image of $\delta_j$ under 
	the composition \[R^{\prm} \to R_0 \overset{h_{U, \sigma}}{\to} \cO \to 
	\cO/\varpi^m\] is 
	non-zero. Recall the constant $d$ from Lemma \ref{lem:findTWprimes}. Choose 
	$m > dn(n-1)$. Then it follows from  Lemma \ref{lem:findTWprimes} that we 
	will be done if we can identify the image of $\delta_j$ in $\cO/\varpi^m$ 
	with the image of the discriminant of the characteristic polynomial of a 
	Frobenius element $\sigma_{\wv}$ for some $v \in Q_N$. Choose $m'$ so that 
	our 
	map $R_0 \to \cO/\varpi^m$ factors through 
	$R_0/\m_{R_0}^{m'}$. Now we can identify the image of 
	$\delta_j$ in $R_0/\m_{R_0}^{m'}$ with the image of an 
	element $(\delta_{j,N})_{N \ge 1} \in \prod_{N\ge 
	1}R_0/\m_{R_0}^{m'}$ in $\RR_\cF\otimes_{\RR}\prod_{N\ge 
	1}R_0/\m_{R_0}^{m'}$, where $\delta_{j,N}$ is the image of  
	the discriminant for the Frobenius element at the $j$th element of $Q_N$. 
	We 
	deduce that the image of $\delta_j$ in $R_0/\m_{R_0}^{m'}$ 
	coincides with one of these Frobenius discriminants.
	
	The same argument shows that the image of $\chi_j(t) \text{ mod }\q^\prm$ 
	splits into linear factors in $\cO / \varpi^{m}[t]$ for all $m \geq 1$. 
	Indeed, for each $\sigma_\wv \in G_F$, the characteristic polynomial of $\rho(\sigma_\wv)$ has all of 
	its roots in $\cO$ (this is part of the definition of an enormous subgroup of $\GL_n(\cO)$). Hensel's lemma implies that $\chi_j(t)$ 
	itself factors in $\cO[t]$.
\end{proof} 
For each $j \in \{1,\ldots,q\}$ we fix an ordering 
$x_{j}^{(1)},\ldots,x_{j}^{(n)}$ of the (pairwise distinct) roots in $E$ of the polynomial $\chi_j(t) \text{ mod }\q^\prm$. For each $j$, we may consider the pseudorepresentation
$(\gamma_{\infty, j})_{\q^{\prm}}$ of $\Z_p\times \ZZ$ with coefficients in 
$(R^{\prm})~\widehat{}_{\q^{\prm}}$ given by composing $\gamma_{\infty, j}$ with the 
natural map $R^{\prm} \to (R^{\prm})~\widehat{}_{\q^{\prm}}$. This 
pseudorepresentation is residually multiplicity free.
\begin{lem}\label{lem:twpseudosplit}
	There is a unique collection of continuous characters $\gamma_{j}^{(i)}: \Z_p\times \ZZ \to 
	((R^{\prm})~\widehat{}_{\q^{\prm}})^\times$ ($i = 1, \dots, n$, $j =1 , \dots, q$) such that $\gamma_{j}^{(i)} \mod 
	\q^{\prm}$ is 
	the character $(a,b) \to (x_{j}^{(i)})^b$ and $(\gamma_{\infty, j})_{\q^{\prm}} = 
	\tr \gamma_{j}^{(1)} \oplus \cdots \oplus \gamma_{j}^{(n)}$.
\end{lem}
\begin{proof}
This follows from, e.g., \cite[Proposition 
1.5.1]{bellaiche_chenevier_pseudobook}, since in a commutative GMA we have 
(using the notation of \emph{loc.cit.}) $\mathcal{A}_{i,j}\mathcal{A}_{j,i} = 
\mathcal{A}_{j,i}\mathcal{A}_{i,j} \subset \mathcal{A}_{i,i} \cap 
\mathcal{A}_{j,j} = 0$ for $i \ne j$.
\end{proof}
The characters $\gamma_{j}^{(i)}|_{\bZ_p \times 0} : \bZ_p \to ((R^{\prm})~\widehat{}_{\q^{\prm}})^\times$ determine an extension of the homomorphism $S_\infty^W \to R^\prm$ to a homomorphism $S_\infty \to (R^{\prm})~\widehat{}_{\q^{\prm}}$. This 
	in turn	naturally extends to a map from
	the formally smooth $E$-algebra $(S_\infty)~\widehat{}_{\ba_\infty}$ and we 
	choose a lift of this through the surjective map (see Proposition 
	\ref{prop:RinftytoRpatched})	
	\[(R_\infty)~\widehat{}_{\q_\infty} \to 
	(R^{\prm})~\widehat{}_{\q^{\prm}}\] to equip 
	$(R_\infty)~\widehat{}_{\q_\infty}$ with a map from 
	$(S_\infty)~\widehat{}_{\ba_\infty}$.	We denote by $A'$ the localization 
	of $A$ at the prime ideal 
	$(t_j^{(i)}-x_{j}^{(i)}: 1 \le j \le q, 1 \le i \le n)$ and define
\begin{itemize}
	\item $\mrm_1' = \mrm_1 \otimes_{A} A'$
	\item $\mrm_0' = \mrm_0 \otimes_{A} A'$.
\end{itemize}
(We recall that the ring $A$, defined at the beginning of \S \ref{subsec_patching_argument}, is a Laurent polynomial ring in elements $(t_j^{(i)})^{\pm 1}$ ($j = 1, \dots, q$, $i = 1, \dots, n$) which represent the patched version of the Hecke operators $t_{v, i}(\varpi_v)^{\pm 1}$ for Taylor--Wiles primes $v$.)
\begin{remark}
	The above localization is our replacement for the usual `localization with 
	respect to a suitable eigenvalue of the $U_q$ operator' which appears in 
	the Taylor--Wiles method. We can only do this after patching and inverting 
	$p$ because we do not assume that $\rhobar(\sigma_{\wv})$ has distinct 
	eigenvalues for Taylor--Wiles places $v$.
\end{remark}

\begin{lem}
	\begin{enumerate}
	\item For each $i = 1, \dots, n$ and $j = 1, \dots, q$, the respective pushforwards of the characters $\alpha_j^{(i)}$, $\gamma_j^{(i)}$ to $\End(\mrm_1')$ are equal.
			\item The two structures of $S_\infty$-module on $\mrm_1'$ (the standard one, and the one induced by the homomorphism $S_\infty \to (R^{\prm})~\widehat{}_{\q^{\prm}}$ constructed above) are the same.
			\item The map $(R^{\prm})~\widehat{}_{\q^{\prm}} \to 
		(R_{0})~\widehat{}_{\q_{0}}$ factors through the 
		quotient 
		$(R^{\prm})~\widehat{}_{\q^{\prm}}/\ba_\infty$.
			\item The 
	trace maps induce  
	$\mrm_1'/\ba_{\infty} \cong \mrm_0'$.

\item $\mrm_1'$ is a finite free (non-zero)
	$S_{\infty,\ba_{\infty}}/\ba_{\infty}^2$-module. 

\end{enumerate}
\end{lem}
\begin{proof}
Let $X = \{ \alpha_j^{(i)}(z), \gamma_j^{(i)}(z) \mid z \in \bZ_p \times \bZ, i = 1, \dots, n$, $j = 1, \dots, q \}$. By construction, the elements of $X$ commute with each other; let $T$ denote the $E$-subalgebra of $\End(\mrm_1')$ generated by the elements of $X$. Then $T$ is an Artinian $E$-algebra. The pushforwards of the characters $\alpha_j^{(i)}$ and $\gamma_j^{(i)}$ take values in $T$ and the pseudocharacters $\tr \alpha_j^{(1)} \oplus \dots \oplus \alpha_j^{(n)}$ and $\tr \gamma_j^{(i)} \oplus \dots \oplus \gamma_j^{(n)}$ are equal after pushforward to $T$ for each $j = 1, \dots, q$. To show that the characters $\alpha_j^{(i)}$ and $\gamma_j^{(i)}$ are equal after pushforward to $T$ for each $j = 1, \dots, q$ it is enough (after the uniqueness assertion of \cite[Proposition 
1.5.1]{bellaiche_chenevier_pseudobook}) to show that they are equal after pushforward to each residue field of $T$. However, our construction shows that for each $i = 1, \dots, n$ and $j = 1, \dots, q$ the elements $\alpha_j^{(i)}(0, 1) - x_{j}^{(i)}$ and $\gamma_j^{(i)}(0, 1) - x_j^{(i)}$ are commuting nilpotent elements of $\End(\mrm_1')$ and therefore their difference $\alpha_j^{(i)}(0, 1) - \gamma_j^{(i)}(0, 1)$ lies in the Jacobson radical of $T$. This proves the first part of the lemma. The second is an immediate consequence since the two $S_\infty$-module structures are determined by the two sets of characters $\alpha_j^{(i)}$ and $\gamma_j^{(i)}$.

The third part of the lemma is equivalent to the assertion that characters $\gamma_j^{(i)}|_{\bZ_p \times 0}$ become trivial after pushforward along the map $(R^{\prm})~\widehat{}_{\q^{\prm}} \to 
		(R_{0})~\widehat{}_{\q_{0}}$. Since the pseudocharacter $\tr \gamma_j^{(1)} \oplus \dots \oplus  \gamma_j^{(n)}$ is residually multiplicity-free, the desired statement follows  from uniqueness and Lemma \ref{lem:patchedtwlgc}.
		
		The fourth part of the lemma follows from the same statement before localisation to $A'$ 
(Lemma \ref{lem:artinianpatched}). We now prove the final part of the lemma. Since $\mrm_1'$ is a 
direct summand of $\mrm_1$, we just need to prove that $\mrm_1'$ is non-zero, or 
indeed that $\mrm_0'$ is non-zero. For this, we note that it follows 
from Lemma~\ref{lem:patchedtwlgc} (compatibility of Galois and automorphic 
pseudocharacters) and the observation above (\ref{patchedRtoT}) that the 
characteristic polynomial 
$\prod_{i=1}^n 
(t-t_j^{(i)})$ of $(0,1)$ under $\alpha_j$ pushes forwards to 
$\prod_{i=1}^n(t-x_j^{(i)}) = 
\chi_j(t) \text{ mod }\q^\prm$ in $\End(\mrm_0)$.  It follows that $A^W$ acts 
on 
$\mrm_0$ via the map $A^W \to E$ induced by $t_j^{(i)} 
\mapsto x_{j}^{(i)}$. Since the $A$-module $\mrm_0$ is isomorphic to 
$S_\lambda(U,\cO)_{\q_0}\otimes_{A^{W}} A$, we deduce that the localisation 
$\mrm_0'$ is non-zero. 
\end{proof}

To complete the proof of this section's main theorem, we recall Brochard's 
freeness criterion:

\begin{theorem}[Theorem 1.1 of \cite{brochard}]
Let $A \to B$ be a local morphism of Noetherian local rings satisfying the 
inequality on embedding dimensions: \[\mathrm{edim}(B)\le \mathrm{edim}(A).\] 
Let $M$ be a non-zero $A$-flat $B$-module which is finitely generated over $B$. 
Then M is finite free over $B$.
\end{theorem}

\begin{theorem}
	The map $(R_{0})~\widehat{}_{\q^{0}} \to 
	(\TT_{\emptyset})_{\q_{0}} = E$ is an isomorphism, and as a 
	consequence we have \[H^1_f(F^+,\ad r_{\pi,\iota}) = 0.\]
\end{theorem}
\begin{proof}
	We apply Brochard's criterion with $A = 
	S_{\infty,\ba_\infty}/\ba_\infty^2$, $B = 
	(R_\infty)~\widehat{}_{\q_\infty}/\ba_\infty^2$, $M = \mrm_1'$. Note that 
	the embedding dimension of $S_{\infty,\ba_\infty}/\ba_\infty^2$ is $qn$ 
	and, since $(R_\infty)~\widehat{}_{\q_\infty}$ is a power 
	series ring over $E$ in $qn$ variables, the embedding dimension of 
	$(R_\infty)~\widehat{}_{\q_\infty}/\ba_\infty^2$ is $\le qn$.
	
	We conclude that $\mrm_1'$ is finite free over 
	$(R_\infty)~\widehat{}_{\q_\infty}/\ba_\infty^2$ and therefore $\mrm_0'$ is 
	finite free over $(R_\infty)~\widehat{}_{\q_\infty}/\ba_\infty$. Since the 
	action of $(R_\infty)~\widehat{}_{\q_\infty}$ on $\mrm_0'$ factors through the 
	action of $(\TT_{\emptyset})_{\q_{0}}$, we deduce that each of the 
	surjective maps \[(R_\infty)~\widehat{}_{\q_\infty}/\ba_\infty\to 
	(R^{\prm})~\widehat{}_{\q^{\prm}}/\ba_\infty \to 
	(R_{0})~\widehat{}_{\q_{0}} \to 
	(\TT_{\emptyset})_{\q_{0}} = E  \] are isomorphisms. The vanishing 
	of the adjoint Selmer group follows from the identification of this with 
	the reduced tangent space of 
	$(R_{0})~\widehat{}_{\q_{0}}$ (i.e.\ Proposition \ref{prop_generic_tangent_spaces}).
\end{proof}

\begin{remark}
	We find it convenient (or amusing) to use Brochard's freeness criterion 
	here, although 
	we could alternatively have worked with the 
	$(S_{\infty})~\widehat{}_{\ba_\infty}$-module $\varprojlim_m\left( 
	(M_1/\ba_\infty^m)_{\q^{\prm}}\right)$ in place of $m_1$ and concluded 
	using Auslander--Buchsbaum as in the work of Diamond and Fujiwara.
\end{remark}

\section{Applications}\label{sec:applications}

We now deduce our main theorems. We begin with a useful lemma.
\begin{lemma}\label{lem_enormousness_under_base_change}
	Let $F$ be a number field, and let $E / \bQ_p$ be a coefficient field. Let $\rho : G_F \to \GL_n(E)$ be a continuous representation which is unramified almost everywhere. Let $S$ be a finite set of places of $F$. Then we can find a finite set $T$ of places of $F$ with the following property:
	\begin{itemize}
		\item $T \cap S = \emptyset$.
		\item For any $T$-split finite extension $F' / F$, $\rho(G_{F'(\zeta_{p^\infty})}) = \rho(G_{F(\zeta_{p^\infty})})$.
	\end{itemize}
\end{lemma}
\begin{proof}
	After replacing $\rho$ by $\rho \oplus \epsilon$, it is enough to show we 
	can choose $T$ so that $\rho(G_F) = \rho(G_{F'})$ if $F' / F$ is $T$-split. 
	Conjugate $\rho$ so that it takes values in $\GL_n(\cO)$, and let $L_\infty 
	/ F$ be the extension cut out by $\rho$, $L_N / F$ the extension cut out by 
	$\rho_N = \rho \text{ mod }\varpi^N$. We have $\rho(G_F) = \varprojlim_N 
	\rho_N(G_F)$, so it's enough to show that we can choose $T$ so that if $F' 
	/ F$ is $T$-split, then $\rho_N(G_F) = \rho_N(G_{F'})$ for all $N \geq 1$.
	
	To this end, we let $M / F$ be the compositum of all of the extensions of $F$ cut out by simple quotients of $\Gal(L_N / F)$ (for any $N \geq 1$). The extension $M / F$ is finite, because simple quotients of $\Gal(L_N / F)$ (for varying $N \geq 1$) correspond to simple quotients of $\rho(G_F)$ by closed normal subgroups. Since $\rho(G_F)$ has a normal, closed subgroup of finite index which is a topologically finitely generated pro-$p$ group, these quotients are finite in number.
	
 We can therefore choose $T$ to be any set disjoint from $S$ and such that for each intermediate field $M / M' / F$ with $\Gal(M' / F)$ simple, there exists $v \in T$ which does not split in $M'$.  
\end{proof}
We prove a theorem for regular 
algebraic, cuspidal, polarized automorphic representations. First we treat the 
case of a CM base field. 
\begin{theorem}\label{thm_vanishing_over_CM_field}
	Let $F$ be a CM number field, and let $(\pi,\chi)$ be a regular 
	algebraic, cuspidal, polarized automorphic representation of 
	$\GL_n(\bA_F)$. Let $\iota : \overline{\bQ}_p \to \bC$ be an isomorphism, 
	and suppose that $r_{\pi, \iota}(G_{F(\zeta_{p^\infty})})$ is enormous. 
	Then $H^1_f(F^+, \ad r_{\pi, \iota}) = 0$.
\end{theorem}
\begin{proof}
	As in the proof of \cite[Theorem 1.2]{blght}, $\pi$ has a twist which is 
	polarized with respect to $\delta_{F / F^+}^n$ (i.e.~of unitary type). The 
	twist does not alter $\ad 
	r_{\pi, \iota}$, so we can assume that $\pi$ is of unitary type. For 
	any finite extension $F' / F^+$, the induced map
	\[ H^1_f(F^+, \ad r_{\pi, \iota}) \to H^1_f(F', \ad r_{\pi, \iota}) \]
	is injective. It is therefore enough to find a soluble totally real extension $L^+ / F^+$ with the following properties:
	\begin{itemize}
		\item Let $L = L^+ F$. Then $r_{\pi, \iota}(G_{L(\zeta_{p^\infty})}) = r_{\pi, \iota}(G_{F(\zeta_{p^\infty})})$.
		\item Let $\pi_L$ denote the base change of $\pi$ (which exists and is regular algebraic, after \cite[Ch. 3, Theorems 4.2, 5.1]{MR1007299}). It is cuspidal, because $r_{\pi, \iota}|_{G_L}$ is irreducible. Each place of $L$ at which $\pi_L$ is ramified, or dividing $p$, is split over $L^+$.
		\item For every place $w$ of $L$, $\pi_{L, w}$ has an Iwahori--fixed vector.
	\end{itemize}
    To achieve this, let $S$ be the set of places of $F^+$ dividing $p$ or above which $\pi$ is ramified, and let $S_F$ denote the set of places of $F$ lying above a place of $S$. Let $T_F$ denote a set as provided by Lemma \ref{lem_enormousness_under_base_change}, disjoint from $S_F$, and let $T$ denote the set of places of $F^+$ lying below a place of $T_F$. Then $S$ and $T$ are disjoint and if $L^+ / F^+$ is $T$-split, then $L / F$ is $T_F$-split. We can choose $L^+ / F^+$ to be any $T$-split soluble totally real extension which has the correct behaviour at the places in $S$, the existence of such an extension being a consequence of \cite[Lemma 4.1.2]{cht}. 
\end{proof}
Next we treat the case of a totally real base field $F$. We consider a regular 
algebraic, cuspidal, polarized automorphic representation $(\pi,\chi)$ of 
$\GL_n(\A_F)$.
Let $\iota : \overline{\bQ}_p \to \bC$ be an isomorphism, and suppose that 
$r_{\pi, \iota}$ is irreducible. Let $V$ denote the space on which $r_{\pi, 
\iota}$ acts. Then there is a unique $G_F$-equivariant pairing $\langle \cdot, 
\cdot \rangle : V \times V \to \epsilon^{1-n} r_{\chi, \iota}$, which is 
symmetric if $n$ is odd or $n$ is even and $r_{\chi, \iota}(c_v) = 1$ ($v | 
\infty$), or antisymmetric if $n$ is even and $r_{\chi, \iota}(c_v) = -1$ (see 
\cite{belchen} and \cite[\S2.1]{BLGGT}). We thus obtain a homomorphism
\[ r_{\pi, \iota} ' : G_F \to \operatorname{GS}(\langle \cdot, \cdot \rangle)(\overline{\bQ}_p) \]
to the general similitude group of the pairing $\langle \cdot, \cdot \rangle$. We write $\mathfrak{gs}$ for the Lie algebra of this reductive group over $\overline{\bQ}_p$. 
\begin{theorem}\label{thm_vanishing_over_real_field}
	Let $F$ be a totally real number field, and let $(\pi, \chi)$ be a regular 
	algebraic, cuspidal, polarized automorphic representation of 
	$\GL_n(\bA_F)$.  Let $\iota : \overline{\bQ}_p \to \bC$ be an isomorphism, 
	and suppose that $r_{\pi, \iota}(G_{F(\zeta_{p^\infty})})$ is enormous. 
	Then $H^1_f(F, \mathfrak{gs}) = 0$.
\end{theorem}
\begin{proof}
	This can be deduced from Theorem \ref{thm_vanishing_over_CM_field} using base change in the same way that \cite[Theorem B]{All16} is deduced from \cite[Theorem A]{All16}. We omit the details. 
\end{proof}
When $n = 2$, these results take a particularly simple form:
\begin{theorem}\label{thm_vanishing_for_HMF}
	Let $F$ be a totally real number field, and let $\pi$ be a regular algebraic, cuspidal automorphic representation of $\GL_2(\bA_F)$. Let $\iota : \overline{\bQ}_p \to \bC$ be an isomorphism. Suppose that one of the following holds:
	\begin{enumerate}
		\item $\pi$ does not have CM.
		\item $\pi$ has CM by an extension $K / F$, and $K \not\subset F(\zeta_{p^\infty})$.
	\end{enumerate}
	Then $H^1_f(F, \ad r_{\pi, \iota}) = 0$.
\end{theorem}
\begin{proof}
	When $n = 2$, $\mathfrak{gs} = \mathfrak{gl}_2$. Our result will follow 
	from Theorem \ref{thm_vanishing_over_real_field} if we can verify that our 
	hypotheses imply that $r_{\pi, \iota}(G_{F(\zeta_{p^\infty})})$ is 
	enormous. If $\pi$ does not have CM, this is example 
	\ref{ex:noncmenormous}. 
	
	Suppose instead that $\pi$ has CM by a CM quadratic extension $K / F$, and 
	$K$ is not contained in $F(\zeta_{p^\infty})$. To show that $r_{\pi, 
	\iota}(G_{F(\zeta_{p^\infty})})$ is enormous, it is enough to show that we 
	can find regular semisimple elements in the image of both 
	$G_{K(\zeta_{p^\infty})}$ and 
	$G_{F(\zeta_{p^\infty})} - G_{K(\zeta_{p^\infty})}$. Elements of the latter type exist 
	because of our assumption that $K$ is not contained in 
	$F(\zeta_{p^\infty})$.
	
	Now suppose for a contradiction that $r_{\pi, \iota} \cong \Ind_{G_K}^{G_F} 
	\chi$ is scalar on restriction to $G_{K(\zeta_{p^\infty})}$. This implies 
	that $\chi / \chi^c$ is trivial on $G_{K(\zeta_{p^\infty})}$, and hence 
	that $(\chi / \chi^c)^2 = 1$ (since $c$ acts trivially on 
	$\Gal(K(\zeta_{p^\infty}) / K)$). This contradicts the fact that $\chi / 
	\chi^c$ has infinite order (because its Hodge--Tate weights are all 
	non-zero, because $\pi$ is regular algebraic). This completes the proof.
\end{proof}
We can also prove results for elliptic curves.
\begin{theorem}
	Let $F$ be a totally real number field, and let $E$ be an elliptic curve over $F$. Let $p$ be a prime, and suppose that one of the following holds:
	\begin{enumerate}
		\item $E$ does not have CM.
		\item $E$ has CM by a quadratic field $K / \bQ$, and $K \not\subset F(\zeta_{p^\infty})$.
	\end{enumerate}
	Then $H^1_f(F, \ad V_p(E)) = 0$.
\end{theorem}
\begin{proof}
	If the elliptic curve $E$ has CM, then its $p$-adic Galois representations are automorphic and we can appeal to Theorem \ref{thm_vanishing_for_HMF}. If $E$ does not have CM, then there exists a totally real extension $F' / F$ such that the $p$-adic Galois representations of $E_{F'}$ are automorphic (for example, by \cite{tay-fm2}) and we can appeal again to the same theorem. 
\end{proof}
Combining our results with potential automorphy theorems, we can prove some 
more general vanishing results. Here is an example for symmetric powers of 
two-dimensional representions.
\begin{theorem}\label{thm_vanishing_for_Sym}
	Let $F$ be a CM number field, and let $(\pi,\chi)$ be a regular algebraic, 
	cuspidal, polarized automorphic representation of $\GL_2(\bA_F)$ such that 
	$\Sym^2 \pi$ is cuspidal. Let $p$ be a prime, and fix an isomorphism $\iota 
	: \overline{\bQ}_p \to \bC$. Then for any $n \geq 1$, $H^1_f(F^+, \ad 
	\Sym^{n-1} r_{\pi, \iota}) = 0$.
\end{theorem}
\begin{proof}
	By \cite[Theorem 5.4.1]{BLGGT}, there exists a Galois, CM extension $F' / 
	F$ such that $\Sym^{n-1} r_{\pi, \iota}|_{G_{F'}}$ is automorphic. It 
	suffices to show the vanishing of $H^1_f((F')^+, \ad \Sym^{n-1} r_{\pi, 
	\iota})$, and this follows from Theorem \ref{thm_vanishing_over_CM_field} 
	once we verify that $\Sym^{n-1} r_{\pi, \iota}(G_{F'(\zeta_{p^\infty})})$ 
	is enormous. However, this follows from Example \ref{ex:noncmenormous}.
\end{proof}

Finally, we give an application to vanishing results for  anticyclotomic 
characters, as predicted by the Bloch--Kato conjecture. Over a general 
CM field our main theorem gives vanishing results which 
are not covered by known cases of the anticyclotomic main conjecture (cf. those 
proved in \cite{Hid09}). 
\begin{theorem}
	Let $F$ be a CM number field, and let $\chi : F^\times \backslash \bA_F^\times \to \bC^\times$ be a unitary character of type $A_0$. Let $\iota : \overline{\bQ}_p \to \bC$ be an isomorphism, and suppose that the following conditions are satisfied:
	\begin{enumerate}
		\item $\chi \chi^c = 1$.
		\item The integers $n_\tau$ \textup{(}$\tau \in \Hom(F, \bC)$\textup{)} 
		defined by $\chi_v(z) = \tau(z)^{n_\tau} \tau c(z)^{n_{\tau c}}$ for 
		each place $v|\infty$ are of constant parity, and none of them are 
		zero. 
		\item $F \not\subset F^+(\zeta_{p^\infty})$.
	\end{enumerate}
	Then $H^1_f(F, r_{\chi, \iota}) = 0$.
\end{theorem}
\begin{proof}
	The given conditions imply that there is a character $\psi : F^\times 
	\backslash \bA_F^\times \to \bC^\times$ of type $A_0$ such that $\psi / 
	\psi^c = \chi$. Given this, let $\pi$ denote the automorphic induction of 
	$\psi$ from $F$ to $F^+$. It is a regular algebraic, cuspidal automorphic 
	representation of $\GL_2(\bA_{F^+})$ and $\Ind_{G_F}^{G_{F^+}} r_{\chi, 
	\iota}$ is a subquotient of $\ad r_{\pi, \iota}$, so the desired vanishing 
	follows from Shapiro's lemma for Bloch--Kato Selmer groups and Theorem 
	\ref{thm_vanishing_for_HMF}.
	
	Let us explain why $\psi$ exists. Choose arbitrarily integers $m_\tau$ such that $2 m_\tau + n_\tau = w$, is independent of $\tau$. Then $m_{\tau c} = w - m_\tau$, so there exists a character $\mu : F^\times \backslash \bA_F^\times \to \bC^\times$ of type $A_0$ such that $\mu_v(z) = \tau(z)^{m_\tau} \tau c(z)^{m_{\tau c}}$. Moreover $m_\tau - m_{\tau c} = - n_\tau$, so $\chi_0 = \chi (\mu / \mu^c)$ has finite order and satisfies $\chi_0 \chi_0^c = 1$. Lemma \ref{lem_norm_1_character} implies that there exists another finite order Hecke character $\phi$ such that $\phi / \phi^c = \chi_0$, so we can then take $\psi = \phi \mu^{-1}$.
\end{proof}
\begin{lemma}\label{lem_norm_1_character}
Let $F$ be a CM field, and let $\chi : G_F \to \bQ / \bZ$ be a continuous character such that $\chi \chi^c = 1$. Then there exists a continuous character $\phi : G_F \to \bQ / \bZ$ such that $\phi / \phi^c = \chi$.
\end{lemma}
\begin{proof}
It is equivalent to ask that $H^1(F / F^+, H^1(F, \bQ / \bZ)) = 0$. We use the Hochschild--Serre spectral sequence 
\[ H^p(F / F^+, H^q(F, \bQ / \bZ)) \Rightarrow H^{p + q}(F^+, \bQ / \bZ). \]
 We recall  that if $K$ is a number field, then the product 
 \[ H^r(K, \bQ / \bZ) \to \prod_{v | \infty} H^r(K_v, \bQ / \bZ) \] of 
 restriction maps is an isomorphism when $r \geq 3$ (\cite[Theorem 
 4.20]{Mil06}) and that $H^2(K, \bQ / \bZ) = 0$ (Tate's theorem, see 
 \cite[Theorem 4]{Ser77} or \cite[Theorem 2.1.1]{patrikis-tate}). 
 Since $F$ has no real places, the groups $H^r(F, \bQ / \bZ)$ vanish when $r 
 \geq 2$ and so the spectral sequence in question has only two rows, and can be 
 pieced together into a long exact sequence (cf. \cite[Exercise 5.2.2]{weibel}) 
 including the terms
\[ \xymatrix@1{  H^2(F^+, \bQ / \bZ) \ar[r] & H^1(F / F^+, H^1(F, \bQ / \bZ)) \ar[d] \\ & H^3(F / F^+, H^0(F, \bQ / \bZ)) \ar[r] & H^3(F^+, \bQ / \bZ).} \]
The edge morphism $H^3(F / F^+, H^0(F, \bQ / \bZ)) \to H^3(F^+, \bQ / \bZ)$ is inflation, and is injective because the extension $F / F^+$ is CM and the map $H^3(F^+, \bQ / \bZ) \to \prod_{v | \infty} H^3(F^+_v, \bQ / \bZ)$ is bijective. This completes the proof.
\end{proof}
\bibliographystyle{amsalpha}
\bibliography{CMpatching}

\end{document}